\documentclass[10pt]{article}
\usepackage{amsmath}
\usepackage{amssymb}
\usepackage{amscd}
\usepackage{graphicx}
\usepackage{amsthm}       
\newtheorem{thm}{Theorem}[section]

\newtheorem{lem}[thm]{Lemma}
\newtheorem{cor}[thm]{Corollary}  \theoremstyle{definition}
\newtheorem{df}[thm]{Definition}   \theoremstyle{definition}

\newtheorem{rem}[thm]{Remark}                \theoremstyle{plain}
 \theoremstyle{definition}
\newtheorem{ex}[thm]{Example}   
 
\def\CC{\Bbb{C}}
\def\RR{\Bbb{R}}  \def\QQ{\Bbb{Q}}

\def\ZZ{\Bbb{Z}}
\def\CCI{\hat{\CC}}        \def\NN{\Bbb{N}} 
\def\B1{{\rm\kern.32em\vrule    width.12em       height1.4ex
depth-.05ex\kern-.28em 1}}
\def\G{\Gamma}
\def\g{\gamma }
\def\l{\lambda }

\def\GN{\Gamma ^{\NN }}

\def\Rat{\text{Rat}}
\def\emRat{\text{{\em Rat}}}
\def\Ratp{\text{Rat}_{+}}
\def\emRatp{\text{{\em Rat}}_{+}}
\def\suppt{\text{supp}\, \tau}

\def\Cpt{\text{Cpt}}
\def\emCpt{\text{\em Cpt}}

\def\Min{\text{Min}}
\def\emMin{\text{\em Min}}

\def\LSfc{\text{LS}({\cal U}_{f,\tau }(\hat{\Bbb{C}}))}
\def\emLSfc{\text{{\em LS}}({\cal U}_{f,\tau }(\hat{\Bbb{C}}))}

\begin{document}
\title{Cooperation principle, stability and bifurcation in \\ random complex dynamics
\footnote{Published in Adv. Math. 245 (2013) 137--181. Date: July 12, 2013. 
2010 Mathematics Subject Classification. 
37F10, 37H10. Keywords: Random dynamical systems; Random complex dynamics;  
Rational semigroups;  
Fractal geometry; Cooperation principle; Noise-induced order; Randomness-induced phenomena 
}}

\author{Hiroki Sumi\\  
Department of Mathematics, 
Graduate School of Science, Osaka University\\ 
1-1, Machikaneyama, Toyonaka, Osaka, 560-0043, Japan \\ 
{\bf E-mail: sumi@math.sci.osaka-u.ac.jp}\\ 
http://www.math.sci.osaka-u.ac.jp/\textasciitilde sumi/welcomeou-e.html
\date{}
}
\maketitle
\begin{abstract}
We investigate the 
random dynamics of rational maps and the dynamics 
of semigroups of rational maps on the Riemann sphere $\CCI $. We show that regarding random complex dynamics of polynomials, generically, the chaos of the averaged system disappears at any point in $\CCI $,   
due to the automatic cooperation of the generators. 
We investigate the iteration and spectral properties of transition operators acting on the space of 
(H\"older) continuous functions on $\CCI.$  
We also investigate the stability and bifurcation of random complex dynamics. 
We show that the set of stable systems is open and dense in the space of random dynamical systems of polynomials.
Moreover, we prove that for a stable system, there exist only finitely many minimal sets, each minimal set is attracting, 
and the orbit of a H\"older continuous function on $\CCI $ under the transition operator 
tends exponentially fast to the finite-dimensional space $U$ of finite linear combinations of unitary eigenvectors of the  
transition operator. 
Combining this with the perturbation theory for linear operators, 
we obtain that for a stable system constructed by a finite family of rational maps, 
the projection to the space $U$ depends real-analytically on the probability 
parameters. By taking  a partial derivative of the function of probability of tending to a minimal set with respect 
to a probability parameter, we introduce  a complex analogue of the Takagi function, which is a new concept.   
\end{abstract}
\section{Introduction}
\label{Intro}
In this paper, we investigate the independent and identically-distributed (i.i.d.) random dynamics of rational maps on the Riemann sphere $\CCI $ and the dynamics 
of rational semigroups (i.e., semigroups of non-constant rational maps 
where the semigroup operation is functional composition) on $\CCI .$ 

One motivation for research in complex dynamical systems is to describe some
 mathematical models on ethology. For 
 example, the behavior of the population 
 of a certain species can be described by the 
 dynamical system associated with iteration of a polynomial 
 $f(z)= az(1-z)$ 
 (cf. \cite{D}). However, when there is a change in the natural environment,  
some species have 
 several strategies to survive in nature. 
From this point of view, 
 it is very natural and important not only to consider the dynamics 
of iteration, where the same survival strategy (i.e., function) is repeatedly applied, but also 
to consider random 
 dynamics, where a new strategy might be applied at each time step.  
 Another motivation for research in complex dynamics is Newton's method to 
 find a root of a complex polynomial,   
which often is expressed as the dynamics of a rational map $g$ on $\CCI $ with 
$\deg (g)\geq 2$, where $\deg (g)$ denotes the degree of $g.$  
 We sometimes use  computers to analyze such dynamics, and since 
 we have some errors at each step of the calculation in the computers, it is quite natural to investigate the 
 random dynamics of rational maps. 
In various fields, we have many mathematical models which are described by 
the dynamical systems associated with polynomial or rational maps. For each 
model, it is natural and important to consider a randomized model, since we always have 
some kind of noise or random terms in nature.  
The first study of random complex dynamics was given by J. E. Fornaess and  N. Sibony (\cite{FS}). 
They mainly investigated random dynamics generated by small perturbations of a single rational map.  
For research on random complex dynamics of quadratic polynomials, 
see \cite{Br1, Br2, BBR, Bu1, Bu2, GQL}.  
For research on random dynamics of polynomials (of general degrees), 
see the author's works \cite{SdpbpI, SdpbpIII, Ssugexp, Splms10, Srcddc}. 
  
  In order to investigate random complex dynamics, it is very natural to study the dynamics of 
  associated rational semigroups. 
In fact, it is a very powerful tool to investigate random complex dynamics, 
since random complex dynamics and the dynamics of rational semigroups are related to each other very deeply.   
The first study of dynamics of rational semigroups was 
conducted by
A. Hinkkanen and G. J. Martin (\cite{HM}),
who were interested in the role of the
dynamics of polynomial semigroups (i.e., semigroups of non-constant polynomial maps) while studying
various one-complex-dimensional
moduli spaces for discrete groups,
and
by F. Ren's group (\cite{GR}), 
 who studied 
such semigroups from the perspective of random dynamical systems.
Since the Julia set $J(G)$ of a finitely generated rational semigroup 
$G=\langle h_{1},\ldots, h_{m}\rangle $ has 
``backward self-similarity,'' i.e.,  
$J(G)=\bigcup _{j=1}^{m}h_{j}^{-1}(J(G))$ (see \cite[Lemma 0.2]{S4}),  
the study of the dynamics of rational semigroups can be regarded as the study of  
``backward iterated function systems,'' and also as a generalization of the study of 
self-similar sets in fractal geometry.  
For recent work on the dynamics of rational semigroups, 
see the author's papers \cite{S3}--\cite{Srcddc}, and 
\cite{St3, SS, SU1, SU2}. 

In this paper, by combining several results from \cite{Splms10} and many new ideas, 
we investigate the random complex dynamics 
and the dynamics of rational semigroups. 
 In the usual iteration dynamics of a single rational map $g$ with $\deg (g)\geq 2$, 
 we always have a non-empty chaotic part, i.e., in the Julia set $J(g)$ of $g$, which is a perfect set,  
 we have sensitive initial values and dense orbits. Moreover, 
 for any ball $B$ with $B\cap J(g)\neq \emptyset $, $g^{n}(B)$ expands as $n\rightarrow \infty.$ 
Regarding random complex dynamics, it is natural to ask the following question. 
Do we have a kind of ``chaos'' in the averaged system? Or do we have no chaos? 
How do many kinds of maps in the system interact? 
What can we say about stability and bifurcation? 
Since the chaotic phenomena hold even for a single rational map, 
one may expect that in random dynamics of rational maps, 
most systems would exhibit a great amount of chaos. 
However, it turns out that this is not true.      
One of the main purposes of this paper is to prove that for a generic system of random complex 
dynamics of polynomials, many kinds of maps in the system ``automatically'' cooperate so that  
they make the chaos of the averaged system disappear at any point in the phase space, even though the dynamics of each map in the system 
 have a chaotic part (Theorems~\ref{t:pmsodint}, \ref{t:pmsod}). We call this phenomenon the ``{\bf cooperation principle}''. 
Moreover, we prove that for a generic system, we have a kind of stability 
(see Theorems~\ref{t:msmtaustint}, \ref{t:msmtaust}). 
 We remark that the chaos disappears in the $C^{0}$ ``sense'', but 
 under certain conditions, the chaos remains in the $C^{\beta }$ ``sense'', where 
 $C^{\beta }$ denotes the space of $\beta $-H\"older continuous functions with 
 exponent $\beta \in (0,1)$ (see Remark~\ref{r:holb}). 

To introduce the main idea of this paper,  
we let $G$ be a rational semigroup and denote by $F(G)$ the Fatou set of $G$, which is defined to be  
the maximal open subset of $\CCI $ where $G$ is equicontinuous with respect to the spherical distance on $\CCI $.    
We call $J(G):=\CCI \setminus F(G)$ the Julia set of $G.$  
The Julia set is backward invariant under each element $h\in G$, but 
might not be forward invariant. This is a difficulty of the theory of rational semigroups. 
Nevertheless, we utilize this as follows.  
The key to investigating random complex dynamics is to consider the 
following {\bf kernel Julia set} of $G$, which is defined by 
$J_{\ker }(G)=\bigcap _{g\in G}g^{-1}(J(G)).$ This is the largest forward 
invariant subset of $J(G)$ under the action of $G.$ Note that 
if $G$ is a group or if $G$ is a commutative semigroup, 
then $J_{\ker }(G)=J(G).$ 
However, for a general rational semigroup $G$ generated by a family of 
rational maps $h$ with $\deg (h)\geq 2$, it may happen that 
$\emptyset =J_{\ker }(G)\neq J(G) $.   

Let Rat be the space of all non-constant rational maps on the Riemann sphere $\CCI $, 
endowed with the distance $\kappa $ which is defined by 
$\kappa (f,g):=\sup _{z\in \CCI }d(f(z),g(z))$, where $d$ denotes the spherical distance on $\CCI .$  
Let Rat$_{+}$ be the space of all rational maps $g$ with $\deg (g)\geq 2.$ Let 
${\cal P}$ be the space of all polynomial maps $g$ with $\deg (g)\geq 2.$   
Let $\tau $ be a Borel probability measure on Rat with compact support. 
We consider the i.i.d. random dynamics on $\CCI $ such that 
at every step we choose a map $h\in \mbox{Rat}$ according to $\tau .$ 
Thus this determines a time-discrete Markov process with time-homogeneous transition probabilities 
on the phase space 
$\CCI $ such that for each $x\in \CCI $ and 
each Borel measurable subset $A$ of $\CCI $, 
the transition probability 
$p(x,A)$ from $x$ to $A$ is defined as $p(x,A)=\tau (\{ g\in \Rat \mid g(x)\in A\} ).$ 
Let $G_{\tau }$ be the 
rational semigroup generated by the support of $\tau .$ 
Let $C(\CCI )$ be the space of all complex-valued continuous functions on $\CCI $ endowed with 
the supremum norm $\| \cdot \| _{\infty }.$  
Let $M_{\tau }$ be the operator on $C(\CCI )$ 
defined by $M_{\tau }(\varphi )(z)=\int \varphi (g(z)) d\tau (g).$ 
This $M_{\tau }$ is called the transition operator of the Markov process induced by $\tau .$ 
For a metric space $X$, let ${\frak M}_{1}(X)$ be the space of all 
Borel probability measures on $X$ endowed with the topology 
induced by weak convergence (thus $\mu _{n}\rightarrow \mu $ in ${\frak M}_{1}(X)$ if and only if 
$\int \varphi d\mu _{n}\rightarrow \int \varphi d\mu $ for each bounded continuous function $\varphi :X\rightarrow \RR $). 
Note that if $X$ is a compact metric space, then ${\frak M}_{1}(X)$ is compact and metrizable. 
For each $\tau \in {\frak M}_{1}(X)$, we denote by supp$\, \tau $ the topological support of $\tau .$  
Let ${\frak M}_{1,c}(X)$ be the space of all Borel probability measures $\tau $ on $X$ such that supp$\,\tau $ is 
compact.     
Let $M_{\tau }^{\ast }:{\frak M}_{1}(\CCI )\rightarrow {\frak M}_{1}(\CCI )$ 
be the dual of $M_{\tau }$.  
This $M_{\tau }^{\ast }$ can be regarded as the ``averaged map'' 
on the extension ${\frak M}_{1}(\CCI )$ of $\CCI $ (see Remark~\ref{r:Phi}).  
We define the ``Julia set'' $J_{meas}(\tau )$ of 
the dynamics of $M_{\tau }^{\ast }$ as the set of all elements $\mu \in {\frak M}_{1}(\CCI )$ 
satisfying that for each neighborhood $B$ of $\mu $, $\{ (M_{\tau }^{\ast })^{n}|_{B}:B\rightarrow {\frak M}_{1}(\CCI )\} _{n\in \NN }$ 
is not equicontinuous on $B$ (see Definition~\ref{d:ytau}). 
For each sequence $\gamma =(\gamma _{1}, \gamma _{2},\ldots )\in (\Rat )^{\NN }$, 
we denote by $J_{\gamma }$ the set of non-equicontinuity of the sequence 
$\{ \gamma _{n}\circ \cdots \circ \gamma _{1}\} _{n\in \NN }$ with respect to the spherical distance on $\CCI .$  This $J_{\gamma }$ is called the Julia set of 
$\gamma .$  Let $\tilde{\tau }:=\otimes _{j=1}^{\infty }\tau \in {\frak M}_{1}((\Rat)^{\NN }).$ 
For a $\tau \in {\frak M}_{1,c}(\Rat)$, we denote by $U_{\tau }$ 
the space of all finite linear combinations of unitary eigenvectors of $M_{\tau }:C(\CCI )\rightarrow 
C(\CCI )$, where an eigenvector  is said to be unitary if the absolute value of 
the corresponding eigenvalue is equal to one. Moreover, 
we set ${\cal B}_{0,\tau }:= \{ \varphi \in C(\CCI )\mid M_{\tau }^{n}(\varphi )\rightarrow 0 \mbox{ as }n\rightarrow \infty \} .$ 
For a metric space $X$, we denote by $\Cpt(X)$ the space of all 
non-empty compact subsets of $X$ endowed with the Hausdorff metric.  
 For a rational semigroup $G$, we say that a non-empty compact subset $L$ of $\CCI $ is a minimal set for $(G,\CCI )$ 
if $L$ is minimal in 
$\{ C\in \Cpt(\CCI ) \mid  \forall g\in G, g(C)\subset C\} $ 
with respect to inclusion.  
Moreover, we set $\Min (G,\CCI ):= \{ L \in \Cpt(\CCI )\mid L \mbox{ is a minimal set for } (G,\CCI )\} .$ 
For a $\tau \in {\frak M}_{1}(\Rat)$, let $S_{\tau }:=\bigcup _{L\in \Min(G_{\tau },\CCI)}L.$  
For a $\tau \in {\frak M}_{1}(\Rat)$, let $\G _{\tau }:=\mbox{supp}\, \tau (\subset \Rat).$ 
 In \cite{Splms10}, the following two theorems were obtained. 
\begin{thm}[Cooperation Principle I, see Theorem 3.14 in \cite{Splms10}]
\label{t:thmA}
Let $\tau \in {\frak M}_{1,c}({\emRat}).$ Suppose   
 that $J_{\ker }(G_{\tau })=\emptyset .$ Then  
 $J_{meas}(\tau )=\emptyset .$ Moreover, for $\tilde{\tau }$-a.e. $\gamma \in (\emRat) ^{\NN }$, 
 the $2$-dimensional Lebesgue measure of $J_{\gamma }$ is equal to zero. 
\end{thm}   
\begin{thm}[Cooperation Principle II: Disappearance of Chaos, see Theorem 3.15 in \cite{Splms10}]
\label{t:thmB}\ \\ 
Let $\tau \in {\frak M}_{1,c}(\emRat).$ Suppose that  
$J_{\ker }(G_{\tau })=\emptyset $ and $J(G_{\tau })\neq \emptyset $.  
Then the following {\em (1)(2)(3)} hold. 
\begin{itemize}
\item[{\em (1)}]
There exists a direct sum decomposition 
$C(\CCI )= U_{\tau }\oplus {\cal B}_{0,\tau }$.  
Moreover, $1\leq \dim _{\CC } U_{\tau }<\infty $ and ${\cal B}_{0,\tau }$ is a closed subspace 
of $C(\CCI ).$   
Furthermore, each element of $U_{\tau }$ is locally constant on 
$F(G_{\tau })$. 
Therefore each element of $U_{\tau }$ is a continuous function 
on $\CCI $ which varies only on the Julia set $J(G_{\tau }).$  
\item[{\em (2)}]
For each $z\in \CCI $, there exists a Borel subset ${\cal A}_{z}$ 
of $(\emRat)^{\NN }$ with $\tilde{\tau }({\cal A}_{z})=1$ with the following properties 
{\em (a)} and {\em (b)}.
{\em (a)} For each $\gamma =(\gamma _{1},\gamma _{2},\ldots )\in {\cal A}_{z}$, 
there exists a number $\delta =\delta (z,\gamma )>0$ such that 
$\mbox{diam}(\gamma _{n}\circ \cdots \circ \gamma _{1}(B(z,\delta )))\rightarrow 0$ as 
$n\rightarrow \infty $, where diam denotes the diameter with respect to the 
spherical distance on $\CCI $, and $B(z,\delta )$ denotes the ball with center $z$ and radius 
$\delta .$ {\em (b)} For each $\gamma =(\gamma _{1},\gamma _{2},\ldots )\in {\cal A}_{z}$, 
$d(\gamma _{n}\circ \cdots \circ \gamma _{1}(z),S_{\tau })\rightarrow 0$ as $n\rightarrow \infty .$  
\item[{\em (3)}] 
We have $1\leq \sharp \emMin(G_{\tau },\CCI )<\infty .$  
\end{itemize}
\end{thm}
\begin{rem}
If $\tau \in {\frak M}_{1}(\Rat)$ and $\G _{\tau } \cap \Ratp \neq \emptyset $, then 
$\sharp J(G_{\tau })=\infty .$  
\end{rem}
Theorems~\ref{t:thmA} and \ref{t:thmB} mean that if all the maps in the support of $\tau $ cooperate, the chaos of the averaged system 
disappears, even though the dynamics of each map of the system has a chaotic part.  
Moreover, Theorems~\ref{t:thmA} and \ref{t:thmB} describe new phenomena which can hold in random complex dynamics but 
cannot hold in the usual iteration dynamics of a single $h\in \mbox{Rat}_{+}.$ For example, 
for any $h\in \mbox{Rat}_{+}$, if we take a point $z\in J(h)$, where $J(h)$ denotes the Julia set of the semigroup generated by 
$h$, then the Dirac measure $\delta _{z}$ at $z$ belongs to $J_{meas}(\delta _{h})$, and 
for any ball $B$ with 
$B\cap J(h)\neq \emptyset $, $h^{n}(B)$ expands as $n\rightarrow \infty $. Moreover, for any 
$h\in \Ratp$,  
we have infinitely many minimal sets (periodic cycles) of $h.$ 

Considering these results, we have the following natural question: ``When is the kernel Julia set empty?'' 
In order to give several answers to this question, 
we say that a family $\{ g_{\lambda }\} _{\lambda \in \Lambda }$ of rational (resp. polynomial) maps 
 is a holomorphic family of 
rational (resp. polynomial) maps if $\Lambda $ is a finite dimensional complex manifold and 
the map $(z,\lambda )\mapsto g_{\lambda }(z)\in \CCI $ is holomorphic on $\CCI \times \Lambda .$  
In \cite{Splms10}, the following result was proved. 
(Remark. In \cite[Lemma 5.34, Definition 3.54-1]{Splms10}, 
$\Lambda $ should be connected.) 
\begin{thm}[Cooperation Principle III, see Theorem 1.7 and Lemma 5.34 from \cite{Splms10}]  
\label{t:cpIIIint}
Let $\tau \in {\frak M}_{1,c}({\cal P})$. 
Suppose that 
for each $z\in \CC $, there exists a holomorphic family $\{ g_{\lambda }\} _{\lambda \in \Lambda }$ 
of polynomial maps with $\bigcup _{\lambda \in \Lambda }\{ g_{\lambda }\} \subset \G _{\tau }$ 
such that $\Lambda $ is connected and $\lambda \mapsto g_{\lambda }(z)$ is non-constant on $\Lambda $. Then 
$J_{\ker }(G_{\tau })=\emptyset $, $J(G_{\tau })\neq \emptyset $ and all statements in 
Theorems~\ref{t:thmA} and  \ref{t:thmB} hold. 
\end{thm}
In this paper, regarding the previous question, we prove the following very strong results.  
To state the results,  
we say that a $\tau \in {\frak M}_{1,c}(\Rat )$ is {\bf mean stable} if 
there exist non-empty open subsets $U,V$ of $F(G_{\tau })$ and a number $n\in \NN $ 
such that all of the following (I)(II)(III) hold: 
(I) $\overline{V}\subset U$ and $\overline{U}\subset F(G_{\tau }).$ 
(II) For each $\gamma =(\g _{1},\g _{2},\ldots )\in \Gamma _{\tau }^{\NN }$, 
$\gamma _{n}\circ \cdots \circ \g _{1}(\overline{U})\subset V.$ 
(III) For each point $z\in \CCI $, there exists an element 
$g\in G_{\tau }$ such that  
$g(z)\in U.$ 
We remark that if $\tau \in {\frak M}_{1,c}(\Rat )$ is mean stable, then $J_{\ker }(G_{\tau })=\emptyset .$ 
Thus if $\tau \in {\frak M}_{1,c}(\Rat )$ is mean stable and $J(G_{\tau })\neq \emptyset $, then 
$J_{\ker }(G_{\tau })=\emptyset $ and all statements in Theorems~\ref{t:thmA} and \ref{t:thmB} hold. 
Note also that by using \cite[Theorem 2.11]{Mi}, it is not so difficult to see that $\tau $ is mean stable if and only if 
$\sharp (\Min (G_{\tau },\CCI ))<\infty $  
and each $L\in \Min (G_{\tau },\CCI )$ is ``attracting'', i.e., 
there exists an open subset $W_{L}$ of $F(G_{\tau }) $ with $L\subset W_{L}$ and an $\epsilon >0$ such that for each $z\in W_{L}$ and for each 
$\g =(\g _{1},\g _{2},\ldots )\in \G _{\tau }^{\NN }$,  
$d(\g _{n}\circ \cdots \circ \g _{1}(z), L)\rightarrow 0$ 
and $\mbox{diam}(\g _{n}\circ \cdots \circ \g _{1}(B(z,\epsilon )))\rightarrow 0$ as $n\rightarrow \infty $ (see Remark~\ref{r:msallatt}).   
Therefore, if $\tau \in {\frak M}_{1,c}(\Ratp )$ is mean stable, 
then (1) $J_{meas }(\tau )=\emptyset $, 
(2) for $\tilde{\tau }$-a.e. $\g \in (\Ratp )^{\NN }$, 
the $2$-dimensional Lebesgue measure of $J_{\g }$ is zero, 
(3) for each $z\in \CCI $ there exists a Borel subset ${\cal C}_{z}$ of $(\Ratp )^{\NN }$ with 
$\tilde{\tau }({\cal C}_{z})=1$ such that 
 for each $\g \in {\cal C}_{z}$, $d(\gamma _{n}\circ \cdots \circ \gamma _{1}(z), S_{\tau })\rightarrow 0$ as 
 $n\rightarrow \infty $, 
(4) for $\tilde{\tau }$-a.e. $\gamma =(\g _{1}, \g _{2},\ldots )\in (\Ratp)^{\NN }$, for each $z\in \CCI \setminus J_{\gamma }$, 
we have $d(\gamma _{n}\circ \cdots \circ \g _{1}(z),S_{\tau })\rightarrow 0$ as $n\rightarrow \infty , $ 
(5) $S_{\tau }$ is a finite union of ``attracting minimal sets'', 
(6)(negativity of Lyapunov exponents) there exists a constant $c<0$ such that 
for each $z\in \CCI $ there exists a Borel subset ${\cal D}_{z}$ of $(\Ratp)^{\NN }$ with 
$\tilde{\tau }({\cal D}_{z})=1$ satisfying that for each $\gamma =(\gamma _{1},\gamma _{2},\ldots )\in {\cal D}_{z}$, 
we have $\limsup _{n\rightarrow \infty }\frac{1}{n}\log \| D(\gamma _{n}\circ \cdots \circ \gamma _{1})_{z}\| _{s}\leq c$, 
where $\| D(\gamma _{n}\circ \cdots \circ \gamma _{1})_{z}\| _{s}$ denotes the 
norm of the derivative of $\gamma _{n}\circ \cdots \circ \gamma _{1}$ at $z$ with respect to the spherical metric, 
 and (7) for the system generated by $\tau $, there exists a stability (Theorem~\ref{t:msmtaustint}).
Thus, in terms of averaged systems, 
the notion ``mean stability'' of random complex dynamics can be regarded as an analogy of ``hyperbolicity'' of the usual iteration dynamics of a single rational map.  
For a metric space $(X,d)$,  
let ${\cal O}$ be the topology of ${\frak M}_{1,c}(X)$ 
such that 
$\mu _{n}\rightarrow \mu $ in $({\frak M}_{1,c}(X),{\cal O})$ 
as $n\rightarrow \infty $ if and only if 
(i) $\int \varphi d\mu _{n}\rightarrow \int \varphi d\mu $ for each 
bounded continuous function $\varphi :X \rightarrow \CC $, 
and (ii) $\G _{\mu _{n}}\rightarrow \G_{\mu }$ 
with respect to the Hausdorff metric in the space $\Cpt(X).$   
We say that a subset 
${\cal Y}$ of Rat satisfies condition $(\ast )$ if ${\cal Y}$ is closed in Rat and at least one of the following (1) and (2) holds: 
(1) for each $(z_{0},h_{0})\in \CCI \times {\cal Y}$, 
there exists a holomorphic family $\{ g_{\lambda }\} _{\lambda \in \Lambda }$ of 
rational maps with $\bigcup _{\lambda \in \Lambda }\{ g_{\lambda }\} \subset {\cal Y}$
and an element $\lambda _{0}\in \Lambda $, such that,  
$g_{\lambda _{0}}=h_{0}$ and $\lambda \mapsto g_{\lambda }(z_{0})$ is non-constant in any neighborhood of 
$\lambda _{0}.$ (2) ${\cal Y}\subset {\cal P}$ and for each $(z_{0},h_{0})\in \CC \times {\cal Y}$, 
there exists a holomorphic family $\{ g_{\lambda }\} _{\lambda \in \Lambda }$ of 
polynomial maps with $\bigcup _{\lambda \in \Lambda }\{ g_{\lambda }\} \subset {\cal Y}$  
and an element $\lambda _{0}\in \Lambda $ such that 
$g_{\lambda _{0}}=h_{0}$ and $\lambda \mapsto g_{\lambda }(z_{0})$ is non-constant in any neighborhood of 
$\lambda _{0}.$ For example, Rat, $\Ratp$, ${\cal P}$, and $\{ z^{d}+c\mid c\in \CC \} \ (d\in \NN , d\geq 2)$ 
satisfy condition $(\ast ).$  
Under these notations, we prove the following theorem. 
\begin{thm}[Cooperation Principle IV, Density of Mean Stable Systems, see Theorem~\ref{t:pmsod}]
\label{t:pmsodint} 
Let ${\cal Y}$ be a subset of ${\cal P}$ satisfying condition $(\ast ).$ 
Then, we have the following. 
\begin{itemize}
\item[{\em (1)}]
The set $\{ \tau \in {\frak M}_{1,c}({\cal Y})\mid \tau \mbox{ is mean stable}\} $ 
is open and dense in $({\frak M}_{1,c}({\cal Y}),{\cal O}).$  
Moreover, 
the set 
$\{ \tau \in {\frak M}_{1,c}({\cal Y})\mid J_{\ker }(G_{\tau })=\emptyset , J(G_{\tau })\neq \emptyset \} $ 
contains 
$\{ \tau \in {\frak M}_{1,c}({\cal Y})\mid \tau \mbox{ is mean stable}\} $.  
\item[{\em (2)}] 
The set 
$\{ \tau \in {\frak M}_{1,c}({\cal Y})\mid \tau \mbox{ is mean stable},\ \sharp \G _{\tau }<\infty \} $ 
is dense in 
$({\frak M}_{1,c}({\cal Y}), {\cal O}).$ 
\end{itemize}
\end{thm}
We remark that in the study of iteration of a single rational map, we have a very famous conjecture 
(HD conjecture, see \cite[Conjecture 1.1]{Mc}) 
which states that hyperbolic rational maps are dense in the space of rational maps. 
Theorem~\ref{t:pmsodint} solves this kind of problem (in terms of averaged systems) in the study of random dynamics of complex polynomials. 
We also prove the following result. 
\begin{thm}[see Corollary~\ref{c:msminfull}]
\label{t:msminfullint}
Let ${\cal Y}$ be a subset of $\emRatp$ satisfying condition $(\ast ).$ Then, 
the set 
$$\{ \tau \in {\frak M}_{1,c}({\cal Y})\mid \tau \mbox{ is mean stable}\} 
\cup \{ \rho \in {\frak M}_{1,c}({\cal Y})\mid \emMin(G_{\rho },\CCI )=\{ \CCI \} ,J(G_{\rho })=\CCI \} $$  
is dense in $({\frak M}_{1,c}({\cal Y}),{\cal O}).$ 
\end{thm}
For the proofs of Theorems~\ref{t:pmsodint} and \ref{t:msminfullint}, 
we need to investigate and classify the minimal sets for $(\langle \G \rangle ,\CCI )$, 
where $\G \in \Cpt(\Rat )$, and $\langle \G \rangle $ denotes the 
rational semigroup generated by $\G $ (thus $\langle \G \rangle =\{ g_{i_{1}}\circ \cdots \circ g_{i_{n}}\mid 
n\in \NN , \forall g_{i_{j}}\in \G \} $) (Lemmas~\ref{l:Lclassify},\ref{l:bebg}).  
In particular, it is important to analyze the reason of instability for a non-attracting minimal set.  

For each $\tau \in {\frak M}_{1,c}(\Rat)$ and  
for each $L\in \Min(G_{\tau },\CCI )$, 
let $T_{L,\tau }$ be the function of probability of tending to $L.$ 
Namely, for each $z\in \CCI $, 
we set $T_{L,\tau }(z):=\tilde{\tau }(\{ \gamma =(\gamma _{1},\gamma _{2},\ldots )\in (\Rat)^{\NN } 
\mid d(\gamma _{n}\circ \cdots \circ \gamma _{1}(z), L)\rightarrow 0\ (n\rightarrow \infty )\} )$.  
We set $C(\CCI )^{\ast }:=\{ \rho :C(\CCI )\rightarrow \CC \mid \rho \mbox{ is linear and continuous}\} $ endowed 
with the weak$^{\ast }$-topology.  
We prove the following stability result.  
\begin{thm}[Cooperation Principle V, ${\cal O}$-Stability for Mean Stable Systems, see Theorem~\ref{t:msmtaust}] 
\label{t:msmtaustint}
Let $\tau \in {\frak M}_{1,c}(\emRat)$ be mean stable. 
Suppose $J(G_{\tau })\neq \emptyset .$ 
Then there exists a neighborhood $\Omega $ of $\tau $ in 
$({\frak M}_{1,c}(\emRat),{\cal O})$ such that 
all of the following hold. 
\begin{itemize}
\item[{\em (1)}]
For each $\nu \in \Omega $, $\nu $ is mean stable, $\sharp (J(G_{\nu }))\geq 3$, 
and $\sharp (\emMin(G_{\nu },\CCI ))=\sharp (\emMin(G_{\tau },\CCI )).$
\item[{\em (2)}]
For each $\nu \in \Omega $, $\dim _{\CC }(U_{\nu })=\dim _{\CC }(U_{\tau })\geq 1.$
\item[{\em (3)}]
The map $\nu \mapsto \pi _{\nu }$ and $\nu \mapsto U_{\nu }$ are continuous on $\Omega $, 
where $\pi _{\nu }:C(\CCI )\rightarrow U_{\nu }$ denotes the canonical projection (see Theorem~\ref{t:thmB}).   
More precisely, for each $\nu \in \Omega $, there exists a family $\{ \varphi _{j,\nu }\} _{j=1}^{q}$ 
of unitary eigenvectors of $M_{\nu }:C(\CCI )\rightarrow C(\CCI )$, 
where $q= \dim _{\CC }(U_{\tau })$, 
and a finite family $\{ \rho _{j,\nu }\} _{j=1}^{q}$ in $C(\CCI )^{\ast }$ such that 
all of the following hold.
\begin{itemize}
\item[{\em (a)}] 
$\{ \varphi _{j,\nu }\} _{j=1}^{q}$ is a basis of $U_{\nu }.$ 
\item[{\em (b)}] 
For each $j$, $\nu \mapsto \varphi _{j,\nu }\in C(\CCI )$ is continuous on $\Omega .$ 
\item[{\em (c)}]
For each $j$, $\nu \mapsto \rho _{j,\nu }\in C(\CCI )^{\ast }$ is continuous on $\Omega .$ 
\item[{\em (d)}] 
For each $(i,j)$ and each $\nu \in \Omega $, $\rho _{i,\nu }(\varphi _{j,\nu })=\delta _{ij}.$ 
\item[{\em (e)}] 
For each $\nu \in \Omega $ and each $\varphi \in C(\CCI )$, 
$\pi _{\nu }(\varphi )=\sum _{j=1}^{q}\rho _{j,\nu }(\varphi )\cdot \varphi _{j,\nu }.$  
\end{itemize}
\item[{\em (4)}] 
For each $L\in \emMin(G_{\tau },\CCI )$, 
there exists a continuous map $\nu \mapsto L_{\nu }\in \emMin(G_{\nu },\CCI )
\subset \emCpt(\CCI )$ on $\Omega $ 
with respect to the Hausdorff metric   
such that $L_{\tau }=L.$ 
Moreover, for each $\nu \in \Omega $, 
$\{ L_{\nu }\} _{L\in \emMin(G_{\tau },\CCI )}=\emMin (G_{\nu },\CCI ).$ 
Moreover, for each $\nu \in \Omega $ and for each 
$L,L'\in \emMin (G_{\tau },\CCI )$ with $L\neq L'$, 
we have $L_{\nu }\cap L'_{\nu }=\emptyset .$ 
Furthermore, for each $L\in \emMin (G_{\tau },\CCI )$, 
the map $\nu \mapsto T_{L_{\nu },\nu }\in (C(\CCI ),\| \cdot \| _{\infty })$ 
is continuous on $\Omega .$   
\end{itemize}
\end{thm}
By applying these results, we give a characterization of mean stability (Theorem~\ref{t:ABCD}). 

We remark that if $\tau \in {\frak M}_{1,c}(\Ratp)$ is mean stable and 
$\sharp (\Min (G_{\tau },\CCI ))>1$, 
then the averaged system of $\tau $ is stable (Theorem~\ref{t:msmtaustint}) and 
the system also has a kind of variety. 
Thus such a $\tau $ can describe a stable system which does not lose variety. 
This fact (with Theorems~\ref{t:pmsodint}, \ref{t:thmA}, \ref{t:thmB}) might be useful 
when we consider mathematical modeling in various fields. 

Let ${\cal Y}$ be a subset of $\Rat $ satisfying $(\ast )$. 
Let $\{ \mu _{t}\} _{t\in [0,1]}$ be a continuous family in $({\frak M}_{1,c}({\cal Y}),{\cal O}).$ 
We consider the bifurcation of $\{ M_{\mu _{t }}\} _{t\in [0,1]}$ and 
$\{ G_{\mu _{t}}\} _{t\in [0,1]}.$ We prove the following result.
\begin{thm}[Bifurcation: see Theorem~\ref{t:bifur} and Lemmas~\ref{l:Lclassify}, \ref{l:bebg}] 
\label{t:bifurint}
Let ${\cal Y}$ be a subset of $\emRatp $ satisfying condition $(\ast ).$  
For each $t\in [0,1]$, let 
$\mu _{t}$ be an element of ${\frak M}_{1,c}({\cal Y}).$ 
Suppose that all of the following conditions {\em (1)--(4)} hold.
\begin{itemize}
\item[{\em (1)}] 
$t\mapsto \mu _{t}\in ({\frak M}_{1,c}({\cal Y}),{\cal O})$ is continuous on 
$[0,1].$ 
\item[{\em (2)}] If $t_{1},t_{2}\in [0,1]$ and $t_{1}<t_{2}$, then 
$\G _{\mu _{t_{1}}}\subset \mbox{{\em int}}(\G _{\mu _{t_{2}}})$ with respect to 
the topology of ${\cal Y}.$ 
\item[{\em (3)}] 
$\mbox{{\em int}}(\G _{\mu _{0}})\neq \emptyset $ and $F(G _{\mu _{1}} )\neq \emptyset $. 
\item[{\em (4)}] 
$\sharp (\emMin (G _{\mu _{0}}, \CCI ))\neq 
\sharp (\emMin (G _{\mu _{1}}, \CCI )).$ 
\end{itemize}
Let $B:=\{ t\in [0,1)\mid \mu _{t} \mbox{ is not mean stable} \}.$ 
Then, we have the following.
\begin{itemize}
\item[{\em (a)}]
For each $t\in [0,1]$, $J_{\ker }(G _{\mu _{t}} )=\emptyset $ and 
$\sharp J(G _{\mu _{t}} )\geq 3$, and 
all statements in {\em \cite[Theorem 3.15]{Splms10}} and Theorems~\ref{t:thmA},\ref{t:thmB} (with $\tau =\mu _{t}$) hold. 
\item[{\em (b)}] We have    
$1\leq \sharp B \leq \sharp (\emMin(G _{\mu _{0}}, \CCI ))-\sharp (\emMin(G _{\mu _{1}}, \CCI )) <\infty .$
Moreover, for each $t\in B$, 
either {\em (i)}  
there exists an element $L\in \emMin (G_{\mu _{t}},\CCI )$, 
a point $z\in L$, and an element $g\in \partial \G_{\mu _{t}}(\subset {\cal Y})$ such that 
 $z\in L\cap J(G_{\mu _{t}})$ and $g(z)\in L\cap J(G_{\mu _{t}})$, or   
{\em (ii)} there exist an element $L\in \emMin (G_{\mu _{t}},\CCI )$, 
a point $z\in L$, and finitely many elements $g_{1},\ldots ,g_{r}\in \partial \G _{\mu _{t}}$ 
such that  
$L\subset F(G_{\mu _{t}})$ and $z$ belongs to a Siegel disk or a Hermann ring of 
$g_{r}\circ \cdots \circ g_{1}.$  
\item[{\em (c)}] 
For each $s\in (0,1]$ there exists a number $t_{s}\in (0,s)$ such that 
for each $t\in [t_{s},s]$, \\ 
$\sharp (\emMin (G_{\mu _{t}},\CCI ))$ $ =\sharp (\emMin (G_{\mu _{s}},\CCI )).$ 
\end{itemize} 
\end{thm}
In Example~\ref{e:bifur}, an example to which we can apply the above theorem is given. 

We also investigate the spectral properties of $M_{\tau }$ acting on H\"older continuous functions on $\CCI $ and 
stability (see subsection~\ref{Spectral}). 
 For each $\alpha \in (0,1)$,  
let \\ $C^{\alpha }(\CCI ):= \{ \varphi \in C(\CCI )\mid \sup _{x,y\in \CCI , x\neq y} 
|\varphi (x)-\varphi (y)|/d(x,y)^{\alpha }<\infty \} $ 
 be the Banach space of all complex-valued $\alpha $-H\"{o}lder continuous functions on $\CCI $
endowed with the $\alpha $-H\"{o}lder norm $\| \cdot \| _{\alpha },$ 
where $\| \varphi \| _{\alpha }:= \sup _{z\in \CCI }| \varphi (z)| 
+\sup _{x,y\in \CCI ,x\neq y}|\varphi (x)-\varphi (y)|/d(x,y)^{\alpha }$ 
for each $\varphi \in C^{\alpha }(\CCI ).$ 

Regarding the space $U_{\tau }$, we prove the following. 
\begin{thm}
\label{t:utaucaint}
Let $\tau \in {\frak M}_{1,c}(\emRat)$. Suppose that 
$J_{\ker }(G_{\tau })=\emptyset $ and $J(G_{\tau })\neq \emptyset .$ 
Then, there exists an $\alpha \in (0,1)$ such that  
$U_{\tau }\subset C^{\alpha }(\CCI ).$ 
Moreover, for each $L\in \emMin(G_{\tau },\CCI )$, the function 
$T_{L,\tau }:\CCI \rightarrow [0,1]$ of probability of tending to $L$ belongs to $C^{\alpha }(\CCI ).$   
\end{thm}
Thus each element of $U_{\tau }$ has a kind of regularity. 
For the proof of Theorem~\ref{t:utaucaint}, 
the result ``each element of $U_{\tau }$ is locally constant on $F(G_{\tau })$''  
(Theorem~\ref{t:thmB} (1)) is used.  

If $\tau \in {\frak M}_{1,c}(\Rat )$ is mean stable and $J(G_{\tau })\neq \emptyset $, 
then by \cite[Proposition 3.65]{Splms10}, we have 
$S_{\tau }\subset F(G_{\tau }).$ 
From this point of view, we consider the situation that 
$\tau \in {\frak M}_{1,c}(\Rat )$ satisfies $J_{\ker }(G_{\tau })=\emptyset $, 
$J(G_{\tau })\neq \emptyset $, and 
$S_{\tau }\subset F(G_{\tau }).$ 
Under this situation, we have several very strong results. 
Note that there exists an example of $\tau \in {\frak M}_{1,c}({\cal P})$ with $\sharp \G _{\tau }<\infty $ 
such that $J_{\ker }(G_{\tau })=\emptyset $, $J(G_{\tau })\neq \emptyset $, $S_{\tau }\subset F(G_{\tau })$, 
and $\tau $ is not mean stable (see Example~\ref{ex:stfnms}). 
\begin{thm}[Cooperation Principle VI, Exponential Rate of Convergence: see Theorem~\ref{t:kjemfhf}]
\label{t:kjemfhfint}
Let $\tau \in {\frak M}_{1,c}(\emRat)$. Suppose that $J_{\ker }(G_{\tau })=\emptyset $, 
$J(G_{\tau })\neq \emptyset $, and $S_{\tau }\subset F(G_{\tau }).$ 
Then, there exists a constant $\alpha \in (0,1)$, a constant $\lambda \in (0,1)$, and a constant $C>0$ such that 
for each $\varphi \in C^{\alpha }(\CCI )$, 
$\| M_{\tau }^{n}(\varphi -\pi _{\tau }(\varphi ))\| _{\alpha }\leq C\lambda ^{n}\| \varphi \| _{\alpha }$
 for each $n\in \NN .$ 
\end{thm} 
For the proof of Theorem~\ref{t:kjemfhfint}, we need some careful arguments on the 
hyperbolic metric on each connected component of $F(G_{\tau }).$

We remark that in 1983, by numerical experiments, K. Matsumoto and I. Tsuda (\cite{MT}) observed that 
if we add some uniform noise to the dynamical system associated with iteration of a chaotic map on the unit interval $[0,1]$, 
then under certain conditions, 
the quantities which represent chaos (e.g.,  entropy, Lyapunov exponent, etc.) decrease. More precisely, 
they observed that the entropy decreases and the Lyapunov exponent turns negative. They called this phenomenon 
``noise-induced order'', and many physicists have investigated it by numerical experiments, although there has been 
only a few mathematical supports for it. In this paper, we deal with not only (i.i.d.) random dynamical systems which are constructed by 
adding relatively small noises to usual dynamical systems but also more general (i.i.d.) random dynamical systems.   
In this paper, we study ``randomness-induced phenomena'' (the phenomena in general random dynamical systems 
which cannot hold in the usual iteration dynamics) which include ``noise-induced phenomena''.

\begin{rem}
\label{r:holb}
Let $\tau \in {\frak M}_{1,c}(\Rat)$ be mean stable and suppose 
$J(G_{\tau })\neq \emptyset .$ Then 
by \cite[Theorem 3.15]{Splms10}, the chaos of the averaged system of $\tau $ disappears (Cooperation Principle II),  
and by Theorem~\ref{t:kjemfhfint}, there exists an $\alpha _{0}\in (0,1)$ such that for each 
$\alpha \in (0,\alpha _{0})$ the action of $\{ M_{\tau }^{n}\} _{n\in \NN }$ on $C^{\alpha }(\CCI )$ is well-behaved.  
However, \cite[Theorem 3.82]{Splms10} tells us that 
under certain conditions on a mean stable $\tau $, there exists a $\beta \in (0,1)$ such that 
any non-constant element $\varphi \in U_{\tau }$ does not belong to $C^{\beta }(\CCI )$ (note: for the proof of this result, 
we use the Birkhoff ergodic theorem and potential theory). Hence, there exists an element  
$\psi  \in C^{\beta }(\CCI )$ such that  
$\| M_{\tau }^{n}(\psi ) \| _{\beta }\rightarrow \infty $ as $n\rightarrow \infty .$ 
Therefore, the action of $\{ M_{\tau }^{n}\} _{n\in \NN }$ on $C^{\beta }(\CCI )$ is not well behaved. 
In other words, regarding the dynamics of the averaged system of $\tau $, 
there still exists a kind of chaos (or complexity) in the space $(C^{\beta }(\CCI ),\| \cdot \| _{\beta })$ 
even though there exists no chaos in the space $(C(\CCI ),\| \cdot \| _{\infty }).$ 
From this point of view, 
in the field of random dynamics, 
we have a kind of gradation or stratification between chaos and non-chaos. 
It may be nice to investigate and reconsider the chaos theory and mathematical modeling 
from this point of view.  

\end{rem}

Let $\tau \in {\frak M}_{1,c}(\Rat).$ 
We now consider the spectrum Spec$_{\alpha }(M_{\tau })$ of $M_{\tau }:C^{\alpha }(\CCI )\rightarrow C^{\alpha }(\CCI ).$ 
Note that since the family $\mbox{supp}\, \tau $ in $\Rat $ is uniformly Lipschitz continuous on $\CCI $, 
we have $M_{\tau }(C^{\alpha }(\CCI ))\subset C^{\alpha }(\CCI )$ for each $\alpha \in (0,1).$ 
If $\tau $ satisfies the assumptions of Theorem~\ref{t:kjemfhfint}, then from Theorem~\ref{t:kjemfhfint}, 
denoting by ${\cal U}_{v,\tau}(\CCI )$ the set of unitary eigenvalues 
of $M_{\tau }:C(\CCI )\rightarrow C(\CCI )$ (note: by Theorem~\ref{t:utaucaint}, 
${\cal U}_{v,\tau }(\CCI )\subset \mbox{Spec}_{\alpha }(M_{\tau })$ for some $\alpha \in (0,1)$), 
we can show that 
the distance between ${\cal U}_{v,\tau }(\CCI )$ and 
$\mbox{Spec}_{\alpha }(M_{\tau })\setminus {\cal U}_{v,\tau }(\CCI )$ is positive. 
\begin{thm}[see Theorem~\ref{t:kjemfsp}]
\label{t:kjemfspint}
Under the assumptions of Theorem~\ref{t:kjemfhfint}, 
$\mbox{{\em Spec}}_{\alpha }(M_{\tau })\subset \{ z\in \CC \mid |z|\leq \lambda \} 
\cup {\cal U}_{v,\tau }(\CCI )$, 
where $\alpha \in (0,1)$ and $\lambda \in (0,1)$ are the constants in Theorem~\ref{t:kjemfhfint}. 
\end{thm} 
Combining Theorem~\ref{t:kjemfspint} and perturbation theory for linear operators (\cite{K}), 
we obtain the following theorem. 
We remark that even if $g_{n}\rightarrow g$ in Rat, for a $\varphi \in C^{\alpha }(\CCI )$, 
$\| \varphi \circ g_{n}-\varphi \circ g\| _{\alpha }$ does not tend to zero in general. 
Thus when we perturb generators $\{ h_{j}\} $ of $\G _{\tau }$,  
we cannot apply perturbation theory for $M_{\tau }$ on $C^{\alpha }(\CCI )$. However, 
by using the method in the proofs of \cite[Lemmas 5.1, 5.2]{SU1}, it is easy to see that 
for each $\alpha \in (0,1)$, for a fixed generator system 
$(h_{1},\ldots ,h_{m})\in (\Rat )^{m}$, the map  
$(p_{1},\ldots ,p_{m})\in {\cal W}_{m}:=\{ (a_{1},\ldots ,a_{m})\in (0,1)^{m}\mid 
\sum _{j=1}^{m}a_{j}=1\} \mapsto M_{\sum _{j=1}^{m}p_{j}\delta _{h_{j}}}\in 
L(C^{\alpha }(\CCI ))$ is real-analytic, where $L(C^{\alpha }(\CCI )) $ denotes the 
Banach space of bounded linear operators on $C^{\alpha }(\CCI )$ endowed with 
the operator norm. Thus we can apply perturbation theory for the above 
real-analytic family of operators. 
  
\begin{thm}[see Theorem~\ref{t:kjemfsppt}]
\label{t:kjemfspptint}
Let $m\in \NN $ with $m\geq 2.$ 
Let $h_{1},\ldots ,h_{m}\in \emRat$. 
Let $G=\langle h_{1},\ldots ,h_{m}\rangle .$ 
Suppose that 
$J_{\ker }(G)=\emptyset, J(G)\neq \emptyset $ and 
$\cup _{L\in \emMin(G,\CCI )}L\subset F(G).$ 
Let ${\cal W}_{m}:= \{ (a_{1},\ldots ,a_{m})\in (0,1)^{m}\mid \sum _{j=1}^{m}a_{j}=1 \} 
\cong \{ (a_{1},\ldots ,a_{m-1})\in (0,1)^{m-1}\mid \sum _{j=1}^{m-1}a_{j}<1 \}.$ 
For each $a=(a_{1},\ldots ,a_{m})\in {\cal W}_{m}$, 
let $\tau _{a}:= \sum _{j=1}^{m}a_{j}\delta _{h_{j}}\in {\frak M}_{1,c}(\emRat).$ 
Then we have all of the following. 
\begin{itemize}
\item[{\em (1)}] 
For each $b\in {\cal W}_{m}$,  
there exists an $\alpha \in (0,1)$ and an open neighborhood of $V_{b}$ of $b$ in ${\cal W}_{m}$ 
such that for each $a\in V_{b}$, we have 
$\mbox{LS}({\cal U}_{f,\tau _{a}}(\CCI ))\subset C^{\alpha }(\CCI )$, 
$\pi _{\tau _{a}}(C^{\alpha }(\CCI ))\subset C^{\alpha }(\CCI )$ and 
$(\pi _{\tau _{a}}:C^{\alpha }(\CCI )\rightarrow C^{\alpha }(\CCI ))\in L(C^{\alpha }(\CCI ))$, 
and such that the map   
$a\mapsto (\pi _{\tau _{a}}:C^{\alpha }(\CCI )\rightarrow C^{\alpha }(\CCI ))\in 
L(C^{\alpha }(\CCI ))$  
 is real-analytic in $V_{b}.$ 
\item[{\em (2)}] 
Let $L\in \emMin( G,\CCI ).$ 
Then, for each $b\in {\cal W}_{m}$, there exists an $\alpha \in (0,1)$ such that 
the map $a\mapsto T_{L,\tau _{a}}\in (C^{\alpha }(\CCI ),\| \cdot \| _{\alpha })$ is real-analytic 
in an open neighborhood of $b$ in ${\cal W}_{m}.$ Moreover, 
the map $a\mapsto T_{L,\tau _{a}}\in (C(\CCI ),\| \cdot \| _{\infty })$ is real-analytic in 
${\cal W}_{m}.$ In particular, for each 
$z\in \CCI $, the map $a\mapsto T_{L,\tau _{a}}(z)$ is real-analytic in ${\cal W}_{m}.$ 
Furthermore, for any multi-index $n=(n_{1},\ldots ,n_{m-1})\in (\NN \cup \{ 0\})^{m-1}$ and for any 
$b\in {\cal W}_{m},$  
 the function 
 $z\mapsto [(\frac{\partial }{\partial a_{1}})^{n_{1}}\cdots (\frac{\partial }{\partial a_{m-1}})^{n_{m-1}}
(T_{L,\tau _{a}}(z))]|_{a=b}$ is H\"{o}lder continuous on $\CCI $ and is locally constant on $F(G).$ 
\item[{\em (3)}] 
Let $L\in \emMin(G,\CCI )$ and let $b\in {\cal W}_{m}.$ 
For each $i=1,\ldots ,m-1$ and for each $z\in \CCI $, let 
$\psi _{i,b}(z):=[\frac{\partial }{\partial a_{i}}(T_{L,\tau _{a}}(z))]|_{a=b}$  
and let $\zeta _{i,b}(z):= T_{L,\tau _{b}}(h_{i}(z))-T_{L,\tau _{b}}(h_{m}(z)).$  
Then, $\psi _{i,b}$ is the unique solution of 
the functional equation $(I-M_{\tau _{b}})(\psi )=\zeta _{i,b}, \psi |_{S_{\tau _{b}}}=0, \psi \in C(\CCI )$,  
where $I$ denotes the identity map. Moreover, there exists a number $\alpha \in (0,1)$ such that 
$\psi _{i,b}=\sum _{n=0}^{\infty }M_{\tau _{b}}^{n}(\zeta _{i,b})$ in $(C^{\alpha }(\CCI ),\| \cdot \| _{\alpha }).$

\end{itemize}  
\end{thm}
\begin{rem}[see also Example~\ref{ex:dc1}] 
\label{r:takagi}
(1) By Theorem~\ref{t:pmsodint}-(2), 
the set of all finite subsets $\G $ of ${\cal P}$ satisfying 
the assumption 
``$J_{\ker }(\langle \G\rangle )=\emptyset, J(\langle \G\rangle )\neq \emptyset $ and 
$\cup _{L\in \Min(\langle \G \rangle ,\CCI )}L\subset F(\langle \G \rangle )$'' of Theorem~\ref{t:kjemfspptint} is dense in 
Cpt$({\cal P})$ with respect to the Hausdorff metric.   
(2) The function $T_{L,\tau _{a}}$ defined on $\CCI $ can be regarded as a complex analogue of 
Lebesgue singular functions or the devil's staircase, and 
the function $z\mapsto \psi _{i,b}(z)=[\frac{\partial }{\partial a_{i}}(T_{L,\tau _{a}}(z))]|_{a=b}$ defined on 
$\CCI $ can be regarded as a {\bf complex analogue of the Takagi function} 
${\cal T}(x):=\sum _{n=0}^{\infty}\frac{1}{2^{n}}\min _{m\in \ZZ }|2^{n}x-m|$ where $x\in \RR $. 
(The Takagi function ${\cal T}$ has many interesting properties. For example, 
it is continuous but nowhere differentiable on $\RR .$ There are many studies on the Takagi function. See \cite{YHK, HY, SeSh, AK}.)   
In order to explain the details, let $g_{1}(x):= 2x, g_{2}(x):= 2(x-1)+1\ (x\in \RR )$ and let 
 $0<a<1$ be a constant. We consider the random dynamical system on $\RR $ such that 
 at every step we choose the map $g_{1}$ with probability $a$ and the map $g_{2}$ with 
probability $1-a.$ Let $T_{+\infty ,a}(x)$ be the probability of tending to 
$+\infty $ starting with the initial value $x\in \RR .$ 
Then, as the author of this paper pointed out in \cite{Splms10}, 
we can see that for each $a\in (0,1)$ with $a\neq 1/2$, 
the function $T_{+\infty ,a}|_{[0,1]}: [0,1]\rightarrow [0,1]$ is 
equal to {\bf Lebesgue's singular function} $L_{a}$ with respect to the parameter $a$. (For the definition of $L_{a}$, see 
\cite{YHK}. See  Figure~\ref{fig:lebfcn}, \cite{Splms10}.) 
The author found that, in a similar way, many singular functions on $\RR $ (including {\bf the devil's staircase}) 
can be regarded as the functions of probability of tending to $+\infty $ with respect to 
some random dynamical systems on $\RR $ (\cite{Splms10,Ssugexp}).   
It is well-known (see \cite{YHK,SeSh}) that for each $x\in [0,1]$, 
$a\mapsto L_{a}(x)$ is real-analytic in $(0,1)$, and that 
$x\mapsto (1/2)[\frac{\partial }{\partial a}(L_{a}(x))]|_{a=1/2}$ is equal to the Takagi function 
restricted to $[0,1]$ (Figure~\ref{fig:lebfcn}).  From this point of view, 
the function $z\mapsto \psi _{i,b}(z)$ defined on 
$\CCI $ can be regarded as a complex analogue of the Takagi function. This is a new concept introduced in this paper. 
In fact, the author found that by using random dynamical systems and the methods in this paper, 
 we can find many analogues of the devil's staircase, Lebesgue's singular functions and the Takagi function 
(on $[-\infty ,\infty ], \CCI $ etc.).  
For the figure of the graph of $\psi _{i,b}$, see Example~\ref{ex:dc1} and Figure~\ref{fig:ctgraphgrey1}. 
Some results on the (non-)differentiability of $T_{L,\tau _{a}}$ were obtained in \cite{Splms10}.  
\begin{figure}[htbp]
\caption{From left to right, the graphs of the devil's staircase, Lebesgue's singular function and the Takagi function. 
The Takagi function is continuous but nowhere differentiable on $\RR .$}
\begin{center}
\includegraphics[width=2.3cm,height=2.3cm, origin =c, angle =-90]{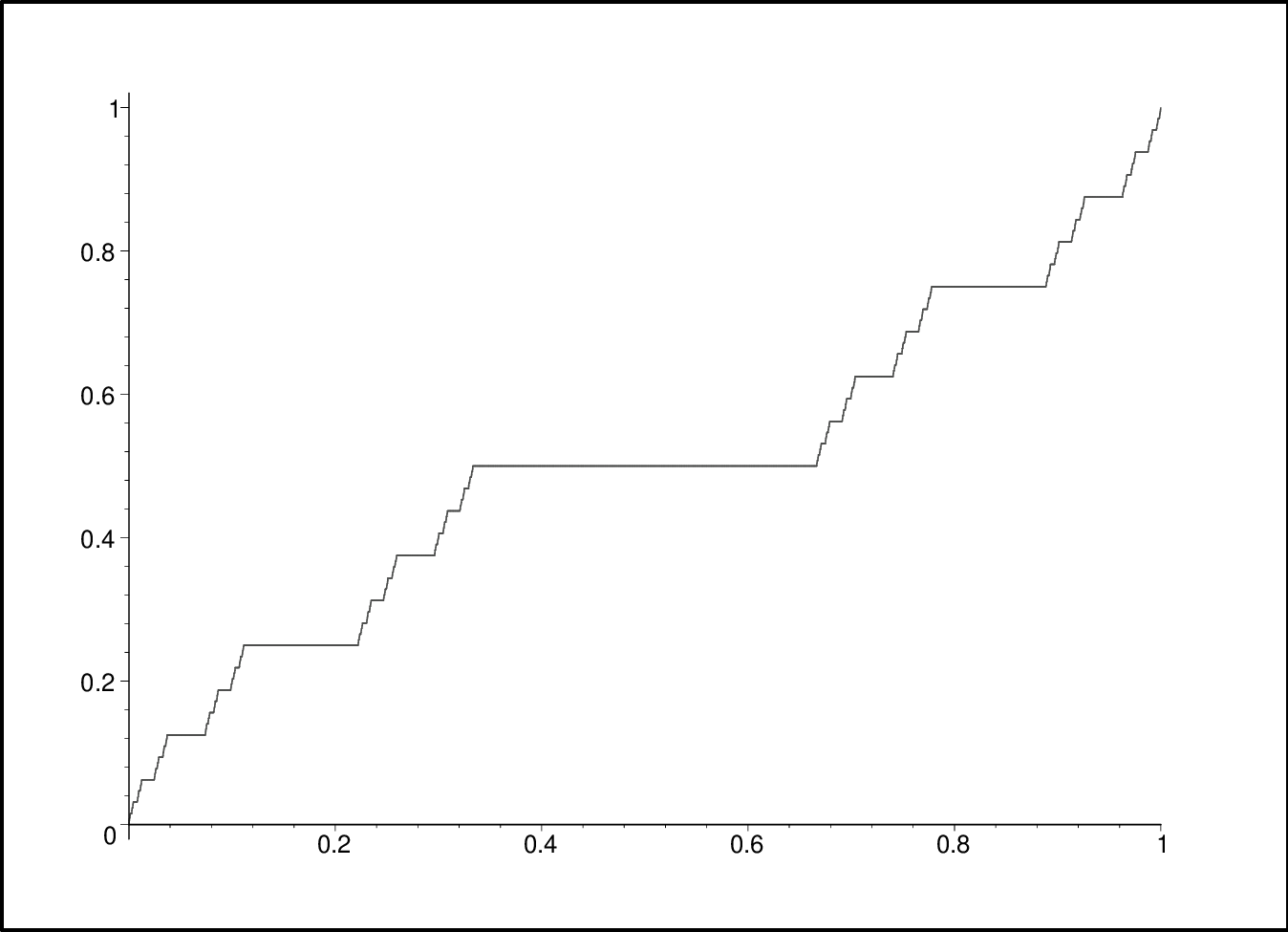}
\includegraphics[width=2.3cm,height=2.3cm, origin =c, angle =-90]{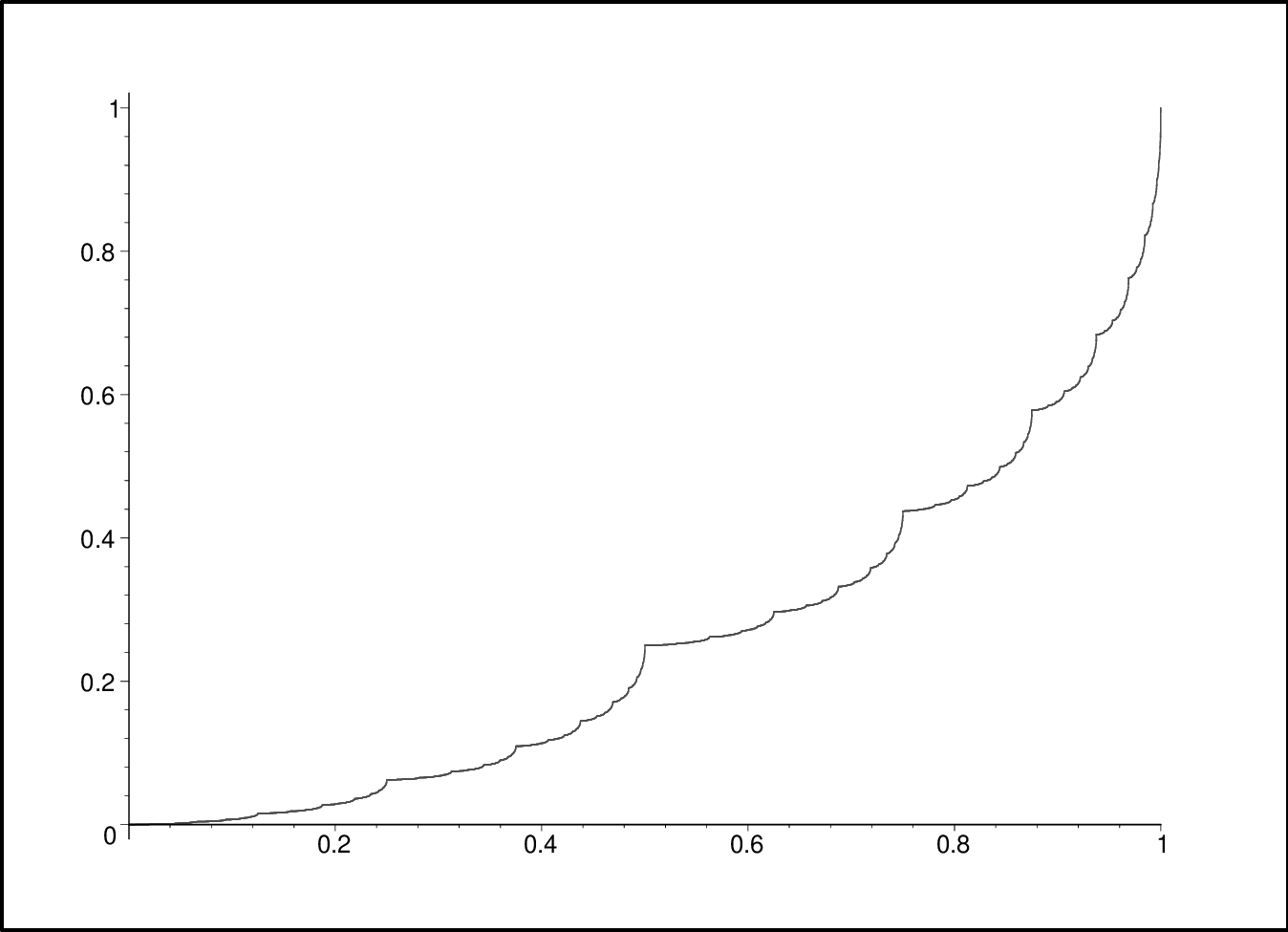}\label{fig:lebfcn}
\includegraphics[width=2.3cm,height=2.3cm, origin =c, angle =-90]{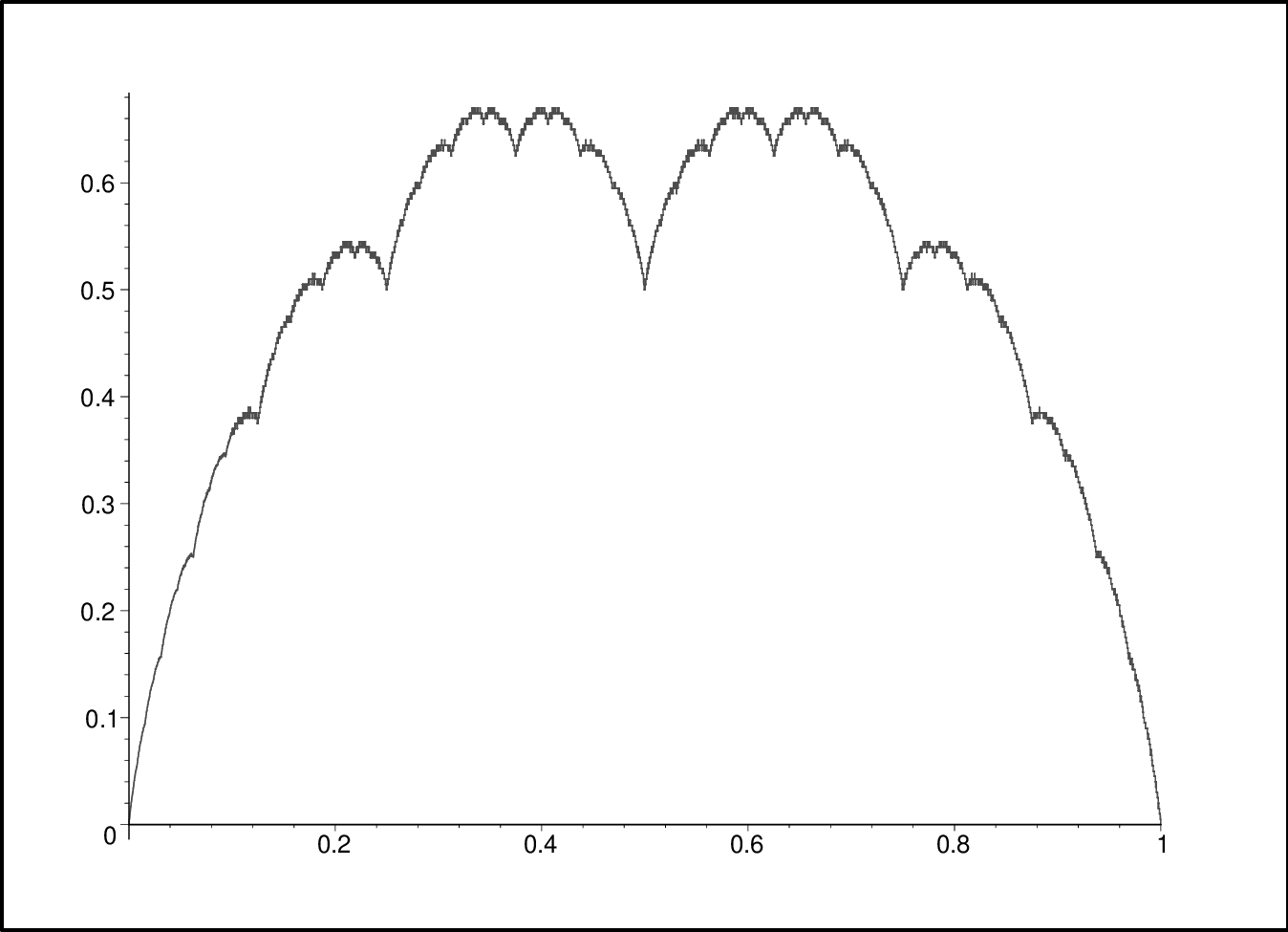}
\end{center}
\end{figure}
\begin{figure}[htbp]
\caption{The Julia set of $G=\langle h_{1}, h_{2}\rangle $, where   
$g_{1}(z):=z^{2}-1, g_{2}(z):=z^{2}/4, h_{1}:=g_{1}^{2}, h_{2}:=
g_{2}^{2}.$ 
The planar postcritical set of $G$ is bounded in $\CC $, 
$J(G)$ is not connected and  
$G$ is hyperbolic (\cite{SdpbpIII}).  
$(h_{1},h_{2})$ satisfies the open set condition and $\dim _{H}(J(G))<2$ (\cite{SU1}). Moreover,  
 for each connected component $J$ of $J(G)$, $\exists !\gamma \in 
\{ h_{1},h_{2}\} ^{\NN }$ 
s.t. $J=J_{\gamma }.$ For almost every $\gamma \in \{ h_{1},h_{2}\} ^{\NN }$ with respect to 
a Bernoulli measure, 
$J_{\gamma }$ is a simple closed curve but not a quasicircle, and the basin $A_{\gamma }$ of infinity for the sequence 
$\g $ is a John domain 
(\cite{SdpbpIII}).}
\begin{center}
\includegraphics[width=2.5cm,width=2.5cm]{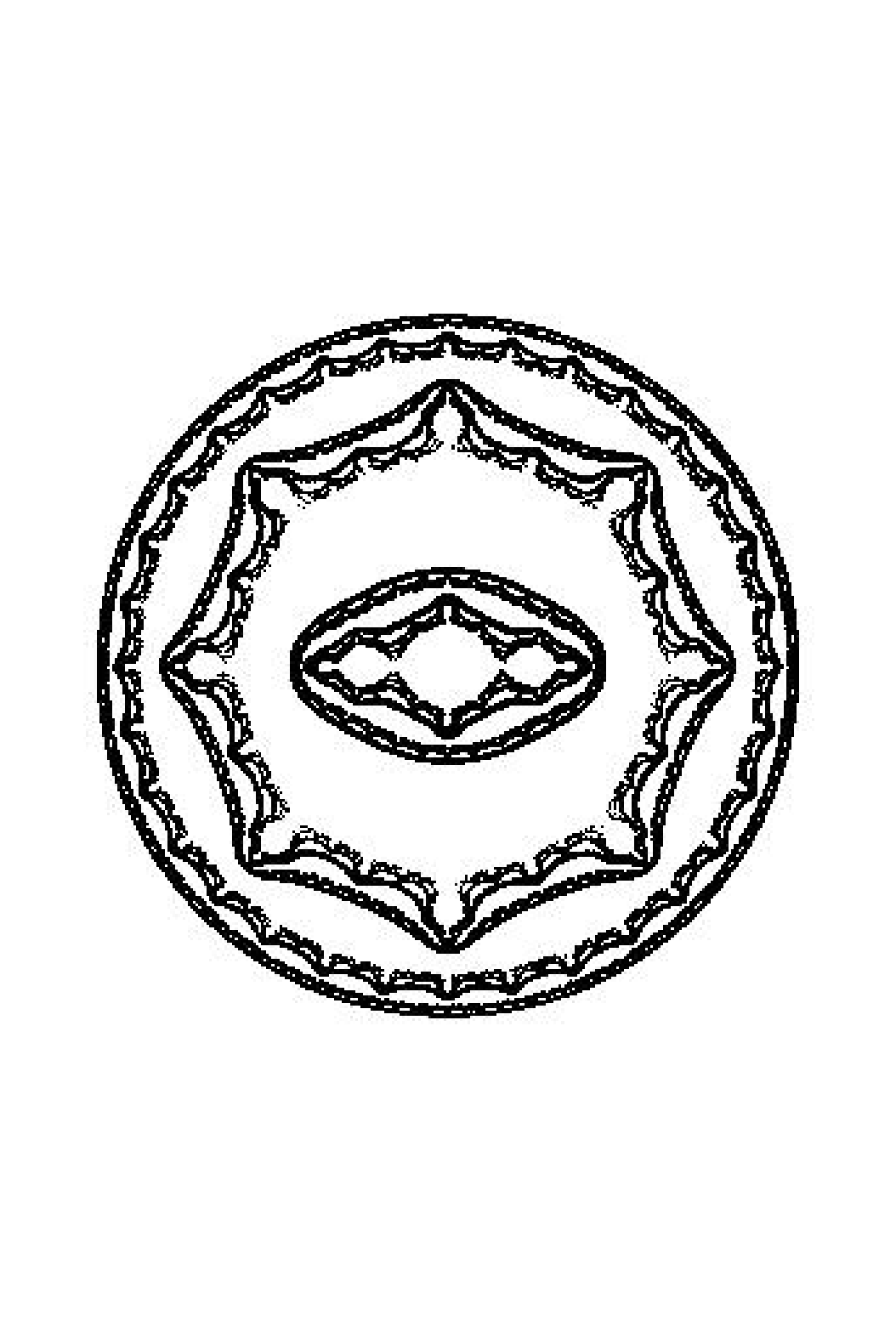}
\label{fig:dcjulia}
\end{center} 
\end{figure} 
\begin{figure}[htbp]
\caption{The graph of $z\mapsto T_{\infty ,\tau _{(1/2,1/2)}}(z)$, where, 
letting $(h_{1},h_{2})$ be the element in Figure~\ref{fig:dcjulia},  
we set $\tau _{a}:=\sum _{j=1}^{2}a_{j}\delta _{h_{j}}$ for each $a\in {\cal W}_{2}.$  
$\tau _{a}$ is mean stable.   
A devil's coliseum (a complex analogue of the devil's staircase or Lebesgue's singular functions).  
This function is continuous on $\CC $ and the set of varying points is equal to $J(G)$ in Figure~\ref{fig:dcjulia}.}
\begin{center}
\includegraphics[width=3.7cm,width=3.7cm]{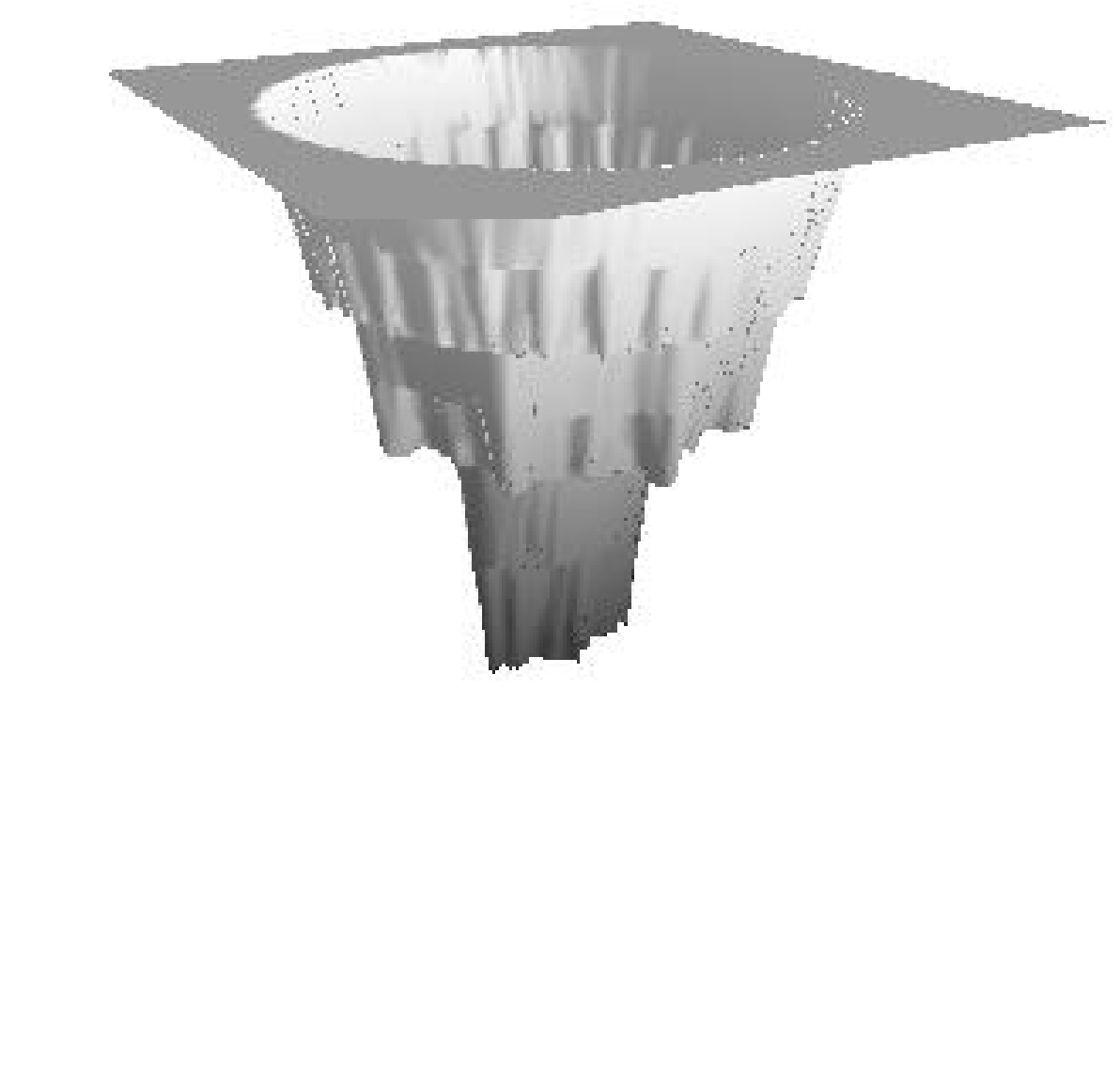}
\label{fig:dcgraphgrey2}
\end{center}
\end{figure}
\begin{figure}[htbp]  
\caption{The graph of $z\mapsto [(\partial T_{\infty ,\tau _{a}}(z)/\partial a_{1})]|_{a_{1}=1/2}$,  
where,  $\tau _{a}$ is the element in Figure~\ref{fig:dcgraphgrey2}. 
A {\bf complex analogue of the Takagi function}. This function is continuous on $\CC $ and the set of varying points is included in $J(G)$ in Figure~\ref{fig:dcjulia}.} 
\begin{center}
\includegraphics[width=3.7cm,width=3.7cm]{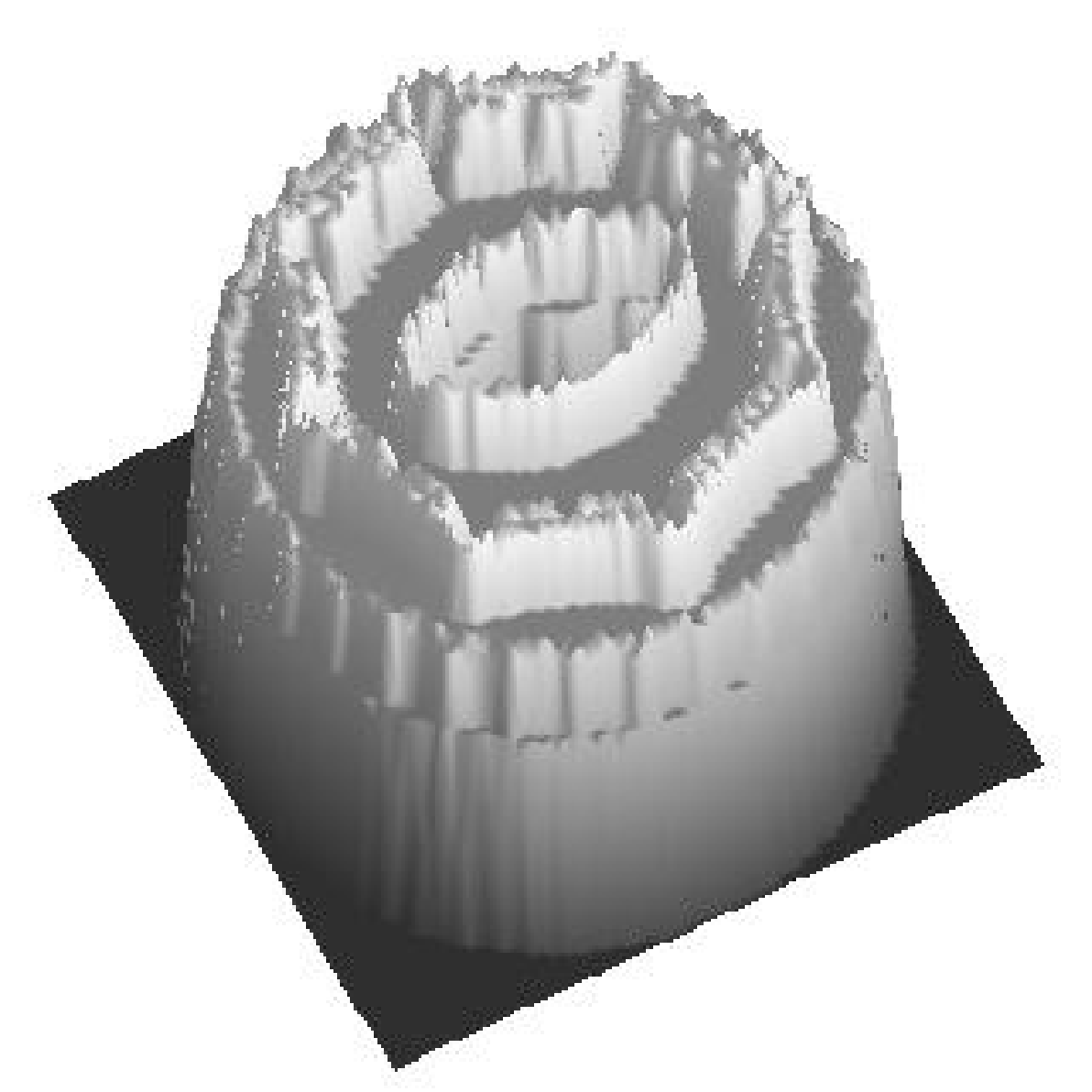}
\label{fig:ctgraphgrey1}
\end{center} 
\end{figure}

\end{rem} 

In this paper, we present a result on the non-differentiability of the function 
$\psi _{i,b}(z)$ of Theorem~\ref{t:kjemfspptint} at points in 
$J(G_{\tau })$ (Theorem~\ref{t:psinondiff}), which is obtained by the application 
of the Birkhoff ergodic theorem, potential theory and some results from \cite{Splms10}.  

Combining these results, we can say that 
for a generic $\tau \in {\frak M}_{1,c}({\cal P})$, 
the chaos of the averaged system associated with $\tau $ disappears, 
the Lyapunov exponents are negative, 
$\sharp (\Min (G_{\tau },\CCI ))<\infty $, each $L\in \Min(G_{\tau },\CCI )$ is attracting, 
there exists a stability on $U_{\tau }$ and $\Min (G_{\tau },\CCI )$ 
in a neighborhood of $\tau $ in $({\frak M}_{1,c}({\cal P}),{\cal O})$, 
and there exists an $\alpha \in (0,1)$ such that 
for each $\varphi \in C^{\alpha }(\CCI )$, 
$M_{\tau }^{n}(\varphi )$ tends to the space $U_{\tau }$ exponentially fast. 
Note that these phenomena can hold in random complex dynamics but cannot hold in the 
usual iteration dynamics of a single rational map $h$ with $\deg (h)\geq 2.$ 
We systematically investigate these phenomena and their mechanisms. 
As the author mentioned in Remark~\ref{r:holb}, these results 
will stimulate the chaos theory and the mathematical modeling in various fields, 
and will lead us to a new interesting field. Moreover, these results are related to fractal geometry very deeply. 

 In section~\ref{Pre}, we give some basic notations and definitions. 
 In section~\ref{Results}, we present the main results of this paper. 
 In section~\ref{Tools}, we give some basic tools to prove the main results. 
 In section~\ref{Proofs}, we give the proofs of the main results. 
 In section~\ref{Examples}, we present several examples which describe the main results.  

\noindent {\bf Acknowledgment}. The author thanks Rich Stankewitz for valuable comments. 
This work was partially supported by JSPS Grant-in-Aid for Scientific Research (C) 21540216.  
\section{Preliminaries}
\label{Pre} 
In this section, we give some fundamental notations and definitions. 

\noindent {\bf Notation:} 
Let $(X,d)$ be a metric space, $A$ a subset of $X$, and $r>0$. We set 
$B(A,r):= \{ z\in X\mid d(z,A)<r\} .$ Moreover, 
for a subset $C$ of $\CC $, we set 
$D(C,r):= \{ z\in \CC \mid \inf _{a\in C}|z-a|<r\} .$ 
Moreover, for any topological space $Y$ and for any subset $A$ of $Y$, we denote by int$(A)$ the set of all interior points of $A.$ 
We denote by Con$(A)$ the set of all connected components of $A.$  
\begin{df}
Let $Y$ be a metric space. 
We set 
$C(Y):= \{ \varphi :Y\rightarrow \CC \mid \varphi \mbox{ is continuous }\} .$ 
When $Y$ is compact, we endow $C(Y)$ with the supremum norm $\| \cdot \| _{\infty }.$ 
Moreover, for a subset ${\cal F}$ of $C(Y)$, we set 
${\cal F}_{nc}:=\{ \varphi \in {\cal F}\mid \varphi \mbox{ is not constant}\} .$ 
\end{df}
\begin{df} A {\bf rational semigroup} is a semigroup  
generated by a family of non-constant rational maps on 
the Riemann sphere $\CCI $ with the semigroup operation being 
functional composition(\cite{HM,GR}). A 
{\bf polynomial semigroup } is a 
semigroup generated by a family of non-constant 
polynomial maps. 
We set 
Rat : $=\{ h:\CCI \rightarrow \CCI \mid 
h \mbox { is a non-constant rational map}\} $
endowed with the distance $\kappa $ which is defined 
by $\kappa (f,g):=\sup _{z\in \CCI }d(f(z),g(z))$, where $d$ denotes the 
spherical distance on $\CCI .$   
Moreover, we set 
$\Ratp:=\{ h\in \mbox{Rat}\mid \deg (h)\geq 2\} $ endowed with the 
relative topology from Rat. 
Furthermore, we set 
${\cal P}:= \{ g:\CCI \rightarrow \CCI \mid 
g \mbox{ is a polynomial}, \deg (g)\geq 2\} $
endowed with the relative topology from 
Rat.  

\end{df}
\begin{rem}[\cite{Be}] 
For each $d\in \NN $, let Rat$_{d}:=\{ g\in \Rat\mid \deg (g)=d\}$ and for each 
$d\in \NN $ with $d\geq 2$, let ${\cal P}_{d}:=\{ g\in {\cal P}\mid \deg (g)=d\} .$ 
Then for each $d$, $\Rat_{d}$ (resp. ${\cal P}_{d}$) is a connected component of 
$\Rat $ (resp. ${\cal P}$). Moreover, 
$\Rat_{d}$ (resp. ${\cal P}_{d}$) is open and closed in $\Rat $ (resp. ${\cal P}$) and 
is a finite dimensional complex manifold. 
Furthermore, $h_{n}\rightarrow h$ in ${\cal P}$ if and only if $\deg (h_{n})=\deg (h)$ for each large $n$ and 
the coefficients of $h_{n}$ tend to the coefficients of $h$ appropriately as $n\rightarrow \infty .$  
\end{rem}
\begin{df}
Let $G$ be a rational semigroup. 
The {\bf Fatou set }of $G$ is defined to be  
$F(G):=\\ \{ z\in \CCI  \mid \exists \mbox{ neighborhood } U \mbox{ of }z$  
\mbox{s.t.} $\{ g|_{U}:U\rightarrow \CCI \} _{g\in G}$ is equicontinuous 
on  $U \} .$ (For the definition of equicontinuity, see \cite{Be}.) 
The {\bf Julia set }of $G$ is defined to be 
 $J(G):= \CCI  \setminus F(G).$ 
If $G$ is generated by $\{ g_{i}\} _{i}$, then 
we write $G=\langle g_{1},g_{2},\ldots \rangle .$
If $G$ is generated by a subset $\G $ of $\Rat$, then we write
 $G=\langle \G \rangle .$  
For finitely many elements $g_{1},\ldots, g_{m}\in \Rat$, 
we set  $F(g_{1},\ldots ,g_{m}):=F(\langle g_{1},\ldots ,g_{m}\rangle )$ 
and $J(g_{1},\ldots ,g_{m}):=J(\langle g_{1},\ldots ,g_{m}\rangle )$.  
For a subset $A$ of $\CCI $, we set 
$G(A):= \bigcup _{g\in G}g(A)$ and 
$G^{-1}(A):= \bigcup _{g\in G}g^{-1}(A).$ 
We set $G^{\ast }:= G\cup \{ \mbox{Id}\} $, where 
Id denotes the identity map. 
\end{df}
\begin{lem}[\cite{HM, GR}] 
\label{ocminvlem}
Let $G$ be a rational semigroup. 
Then, for each $h\in G$, $h(F(G))\subset F(G)$ and $h^{-1}(J(G))\subset J(G).$ 
Note that the equality does not hold in general. 
\end{lem}

The following is the key to investigating random complex dynamics. 
\begin{df}
Let $G$ be a rational semigroup. 
We set $J_{\ker }(G):= \bigcap _{g\in G}g^{-1}(J(G)).$ 
This is called the {\bf kernel Julia set} of $G.$  
\end{df}
\begin{rem}
\label{r:kjulia}
Let $G$ be a rational semigroup. 
(1) $J_{\ker }(G)$ is a compact subset of $J(G).$ 
(2) For each $h\in G$, 
$h(J_{\ker }(G))\subset J_{\ker }(G).$  
(3) If $G$ is a rational semigroup and 
if $F(G)\neq \emptyset $, then 
int$(J_{\ker }(G))=\emptyset .$ 
(4) If $G$ is generated by a single map or if $G$ is a group, then 
$J_{\ker }(G)=J(G).$ However, for a general rational semigroup $G$, 
it may happen that $\emptyset =J_{\ker }(G)\neq J(G)$ 
(see \cite{Splms10}).  
\end{rem}
It is sometimes important to investigate the dynamics of sequences of maps. 
\begin{df}
For each $\gamma =(\gamma _{1},\gamma _{2},\ldots )\in 
(\Rat)^{\NN }$ and each $m,n\in \NN $ with $m\geq n$, we set
$\gamma _{m,n}=\gamma _{m}\circ \cdots \circ \gamma _{n}$ and we set  
$$F_{\gamma }:= \{ z\in \CCI \mid 
\exists \mbox{ neighborhood } U \mbox{ of } z \mbox{ s.t. } 
\{ \gamma _{n,1}\} _{n\in \NN } \mbox{ is equicontinuous on }  
U\} $$ and $J_{\gamma }:= \CCI  \setminus F_{\gamma }.$ 
The set $F_{\gamma }$ is called the {\bf Fatou set} of the sequence $\gamma $ and 
the set $J_{\gamma }$ is called the {\bf Julia set} of the sequence $\gamma .$ 
\end{df}
\begin{rem}
Let $\gamma \in (\Ratp) ^{\NN }$. Then by \cite[Theorem 2.8.2]{Be}, $J_{\gamma }\neq \emptyset .$ 
Moreover, if $\Gamma $ is a non-empty compact subset of $\Ratp$ and $\gamma \in \Gamma ^{\NN }$, 
then by \cite{S4}, $J_{\gamma }$ is a perfect set and $J_{\gamma }$ has uncountably many points.  
\end{rem}
We now give some notations on random dynamics. 
\begin{df}
\label{d:d0}
For a metric space $Y$, we denote by 
${\frak M}_{1}(Y)$ the space of all Borel probability measures on  $Y$ endowed 
with the topology such that 
$\mu _{n}\rightarrow \mu $ in ${\frak M}_{1}(Y)$ if and only if 
for each bounded continuous function $\varphi :Y\rightarrow \CC $, 
$\int \varphi \ d\mu _{n}\rightarrow \int \varphi \ d\mu .$ 
 Note that if $Y$ is a compact metric space, then 
${\frak M}_{1}(Y)$ is a compact metric space with the metric 
$d_{0}(\mu _{1},\mu _{2}):=\sum _{j=1}^{\infty }\frac{1}{2^{j}}
\frac{|\int \phi _{j}d\mu _{1}-\int \phi _{j}d\mu _{2}|}{1+|\int \phi _{j}d\mu _{1}-\int \phi _{j}d\mu _{2}|}$, 
where $\{ \phi _{j}\} _{j\in \NN }$ is a dense subset of $C(Y).$  
Moreover, for each $\tau \in {\frak M}_{1}(Y)$, 
we set $\suppt :=\{ z\in Y\mid \forall \mbox{ neighborhood } U \mbox{ of }z,\ 
\tau (U)>0\} .$ Note that $\suppt $ is a closed subset of $Y.$ 
Furthermore, 
we set ${\frak M}_{1,c}(Y):= \{ \tau \in {\frak M}_{1}(Y)\mid \suppt \mbox{ is compact}\} .$ 

For a complex Banach space ${\cal B}$, we denote by ${\cal B}^{\ast }$ the 
space of all continuous complex linear functionals $\rho :{\cal B}\rightarrow \CC $, 
endowed with the weak$^{\ast }$ topology. 
\end{df}
For any $\tau \in {\frak M}_{1}(\Rat)$, we will consider the i.i.d. random dynamics on $\CCI $ such that 
at every step we choose a map $g\in \Rat $ according to $\tau $ 
(thus this determines a time-discrete Markov process with time-homogeneous transition probabilities 
on the phase space 
$\CCI $ such that for each $x\in \CCI $ and 
each Borel measurable subset $A$ of $\CCI $, 
the transition probability 
$p(x,A)$ from $x$ to $A$ is defined as $p(x,A)=\tau (\{ g\in \Rat \mid g(x)\in A\} )$).  
\begin{df} 
\label{d:ytau}
Let $\tau \in {\frak M}_{1}(\Rat).$  
\begin{enumerate}
\item  We set $\Gamma _{\tau }:= \mbox{supp}\, \tau $ (thus $\Gamma _{\tau }$ is a 
closed subset of $\Rat $).     
Moreover, we set $X_{\tau }:= (\Gamma _{\tau })^{\NN }$ $
 (=\{ \gamma  =(\gamma  _{1},\gamma  _{2},\ldots )\mid \gamma  _{j}\in \Gamma _{\tau }\ (\forall j)\} )$ endowed with the product topology.  
Furthermore, we set $\tilde{\tau }:= \otimes\displaystyle _{j=1}^{\infty }\tau .$ 
This is the unique Borel probability measure 
on $X_{\tau }$ such that for each cylinder set 
$A=A_{1}\times \cdots \times A_{n}\times \Gamma _{\tau }\times  \Gamma _{\tau }\times \cdots $ in 
$X_{\tau }$, $\tilde{\tau }(A)=\prod _{j=1}^{n}\tau (A_{j}).$ 
 We denote by $G_{\tau }$ the subsemigroup of $\Rat$ generated by 
the subset $\Gamma _{\tau }$ of $\Rat .$   
\item 
Let $M_{\tau }$ be the operator 
on $C(\CCI ) $ defined by $M_{\tau }(\varphi )(z):=\int _{\Gamma _{\tau }}\varphi (g(z))\ d\tau (g).$ 
$M_{\tau }$ is called the {\bf transition operator} of the Markov process induced by $\tau .$ 
Moreover, let $M_{\tau }^{\ast }: C(\CCI )^{\ast }
\rightarrow C(\CCI )^{\ast }$ be the dual of $M_{\tau }$, which is defined as 
$M_{\tau }^{\ast }(\mu )(\varphi )=\mu (M_{\tau }(\varphi ))$ for each 
$\mu \in C(\CCI )^{\ast }$ and each $\varphi \in C(\CCI ).$ 
Remark: we have $M_{\tau }^{\ast }({\frak M}_{1}(\CCI ))\subset {\frak M}_{1}(\CCI )$ and  
for each $\mu \in {\frak M}_{1}(\CCI  )$ and each open subset $V$ of $\CCI  $, 
we have $M_{\tau }^{\ast }(\mu )(V)=\int _{\Gamma _{\tau }}\mu (g^{-1}(V))\ d\tau (g).$ 
\item  
We denote by $F_{meas }(\tau )$ the 
set of $\mu \in {\frak M}_{1}(\CCI  )$
satisfying that there exists a neighborhood $B$ 
of $\mu $ in ${\frak M}_{1}(\CCI  )$ such that 
the sequence  $\{ (M_{\tau }^{\ast })^{n}|_{B}: 
B\rightarrow {\frak M}_{1}(\CCI  ) \} _{n\in \NN }$
is equicontinuous on $B.$
We set $J_{meas}(\tau ):= {\frak M}_{1}(\CCI )\setminus 
F_{meas}(\tau ).$
\end{enumerate}
\end{df}
\begin{rem}
\label{r:ggt}
Let $\Gamma $ be a closed subset of Rat. Then there exists a 
$\tau \in {\frak M}_{1}(\mbox{Rat})$ such that $\Gamma _{\tau }=\Gamma .$ 
By using this fact, we sometimes apply the results on random complex dynamics  
to the study of the dynamics of rational semigroups. 
\end{rem}
\begin{df}
\label{d:Phi}
Let $Y$ be a compact metric space. 
Let $\Phi :Y \rightarrow {\frak M}_{1}(Y )$ be the topological embedding 
defined by: $\Phi (z):=\delta _{z}$, where $\delta _{z}$ denotes the 
Dirac measure at $z.$ Using this topological embedding $\Phi :Y \rightarrow {\frak M}_{1}(Y )$, 
we regard $Y $ as a compact subset of ${\frak M}_{1}(Y ).$ 
\end{df}
\begin{rem}
\label{r:Phi}
If $h\in \Rat $ and $\tau =\delta _{h}$, then 
we have $M_{\tau }^{\ast }\circ \Phi = \Phi \circ h$ on $\CCI .$ 
Moreover, for a general $\tau \in {\frak M}_{1}(\Rat)$, 
$M_{\tau }^{\ast }(\mu )=\int h_{\ast }(\mu )d\tau (h)$ for each 
$\mu \in {\frak M}_{1}(\CCI ).$ 
Therefore, for a general $\tau \in {\frak M}_{1}(\Rat)$, 
the map $M_{\tau }^{\ast }:{\frak M}_{1}(\CCI )\rightarrow {\frak M}_{1}(\CCI )$ 
can be regarded as the ``averaged map'' on the extension ${\frak M}_{1}(\CCI )$ of 
$\CCI .$ 
\end{rem}

\begin{rem}
If $\tau =\delta _{h}\in {\frak M}_{1}(\Ratp)$ with $h\in  \Ratp $, then 
$J_{meas}(\tau )\neq \emptyset $. In fact, 
using the embedding $\Phi :\CCI \rightarrow {\frak M}_{1}(\CCI )$, 
we have $\emptyset \neq \Phi (J(h))\subset J_{meas}(\tau ).$     
\end{rem}
The following is an important and interesting object in random dynamics. 
\begin{df}
Let $A$ be a subset of $\CCI .$ 
Let $\tau \in {\frak M}_{1}(\Rat).$ For each $z\in \CCI $,  
we set 
$T_{A,\tau }(z):= \tilde{\tau }(\{ \gamma =(\gamma _{1},\gamma _{2},\ldots )\in X_{\tau }\mid 
d(\gamma _{n,1}(z),A)\rightarrow 0 \mbox{ as } n\rightarrow \infty \}).$ 
This is the probability of tending to $A$ starting with the initial value $z\in \CCI .$    
For any $a\in \CCI $,  we set $T_{a,\tau }:=T_{\{ a\} ,\tau }.$ 
\end{df}
\begin{df}
Let ${\cal B}$ be a complex vector space and let $M:{\cal B}\rightarrow {\cal B}$ be a linear operator.  
Let $\varphi \in {\cal B}$ and $a\in \CC $ be such that 
$\varphi \neq 0, |a|=1$, and $M(\varphi )=a\varphi .$ Then we say that  
$\varphi $ is a unitary eigenvector of $M$ with respect to $a$, 
and we say that $a$ is a unitary eigenvalue.   
\end{df}
\begin{df}
Let $\tau \in {\frak M}_{1}(\Rat).$ 
Let $K$ be a non-empty subset of $\CCI  $ such that 
$G_{\tau }(K)\subset K$. 
We denote by ${\cal U}_{f,\tau }(K)$ the set of 
all unitary eigenvectors of $M_{\tau }:C(K)\rightarrow C(K)$. Moreover, 
we denote by ${\cal U}_{v,\tau }(K)$ the set of all 
unitary eigenvalues of $M_{\tau }:C(K)\rightarrow C(K).$  
Similarly, we denote by 
${\cal U}_{f,\tau ,\ast}(K)$ the set of all unitary eigenvectors of 
$M_{\tau }^{\ast }:C(K)^{\ast }\rightarrow C(K)^{\ast }$, 
 and we denote by ${\cal U}_{v,\tau ,\ast }(K)$ the set of all 
unitary eigenvalues of $M_{\tau }^{\ast }:C(K)^{\ast }\rightarrow C(K)^{\ast }.$ 
\end{df}
\begin{df}
Let $V$ be a complex vector space and let $A$ be a subset of $V.$ 
We set $\mbox{LS}(A):= \{ \sum _{j=1}^{m}a_{j}v_{j}\mid a_{1},\ldots ,a_{m}\in \CC , 
v_{1},\ldots ,v_{m}\in A, m\in \NN \} .$ 
\end{df}
\begin{df}
Let $Y$ be a topological space and let $V$ be a subset of $Y.$ 
We denote by $C_{V}(Y)$ the space of all $\varphi \in C(Y)$ such that 
for each connected component $U$ of $V$, there exists a constant $c_{U}\in \CC $ with 
$\varphi |_{U}\equiv c_{U}.$ 
\end{df}
\begin{df}
For a topological space $Y$, we denote by Cpt$(Y)$ the space of all non-empty compact subsets of $Y$. 
If $Y$ is a metric space, we endow Cpt$(Y)$  
with the Hausdorff metric.
\end{df}
\begin{df}
\label{d:minimal}
Let $G$ be a rational semigroup. 
Let $Y\in \Cpt (\CCI )$ be such that 
$G(Y)\subset Y.$  
Let $K\in \Cpt(\CCI ).$ 
We say that $K$ is a minimal set for $(G,Y)$ if 
$K$ is minimal among the space 
$\{ L\in \Cpt(Y)\mid G(L)\subset L\} $ with respect to inclusion. 
Moreover, we set $\Min(G,Y):= \{ K\in \Cpt(Y)\mid K \mbox{ is a minimal set for } (G,Y)\} .$  
\end{df}
\begin{rem}
\label{r:minimal}
Let $G$ be a rational semigroup. 
By Zorn's lemma, it is easy to see that 
if $K_{1}\in \Cpt(\CCI )$ and $G(K_{1})\subset K_{1}$, then 
there exists a $K\in \Min(G,\CCI )$ with $K\subset K_{1}.$  
Moreover, it is easy to see that 
for each $K\in \Min(G,\CCI )$ and each $z\in K$, 
$\overline{G(z)}=K.$ In particular, if $K_{1},K_{2}\in \Min(G,\CCI )$ with $K_{1}\neq K_{2}$, then 
$K_{1}\cap K_{2}=\emptyset .$ Moreover, 
by the formula $\overline{G(z)}=K$, we obtain that 
for each $K\in \Min(G,\CCI )$, either (1) $\sharp K<\infty $ or (2) $K$ is perfect and 
$\sharp K>\aleph _{0}.$ 
Furthermore, it is easy to see that 
if $\Gamma \in \Cpt(\Rat), G=\langle \Gamma \rangle $, and 
$K\in \Min(G,\CCI )$, then $K=\bigcup _{h\in \Gamma }h(K).$   
\end{rem}
\begin{rem}
\label{r:minsetcorr}
In \cite[Remark 3.9]{Splms10}, 
for the statement ``for each $K\in \Min (G,Y)$, either 
(1) $\sharp K<\infty $ or (2) $K$ is perfect'', 
we should assume that each element $g\in G$ is a finite-to-one map. 
\end{rem}
\begin{df}
\label{d:stau}
For each $\tau \in {\frak M}_{1,c}(\Rat)$, we set 
$S_{\tau }:=\bigcup _{L\in \Min (G_{\tau },\CCI )}L.$
\end{df}
In \cite{Splms10}, the following result was proved by the author of this paper. 
\begin{thm}[\cite{Splms10}, Cooperation Principle II: Disappearance of Chaos]
\label{t:mtauspec}
Let $\tau \in {\frak M}_{1,c}(\emRat)$. 
Suppose that  
$J_{\ker }(G_{\tau })=\emptyset $ and $J(G_{\tau })\neq \emptyset $.
Then, all of the following statements hold. 
\begin{enumerate}
\item \label{t:mtauspec2}
Let ${\cal B}_{0,\tau }:= \{ \varphi \in C(\CCI )\mid M_{\tau }^{n}(\varphi )\rightarrow 0 \mbox{ as }n\rightarrow \infty \} .$ 
Then, ${\cal B}_{0,\tau }$ is a closed subspace of $C(\CCI )$, 
$\mbox{{\em LS}}({\cal U}_{f,\tau }(\CCI ))\neq \emptyset $ 
 and there exists a direct sum decomposition 
$C(\CCI )=\mbox{{\em LS}}({\cal U}_{f,\tau }(\CCI ))\oplus {\cal B}_{0,\tau }.$ 
Moreover, $\mbox{{\em LS}}({\cal U}_{f,\tau }(\CCI ))\subset C_{F(G_{\tau })}(\CCI )$ and 
$1\leq \dim _{\CC }(\mbox{{\em LS}}({\cal U}_{f,\tau }(\CCI )))<\infty .$ 
\item \label{t:mtauspec3}
$\sharp \emMin(G_{\tau },\CCI )<\infty .$
\item \label{t:mtauspec4}
Let $W:= \bigcup _{A\in \text{{\em Con}}(F(G_{\tau })), A\cap S_{\tau }\neq \emptyset }A$. 
Then $S_{\tau }$ is compact. Moreover, for each $z\in \CCI $ there exists a Borel measurable subset 
${\cal C}_{z}$ of $(\emRat)^{\NN }$ with $\tilde{\tau }({\cal C}_{z})=1$ such that 
for each $\gamma \in {\cal C}_{z}$, there exists an $n\in \NN $ with 
$\gamma _{n,1}(z)\in W$ and   
$d(\gamma _{m,1}(z),S_{\tau })\rightarrow 0$ as $m\rightarrow \infty .$ 
\end{enumerate} 
\end{thm} 
\begin{df}
\label{d:pitau}
Under the assumptions of Theorem~\ref{t:mtauspec}, 
we denote by 
$\pi _{\tau }: C(\CCI )\rightarrow \mbox{LS}({\cal U}_{f,\tau }(\CCI ))$ the 
projection determined by the direct sum decomposition 
$C(\CCI )=\mbox{LS}({\cal U}_{f,\tau }(\CCI ))\oplus {\cal B}_{0,\tau }.$  
\end{df}
\begin{rem}
\label{r:pitau}
Under the assumptions of Theorem~\ref{t:mtauspec}, by the theorem, 
we have that 
$\| M_{\tau }^{n}(\varphi -\pi _{\tau }(\varphi ))\| _{\infty }\rightarrow 0 $ as $n\rightarrow \infty $, 
for each $\varphi \in C(\CCI ).$ 

\end{rem}
\section{Results}
\label{Results} 
In this section, we present the main results of this paper. 
\subsection{Stability and bifurcation}
\label{Stability}
In this subsection, we present some results on stability and bifurcation 
of $M_{\tau }$ or $M_{\tau }^{\ast }.$ The proofs of the results are given in subsection~\ref{Proofs of Stability}. 
\begin{df}
\label{d:m1cnewt}
Let $(X,d)$ be a metric space. 
Let ${\cal O}$ be the topology of ${\frak M}_{1,c}(X)$ 
such that 
$\mu _{n}\rightarrow \mu $ in $({\frak M}_{1,c}(X),{\cal O})$ 
as $n\rightarrow \infty $ if and only if 
(1) $\int \varphi d\mu _{n}\rightarrow \int \varphi d\mu $ for each 
bounded continuous function $\varphi :X \rightarrow \CC $, 
and (2) $\mbox{supp}\,\mu _{n}\rightarrow \mbox{supp}\,\mu $ 
with respect to the Hausdorff metric in the space $\Cpt(X).$   
\end{df}

\begin{df}
\label{d:as}
Let $\G \in \Cpt (\Rat ).$ 
Let $G=\langle \Gamma \rangle .$ 
We say that $G$ is {\bf mean stable} if 
there exist non-empty open subsets $U,V$ of $F(G)$ and a number $n\in \NN $ 
such that all of the following hold.
\begin{itemize}
\item[(1)]
$\overline{V}\subset U$ and $\overline{U}\subset F(G).$ 
\item[(2)]
For each $\gamma \in \Gamma ^{\NN }$, 
$\gamma _{n,1}(\overline{U})\subset V.$ 
\item[(3)]
For each point $z\in \CCI $, there exists an element 
$g\in G$ such that  
$g(z)\in U.$ 
\end{itemize} 
Note that this definition does not depend on the choice of a compact set $\Gamma $ which generates 
$G.$  Moreover, for a $\Gamma \in \Cpt(\Rat)$, 
we say that $\Gamma $ is mean stable if $\langle \Gamma \rangle $ is mean stable. 
Furthermore, for a $\tau \in {\frak M}_{1,c}(\Rat)$, we say that $\tau $ is mean stable if $G_{\tau }$ is mean stable. 
\end{df}
\begin{rem}
\label{r:msjke} 
If $G$ is mean stable, then $J_{\ker }(G)=\emptyset .$ 
\end{rem}
\begin{df}
\label{d:Lattracting}
Let $\G \in \Cpt(\Rat)$ and let $G=\langle \G \rangle .$ 
We say that $L\in \Min (G,\CCI )$ is attracting (for $(G,\CCI )$) 
if there exist non-empty open subsets $U,V$ of $F(G)$ and a number $n\in \NN $ such that 
both of the following hold.
\begin{itemize}
\item[(1)]
$L\subset V\subset \overline{V}\subset U\subset \overline{U}\subset F(G)$, $\sharp (\CCI \setminus V)\geq 3.$ 
\item[(2)] 
For each $\g \in \G ^{\NN }$, $\gamma _{n,1}(\overline{U})\subset V.$ 
\end{itemize} 
\end{df}
\begin{rem}
\label{r:attractinghyp}
If $L$ is attracting for $G=\langle \G \rangle$, then the set $U$ coming from Definition~\ref{d:Lattracting} satisfies 
$\sharp (\CCI \setminus U)\geq 3.$ Therefore for each connected component of $U$, we can take the hyperbolic metric. 
Thus \cite[Theorem 2.11]{Mi} implies that 
there exist an $n\in \NN $ and a constant $0<c<1$ such that for each $\gamma \in \GN $ and for each connected component $W$ of $U$, 
the map $\g _{n,1}:W\rightarrow W'$, where $W'$ denotes the connected component of $U$ with $\g _{n,1}(W)\subset W'$, 
satisfies $d_{h}(\g _{n,1}(z),\g _{n,1}(w))\leq cd_{h}(z,w)$ for each $z,w\in U$, where $d_{h}$ denotes the hyperbolic distance.
\end{rem}
\begin{rem}
\label{r:Lattractingrem}
For each $h\in G$,  
$$\sharp \{ \mbox{attracting minimal set for }(G,\CCI )\} \leq \sharp \{ \mbox{attracting cycles of }h\} <\infty .$$   
For, suppose $L$ is an attracting minimal set for $(G,\CCI ).$ 
Then for each $h\in G$, we have $h(L)\subset L\subset F(G).$ 
Since $L$ is attracting, from \cite[Theorem 2.11]{Mi} it follows that 
for each $z\in L$, $h^{n}(z)$ tends to an attracting cycle of $h.$ Thus 
$L$ contains an attracting cycle of $h.$   
\end{rem}
\begin{rem}
\label{r:msallatt}
Let $\G \in \Cpt(\Rat ).$ Let $G=\langle \G \rangle .$ Suppose that  
$\sharp J(G)\geq 3.$ 
Then \cite[Theorem 3.15, Remark 3.61, Proposition 3.65]{Splms10} imply 
that 
$\G $ is mean stable if and only if 
$\sharp (\Min(G,\CCI ))<\infty $ and 
each $L\in \Min(G,\CCI )$ is attracting for $(G,\CCI ).$ 
Combining this with Remark~\ref{r:Lattractingrem}, it follows that 
$\G $ is mean stable if and only if each 
$L\in \Min(G,\CCI )$ is attracting for $(G,\CCI ).$ 
\end{rem} 
We now give a classification of minimal sets. 
\begin{lem}
\label{l:Lclassify}
Let $\G \in \emCpt(\emRatp)$ and let $G=\langle \G \rangle .$ 
Let $L\in \emMin (G,\CCI ).$ 
Then exactly one of the following holds. 
\begin{itemize}
\item[{\em (1)}] 
$L$ is attracting. 
\item[{\em (2)}]
$L\cap J(G)\neq \emptyset .$ 
Moreover, for each $z\in L\cap J(G)$, there exists an element 
$g\in \G $ with $g(z)\in L\cap J(G).$ 
\item[{\em (3)}]  
$L\subset F(G)$ and 
there exists an element $g\in G$ and an element $U\in \mbox{{\em Con}}(F(G))$ with 
$L\cap U\neq \emptyset $ such that 
$g(U)\subset U$ and $U$ is a subset of a Siegel disk or a Hermann ring of $g.$ 
\end{itemize}
\end{lem}
\begin{df}
\label{d:Lclassify}
Let $\G \in \Cpt(\Ratp)$ and let $G=\langle \G \rangle .$ 
Let $L\in \Min (G,\CCI ).$ 
\begin{itemize}
\item 
We say that $L$ is $J$-touching (for $(G,\CCI )$) if $L\cap J(G)\neq \emptyset .$ 
\item We say that $L$ is sub-rotative (for $(G,\CCI )$) if (3) in Lemma~\ref{l:Lclassify} holds. 
\end{itemize}

\end{df}
\begin{df}
\label{d:bifele}
Let $\G \in \Cpt (\Ratp)$ and let $L\in \Min(\langle \G \rangle ,\CCI )$. 
Suppose $L$ is $J$-touching or sub-rotative. Moreover, 
suppose $L\neq \CCI .$ 
Let $g\in \G $. We say that 
$g$ is a bifurcation element for $(\G ,L)$ if one of the following statements (1)(2) holds.
\begin{itemize}
\item[(1)]
$L$ is $J$-touching and there exists a point $z\in L\cap J(\langle \G \rangle )$ such that 
$g(z)\in J(\langle \G \rangle ).$ 
\item[(2)]
There exist an open subset $U$ of $\CCI $ with $U\cap L\neq \emptyset $ and 
finitely many elements  $\g _{1},\ldots ,\g _{n-1}\in \G $ such that 
$g\circ \g _{n-1}\cdots \circ \g _{1}(U)\subset U$ and $U$ is a subset of 
a Siegel disk or a Hermann ring of $g\circ \g _{n-1}\cdots \circ \g _{1}.$ 
\end{itemize}
Furthermore, we say that an element $g\in \G $ is a bifurcation element for $\G $ if 
there exists an $L\in \Min (\langle \G \rangle ,\CCI )$ such that 
 $g$ is a bifurcation element for $(\G , L).$  
\end{df}
We now consider families of rational maps. 
\begin{df}
Let $\Lambda $ be a finite dimensional complex manifold and 
let $\{ g_{\lambda }\} _{\lambda \in \Lambda }$ be a 
family of rational maps on $\CCI .$ We say that 
$\{ g_{\lambda }\} _{\lambda \in \Lambda }$ is 
a holomorphic family of rational maps if 
the map $(z,\lambda )\in \CCI \times \Lambda \mapsto g_{\lambda }(z)\in \CCI $ is 
holomorphic on $\CCI \times \Lambda .$ We say that 
 $\{ g_{\lambda }\} _{\lambda \in \Lambda }$ is 
a holomorphic family of polynomials if $\{ g_{\lambda }\} _{\lambda \in \Lambda }$ is 
a holomorphic family of rational maps and each $g_{\lambda }$ is a polynomial.   
\end{df} 
\begin{df}
\label{d:uadm}
Let ${\cal Y}$ be a subset of $\Rat$ and let $U$ be a non-empty open subset of $\CCI .$ 
We say that ${\cal Y}$ is strongly $U$-admissible if 
for each $(z_{0},h_{0})\in U\times {\cal Y}$, there exists a holomorphic family 
$\{ g_{\lambda }\} _{\lambda \in \Lambda }$ 
 of 
rational maps with $\bigcup _{\lambda \in \Lambda }\{ g_{\lambda }\} \subset {\cal Y}$ 
and an element $\lambda _{0}\in \Lambda $ such that 
$g_{\lambda _{0}}=h_{0}$ and $\lambda \mapsto g_{\lambda }(z_{0})$ is non-constant in any neighborhood of 
$\lambda _{0}.$ 

\end{df}
\begin{ex}
\label{e:adm}
$\Ratp$ is strongly $\CCI $-admissible. ${\cal P}$ is strongly $\CC $-admissible. 
Let $f_{0}\in {\cal P}$. Then $\{ f_{0}+c\mid c\in \CC \} $ is strongly $\CC $-admissible. 
\end{ex}

\begin{df}
\label{d:condstar}
Let ${\cal Y}$ be a subset of Rat. We say that ${\cal Y}$ satisfies condition $(\ast )$ if 
${\cal Y}$ is a closed subset of Rat and 
at least one of the following (1) and (2) holds. 
(1): ${\cal Y}$ is strongly $\CCI$-admissible. (2) ${\cal Y}\subset {\cal P}$ and 
${\cal Y}$ is strongly $\CC $-admissible.  

\end{df}
\begin{ex}
\label{ex:starsets}
The sets Rat, $\Ratp $ and ${\cal P}$ satisfy $(\ast ).$ 
For an $h_{0}\in {\cal P}$, the set $\{ h_{0}+c\mid c\in \CC \} $ is a subset of 
${\cal P}$ and satisfies $(\ast ).$ 

\end{ex}
We now present a result on bifurcation elements. 
\begin{lem}
\label{l:bebg}
Let ${\cal Y} $ be a subset of $\emRatp $ satisfying condition $(\ast ).$ 
Let $\G \in \emCpt ({\cal Y})$ and let $L\in \emMin(\langle \G \rangle ,\CCI )$. 
Suppose that $L$ is $J$-touching or sub-rotative. Moreover, 
suppose $L\neq \CCI .$ 
Then, there exists a bifurcation element for $(\G ,L).$ 
Moreover, each bifurcation element $g\in \G $ for $(\G ,L)$ belongs to 
$\partial \G $, 
where the boundary $\partial \G $ of $\G $ is taken 
in the topological space ${\cal Y}.$    

\end{lem}
We now present several results on the density of mean stable systems. 
\begin{thm}
\label{t:msminr}
Let ${\cal Y}$ be a subset of $\emRatp$ satisfying condition $(\ast ).$ 
Let $\G \in \emCpt({\cal Y}).$ Suppose that there exists an attracting 
$L\in \emMin(\langle \G \rangle ,\CCI ).$ Let 
$\{ L_{j}\} _{j=1}^{r}$ be the set of attracting minimal sets for 
$(\langle \G \rangle ,  \CCI )$ such that $L_{i}\neq L_{j}$ if $i\neq j$ (Remark: by Remark~\ref{r:Lattractingrem}, the set of attracting minimal sets is finite). 
Let ${\cal U}$ be a neighborhood of $\G $ in $\emCpt({\cal Y})$. 
For each $j=1,\ldots ,r$, let ${\cal V}_{j}$ be a neighborhood of $L_{j}$ with respect to 
the Hausdorff metric in $\emCpt(\CCI )$. Suppose 
that ${\cal V}_{i}\cap {\cal V}_{j}=\emptyset $ for each $(i,j)$ with $i\neq j.$ 
Then, there exists an open neighborhood ${\cal U}'$ of $\G $ in ${\cal U}$ 
such that for any element $\G '\in {\cal U}'$ satisfying that 
$\G \subset \mbox{{\em int}}(\G ')$ with respect to the topology in ${\cal Y}$, 
both of the following statements hold.
\begin{itemize}
\item[{\em (1)}]
$\langle \G '\rangle $ is mean stable and 
$\sharp \emMin(\langle \G '\rangle ,\CCI )=\sharp \{ L'\in \emMin(\langle \G '\rangle ,\CCI )\mid L' \mbox{ is attracting 
for } (\langle \G '\rangle ,\CCI )\} = r.$ 

\item[{\em (2)}]
For each $j=1,\ldots ,r$, there exists a unique element $L_{j}'\in \emMin(\langle \G '\rangle ,\CCI )$ with 
$L_{j}'\in {\cal V}_{j}.$ Moreover, $L_{j}'$ is attracting for $(\langle \G '\rangle ,\CCI )$ for each 
$j=1,\ldots ,r.$     

\end{itemize}
  
\end{thm}
\begin{rem}
\label{r:fsg}
Theorem~\ref{t:msminr} (with \cite[Theorem 3.15]{Splms10}) generalizes \cite[Theorem 0.1]{FS}. 
\end{rem}
\begin{thm}
\label{t:msminrme}
Let ${\cal Y}$ be a subset of $\emRatp$ satisfying condition $(\ast ).$ 
Let $\tau \in {\frak M}_{1,c}({\cal Y}).$ 
Suppose that there exists an attracting 
$L\in \emMin(G_{\tau },\CCI ).$ Let 
$\{ L_{j}\} _{j=1}^{r}$ be the set of attracting minimal sets for 
$(G_{\tau }, \CCI )$ such that $L_{i}\neq L_{j}$ if $i\neq j$.  
Let ${\cal U}$ be a neighborhood of $\tau $ in $({\frak M}_{1,c}({\cal Y}),{\cal O})$. 
For each $j=1,\ldots ,r$, let ${\cal V}_{j}$ be a neighborhood of $L_{j}$ with respect to 
the Hausdorff metric in $\emCpt(\CCI )$. Suppose that 
${\cal V}_{i}\cap {\cal V}_{j}=\emptyset $ for each $(i,j)$ with $i\neq j.$ 
Then, there exists an element $\rho \in {\cal U}$ with $\sharp \G _{\rho }<\infty $ such that 
all of the following hold.
\begin{itemize}
\item[{\em (1)}] 
$G_{\rho } $ is mean stable and 
$\sharp \emMin(G_{\rho } ,\CCI )=
\sharp \{ L'\in \emMin(G_{\rho } ,\CCI )\mid L' \mbox{ is attracting for }(\G _{\rho },\CCI )\} =r.$

\item[{\em (2)}]
For each $j=1,\ldots ,r$, there exists a unique element $L_{j}'\in \emMin(G_{\rho } ,\CCI )$ with 
$L_{j}'\in {\cal V}_{j}$. Moreover, $L_{j}'$ is attracting for $(G_{\rho } ,\CCI )$ for each 
$j=1,\ldots ,r.$

\end{itemize}

\end{thm}
\begin{thm}[Cooperation Principle IV: Density of Mean Stable Systems]
\label{t:pmsod} 
Let ${\cal Y}$ be a subset of ${\cal P}$ satisfying condition $(\ast ).$ 
Then, we have the following. 
\begin{itemize}
\item[{\em (1)}]
The set $\{ \tau \in {\frak M}_{1,c}({\cal Y})\mid \tau \mbox{ is mean stable}\} $ 
is open and dense in $({\frak M}_{1,c}({\cal Y}),{\cal O}).$  
Moreover, 
the set 
$\{ \tau \in {\frak M}_{1,c}({\cal Y})\mid J_{\ker }(G_{\tau })=\emptyset , J(G_{\tau })\neq \emptyset \} $ 
contains 
$\{ \tau \in {\frak M}_{1,c}({\cal Y})\mid \tau \mbox{ is mean stable}\} $.  
\item[{\em (2)}] 
The set 
$\{ \tau \in {\frak M}_{1,c}({\cal Y})\mid \tau \mbox{ is mean stable},\ \sharp \G _{\tau }<\infty \} $ 
is dense in 
$({\frak M}_{1,c}({\cal Y}), {\cal O}).$ 
\end{itemize}
\end{thm}
\begin{thm}
\label{t:noattmin}
Let ${\cal Y}$ be a subset of $\emRatp$ satisfying condition $(\ast ).$ 
Let $\G \in \emCpt({\cal Y}).$ Suppose that there exists no attracting 
minimal set for $(\langle \G \rangle ,\CCI ).$ Then we have the following.
\begin{itemize}
\item[{\em (1)}]
For any element $\G '\in \emCpt(\emRat)$ such that $\G \subset \mbox{{\em int}}(\G')$ with respect to 
the topology in ${\cal Y}$, we have that 
$\emMin(\langle \G '\rangle ,\CCI )=\{ \CCI \} $ 
and $J(\langle \G '\rangle )=\CCI .$ 
\item[{\em (2)}] 
For any neighborhood ${\cal U}$  of 
$\G $ in $\emCpt({\cal Y})$, there exists an element 
$\G '\in {\cal U}$ with $\G '\supset \G $ such that 
$\emMin(\langle \G '\rangle ,\CCI )=\{ \CCI \} $ 
and $J(\langle \G '\rangle )=\CCI .$ 
\end{itemize} 
\end{thm}
\begin{cor}
\label{c:noattminme}
Let ${\cal Y}$ be a subset of $\emRatp$ satisfying condition $(\ast ).$ 
Let $\tau \in {\frak M}_{1,c}({\cal Y})$. 
Suppose that there exists no attracting minimal set for $(G_{\tau },\CCI ).$ 
Let ${\cal U}$ be a neighborhood of $\tau $ in $({\frak M}_{1,c}({\cal Y}),{\cal O}).$ 
Then, there exists an element $\rho \in {\cal U}$ such that 
$\emMin(G_{\rho },\CCI )=\{ \CCI \} $ and $J(G_{\rho })=\CCI .$  
\end{cor}
\begin{cor}
\label{c:msminfull}
Let ${\cal Y}$ be a subset of $\emRatp$ satisfying condition $(\ast ).$ Then, 
the set 
$$\{ \tau \in {\frak M}_{1,c}({\cal Y})\mid \tau \mbox{ is mean stable }\} 
\cup \{ \rho \in {\frak M}_{1,c}({\cal Y})\mid \emMin(G_{\rho },\CCI )=\{ \CCI \} ,J(G_{\rho })=\CCI \} $$  
is dense in $({\frak M}_{1,c}({\cal Y}),{\cal O}).$ 
\end{cor}
We now present a result on the stability of mean stable systems. 
\begin{thm}[Cooperation Principle V: ${\cal O}$-stability of mean stable systems]
\label{t:msmtaust}
Let $\tau \in {\frak M}_{1,c}(\emRat)$ be mean stable. 
Suppose $J(G_{\tau })\neq \emptyset .$ 
Then there exists a neighborhood $\Omega $ of $\tau $ in 
$({\frak M}_{1,c}(\emRat),{\cal O})$ such that 
all of the following statements hold. 
\begin{enumerate}
\item \label{t:msmtaust1}
For each $\nu \in \Omega $, $\nu $ is mean stable, $\sharp (J(G_{\nu }))\geq 3$, 
and $\sharp (\emMin(G_{\nu },\CCI ))=\sharp (\emMin(G_{\tau },\CCI )).$
\item \label{t:msmtaust1-1} 
For each $L\in \emMin(G_{\tau },\CCI )$, 
there exists a continuous map $\nu \mapsto Q_{L,\nu }
\in \emCpt(\CCI )$ on $\Omega $ 
with respect to the Hausdorff metric   
such that $Q_{L,\tau }=L.$ 
Moreover, for each $\nu \in \Omega $, 
$\{ Q_{L,\nu }\} _{L\in \emMin(G_{\tau },\CCI )}=\emMin (G_{\nu },\CCI ).$ 
Moreover, for each $\nu \in \Omega $ and for each 
$L,L'\in \emMin (G_{\tau },\CCI )$ with $L\neq L'$, 
we have $Q_{L,\nu }\cap Q_{L',\nu }=\emptyset .$ 
\item \label{t:msmtaust1-2} 
For each $L\in \emMin (G_{\tau },\CCI )$ and $\nu \in \Omega $, let 
$r_{L}:= \dim _{\CC }(\mbox{{\em LS}}({\cal U}_{f,\tau }(L)))$,  
$\Lambda _{r_{L},\nu }:= \{ h_{r_{L}}\circ \cdots \circ h_{1}\mid h_{j}\in \G _{\nu } (\forall j)\} $, and 
$G_{\nu }^{r_{L}}:= \langle \Lambda _{r_{L},\nu }\rangle $. 
Let $\{ L_{j}\} _{j=1}^{r_{L}}=\emMin(G_{\tau }^{r_{L}},L)$ {\em (Remark:} 
by {\em \cite[Theorem 3.15-12]{Splms10}}, we have $r_{L}=\sharp \emMin (G_{\tau }^{r_{L}},L)${\em )}.    
Then, for each $L\in \emMin (G_{\tau },\CCI )$ and for each $j=1,\ldots ,r_{L}$, 
there exists a continuous map $\nu \mapsto L_{j,\nu }\in \emCpt(\CCI )$ with respect to the Hausdorff metric 
such that, for each $\nu \in \Omega $, 
$\{ L_{j,\nu }\} _{j=1}^{r_{L}}=\emMin (G_{\nu }^{r_{L}},Q_{L,\nu })$ and 
$L_{i,\nu }\neq L_{j,\nu }$ whenever $i\neq j.$    
Moreover, for each $L\in \emMin (G_{\tau },\CCI )$, for each $j=1,\ldots ,r_{L}$, and 
for each $\nu \in \Omega $, we have  
$L_{j+1,\nu }=\bigcup _{h\in \G _{\nu }}h(L_{j,\nu })$, where 
$L_{r_{L}+1,\nu }:=L_{1,\nu }.$ 
\item \label{t:msmtaust2}
For each $\nu \in \Omega $, $\dim _{\CC }(\mbox{{\em LS}}({\cal U}_{f,\nu }(\CCI )))=\dim _{\CC }(\emLSfc) 
=\sum _{L\in \emMin(G_{\tau },\CCI )}r_{L}$. For each 
$\nu \in \Omega $ and for each $L\in \emMin (G_{\tau },\CCI )$, we have  
$\dim _{\CC }(\mbox{{\em LS}}({\cal U}_{f,\nu }(Q_{L,\nu })))=r_{L}$, 
${\cal U}_{v,\nu }(Q_{L,\nu })=\{ a_{L}^{i}\} _{i=1}^{r_{L}}$, and  
${\cal U}_{v,\nu }(\CCI )=\bigcup _{L\in \emMin (G_{\tau },\CCI )}\{ a_{L}^{i}\} _{i=1}^{r_{L}}$, 
where $a_{L}:= \exp(2\pi i/r_{L}).$  
\item \label{t:msmtaust3}
The maps $\nu \mapsto \pi _{\nu }$ and $\nu \mapsto \mbox{{\em LS}}({\cal U}_{f,\nu }(\CCI ))$ are continuous on $\Omega $. 
More precisely, for each $\nu \in \Omega $, there exists a finite family 
$\{ \varphi _{L,i,\nu }\mid L\in \emMin (G_{\tau },\CCI ), i=1,\ldots ,r_{L}\} $ 
in $ {\cal U}_{f,\nu }(\CCI )$  
and a finite family $\{ \rho _{L,i,\nu }\mid L\in \emMin (G_{\tau },\CCI ),i=1,\ldots ,r_{L}\} $ in $C(\CCI )^{\ast }$ such that 
all of the following hold.
\begin{itemize}
\item[{\em (a)}] 
$\{ \varphi _{L,i,\nu }\mid L\in \emMin (G_{\tau },\CCI ), i=1,\ldots ,r_{L}\}$ is a basis of $\mbox{{\em LS}}({\cal U}_{f,\nu }(\CCI ))$ and  
$\{ \rho _{L,i,\nu }\mid L\in \emMin (G_{\tau },\CCI ),i=1,\ldots ,r_{L}\} $ is a basis of $\mbox{{\em LS}}({\cal U}_{f,\nu,\ast }(\CCI )).$ 
\item[{\em (b)}] 
Let  $L\in \emMin (G_{\tau },\CCI )$ and let $i=1,\ldots ,r_{L}.$ 
Let $\nu \in \Omega .$ 
Then $M_{\tau }(\varphi _{L,i,\nu })=a_{L}^{i}\varphi _{L,i,\nu }$, 
$\varphi _{L,i,\nu }|_{Q_{L,\nu }}=(\varphi _{L,1,\nu }|_{Q_{L,\nu }})^{i}$, 
$\varphi _{L,i,\nu }|_{Q_{L',\nu }}\equiv 0 $ for any $L'\in \emMin (G_{\nu },\CCI )$ with $L'\neq L$, 
and {\em supp}$\, \rho _{L,i,\nu }=Q_{L,\nu }.$ 
Moreover, 
$\{ \varphi _{L,i,\nu }|_{Q_{L,\nu }} \} _{i=1}^{r_{L}}$ is a basis of $\mbox{{\em LS}}({\cal U}_{f,\nu }(Q_{L,\nu }))$ and  
$\{ \rho _{L,i,\nu }|_{C(Q_{L,\nu })}\mid i=1,\ldots ,r_{L}\} $ is a basis of 
$\mbox{{\em LS}}({\cal U}_{f,\nu ,\ast }(Q_{L,\nu })).$ In particular, \\  
$\dim _{\CC }(\mbox{{\em LS}}({\cal U}_{f,\nu }(Q_{L,\nu })))=r_{L}.$ 
\item[{\em (c)}] 
For each $L\in \emMin (G_{\tau },\CCI )$ and for each $i=1,\ldots r_{L}$, 
 the map $\nu \mapsto \varphi _{L,i,\nu }\in C(\CCI )$ is continuous on $\Omega $ and  
the map $\nu \mapsto \rho _{j,\nu }\in C(\CCI )^{\ast }$ is continuous on $\Omega .$ 
\item[{\em (d)}] 
For each $L\in \emMin (G_{\tau },\CCI )$, for each $(i,j)$ and 
for each $\nu \in \Omega $, $\rho _{L,i,\nu }(\varphi _{L,j,\nu })=\delta _{ij}.$ 
Moreover, For each $L, L'\in \emMin (G_{\tau },\CCI )$ with $L\neq L'$, for each $(i,j)$, and for each 
$\nu \in \Omega $,  
$\rho _{L,i,\nu }(\varphi _{L',j,\nu })=0.$ 
\item[{\em (e)}] 
For each $\nu \in \Omega $ and for each $\varphi \in C(\CCI )$, 
$\pi _{\nu }(\varphi )=\sum _{L\in \emMin (G_{\tau },\CCI )}\sum _{i=1}^{r_{L}}\rho _{L,i,\nu }(\varphi )\cdot \varphi _{L,i,\nu }.$  
\end{itemize}
\item \label{t:msmtaust4}  
For each $L\in \emMin (G_{\tau },\CCI )$, 
the map $\nu \mapsto T_{Q_{L,\nu },\nu }\in (C(\CCI ),\| \cdot \| _{\infty })$ 
is continuous on $\Omega .$   
\end{enumerate}
\end{thm}
We now present a result on a characterization of mean stability. 
\begin{thm}
\label{t:ABCD}
Let ${\cal Y}$ be a subset of $\emRatp$ satisfying condition $(\ast ).$ 
We consider the following subsets $A,B,C,D,E$ of ${\frak M}_{1,c}({\cal Y})$ which are defined as follows.
\begin{itemize}
\item[{\em (1)}]
$A:= \{ \tau  \in {\frak M}_{1,c}({\cal Y})\mid \tau \mbox{ is mean stable} \} .$ 
\item[{\em (2)}] 
Let $B$ be the set of $\tau \in {\frak M}_{1,c}({\cal Y})$ satisfying that  
there exists a neighborhood $\Omega $ of $\tau $ in $({\frak M}_{1,c}({\cal Y}),{\cal O})$ 
such that {\em (a)} for each $\nu \in \Omega $, $J_{\ker }(G_{\nu })=\emptyset $, and {\em (b)}   
$\nu \mapsto \sharp \emMin(G_{\nu },\CCI )$ is constant on $\Omega $. 
\item[{\em (3)}]
Let $C$ be the set of $\tau \in {\frak M}_{1,c}({\cal Y})$ satisfying that  
there exists a neighborhood $\Omega $ of $\tau $ in $({\frak M}_{1,c}({\cal Y}),{\cal O})$ 
such that {\em (a)} for each $\nu \in \Omega $, $F(G_{\nu })\neq \emptyset $, and {\em (b)}   
$\nu \mapsto \sharp \emMin(G_{\nu },\CCI )$ is constant on $\Omega $. 
\item[{\em (4)}] 
Let $D$ be the set of $\tau \in {\frak M}_{1,c}({\cal Y})$ satisfying that  
there exists a neighborhood $\Omega $ of $\tau $ in $({\frak M}_{1,c}({\cal Y}),{\cal O})$ 
such that for each $\nu \in \Omega $, $J_{\ker }(G_{\nu })=\emptyset $ and 
$\dim _{\CC }(\mbox{{\em LS}}({\cal U}_{f,\nu }(\CCI )))=\dim _{\CC }(\mbox{{\em LS}}({\cal U}_{f,\tau }(\CCI )))$. 
\item[{\em (5)}] 
Let $E$ be the set of $\tau \in {\frak M}_{1,c}({\cal Y})$ satisfying that 
for each $\varphi \in C(\CCI )$, there exists a neighborhood 
$\Omega $ of $\tau $ in $({\frak M}_{1,c}({\cal Y}),{\cal O})$ 
such that {\em (a)} for each $\nu \in \Omega $, $J_{\ker }(G_{\nu })=\emptyset $, and 
{\em (b)} the map $\nu \mapsto \pi _{\nu }(\varphi )\in (C(\CCI ),\| \cdot \| _{\infty })$ defined on $\Omega $ 
is continuous at $\tau .$ 
\end{itemize}
Then, $A=B=C=D=E.$ 

\end{thm}
We now present a result on bifurcation of dynamics of $G_{\tau }$ and $M_{\tau }$ regarding a 
continuous family of measures $\tau .$  
\begin{thm}
\label{t:bifur}
Let ${\cal Y}$ be a subset of $\emRatp $ satisfying condition $(\ast ).$  
For each $t\in [0,1]$, let 
$\mu _{t}$ be an element of ${\frak M}_{1,c}({\cal Y}).$ 
Suppose that all of the following conditions {\em (1)--(4)} hold.
\begin{itemize}
\item[{\em (1)}] 
$t\mapsto \mu _{t}\in ({\frak M}_{1,c}({\cal Y}),{\cal O})$ is continuous on 
$[0,1].$ 
\item[{\em (2)}] If $t_{1},t_{2}\in [0,1]$ and $t_{1}<t_{2}$, then 
$\G _{\mu _{t_{1}}}\subset \mbox{{\em int}}(\G _{\mu _{t_{2}}})$ with respect to 
the topology of ${\cal Y}.$ 
\item[{\em (3)}] 
$\mbox{{\em int}}(\G _{\mu _{0}})\neq \emptyset $ with respect to the topology of ${\cal Y}$ and 
$F(G _{\mu _{1}}) \neq \emptyset $. 
\item[{\em (4)}] 
$\sharp (\emMin (G _{\mu _{0}}, \CCI ))\neq 
\sharp (\emMin (G _{\mu _{1}}, \CCI )).$ 
\end{itemize}
Let $B:=\{ t\in [0,1)\mid \mbox{there exists a bifurcation element } 
g\in \G _{\mu _{t}} \mbox{ for } \G _{\mu _{t}}\}.$ 
Then, we have the following.
\begin{itemize}
\item[{\em (a)}]
For each $t\in [0,1]$, $J_{\ker }(G _{\mu _{t}} )=\emptyset $ and 
$\sharp J(G _{\mu _{t}} )\geq 3$, and 
all statements in {\em \cite[Theorem 3.15]{Splms10}} (with $\tau =\mu _{t}$) hold. 
\item[{\em (b)}] We have    
$$1\leq \sharp B \leq \sharp (\emMin(G _{\mu _{0}}, \CCI ))-\sharp (\emMin(G _{\mu _{1}}, \CCI )) <\infty .$$
Moreover, for each $t\in B$, $\mu _{t}$ is not mean stable. 
Furthermore, for each $t\in [0,1)\setminus B$, $\mu _{t}$ is mean stable. 
\item[{\em (c)}] 
For each $s\in (0,1]$ there exists a number $t_{s}\in (0,s)$ such that 
for each $t\in [t_{s},s]$, 
$\sharp (\emMin (G_{\mu _{t}},\CCI ))=\sharp (\emMin (G_{\mu _{s}},\CCI )).$
\end{itemize}  
\end{thm}
\begin{ex}
\label{e:bifur}
Let $c$ be a point in the interior of the Mandelbrot set ${\cal M}.$ Suppose $z\mapsto z^{2}+c$ is hyperbolic.  
Let $r_{0}>0$ be a small number such that 
$\overline{D(c,r_{0})}\subset \mbox{int}({\cal M}).$ Let $r_{1}>0$ be a large number such that 
$D(c,r_{1})\cap (\CC \setminus {\cal M})\neq \emptyset .$ 
For each $t\in [0,1]$, let 
$\mu _{t}\in {\frak M}_{1}(\overline{D(c,(1-t)r_{0}+tr_{1})})$ be the normalized 
$2$-dimensional Lebesgue measure on  $\overline{D(c,(1-t)r_{0}+tr_{1})}$. 
Then $\{ \mu _{t}\} _{t\in [0,1]}$ satisfies the conditions (1)--(4) in 
Theorem~\ref{t:bifur} 
(for example, $2=\sharp (\Min(G _{\mu _{0}}, \CCI ))>
\sharp (\Min (G _{\mu _{1}}, \CCI ))=1$). 
Thus 
 $$\sharp (\{ t\in [0,1]\mid \mbox{there exists a bifurcation element } 
g\in \G _{\mu _{t}} \mbox{ for } \G _{\mu _{t}}\} )=1.$$ 

\end{ex}
\subsection{Spectral properties of $M_{\tau }$ and stability}
\label{Spectral} 
In this subsection, we present some results on 
spectral properties of $M_{\tau }$ acting on the space of H\"{o}lder continuous functions on $\CCI $ 
and the stability. The proofs of the results are given in 
subsection~\ref{Proofs of Spectral}.
\begin{df}
Let $K\in \Cpt (\CCI ).$ 
For each $\alpha \in (0,1)$,  
let \\ $C^{\alpha }(K ):= \{ \varphi \in C(K)\mid \sup _{x,y\in K, x\neq y} |\varphi (x)-\varphi (y)|/d(x,y)^{\alpha} <\infty \} $ 
 be the Banach space of all complex-valued $\alpha $-H\"{o}lder continuous functions on $K $
endowed with the $\alpha $-H\"{o}lder norm $\| \cdot \| _{\alpha },$ 
where $\| \varphi \| _{\alpha }:= \sup _{z\in K}| \varphi (z)| +\sup _{x,y\in K,x\neq y}|\varphi (x)-\varphi (y)|/d(x,y)^{\alpha }$ 
for each $\varphi \in C^{\alpha }(K).$ 
\end{df}

\begin{thm}
\label{t:utauca}
Let $\tau \in {\frak M}_{1,c}(\emRat)$. Suppose that 
$J_{\ker }(G_{\tau })=\emptyset $ and $J(G_{\tau })\neq \emptyset .$ 
Then, there exists an $\alpha _{0}>0$ such that for each $\alpha \in (0,\alpha _{0})$, 
$\mbox{{\em LS}}({\cal U}_{f,\tau }(\CCI ))\subset C^{\alpha }(\CCI ).$ 
Moreover, for each $\alpha \in (0,\alpha _{0})$, 
there exists a constant $E_{\alpha }>0$ such that 
for each $\varphi \in C^{\alpha }(\CCI )$, 
$\| \pi _{\tau }(\varphi )\| _{\alpha }\leq E_{\alpha }\| \varphi \| _{\infty }.$  
Furthermore, for each $\alpha \in (0,\alpha _{0})$ and for each $L\in \emMin (G_{\tau },\CCI )$, 
$T_{L,\tau }\in C^{\alpha }(\CCI ).$ 
\end{thm}
\ 

If $\tau \in {\frak M}_{1,c}(\Rat )$ is mean stable and $J(G_{\tau })\neq \emptyset $, 
then by \cite[Proposition 3.65]{Splms10}, we have 
$S_{\tau }\subset F(G_{\tau })$ (see Definition~\ref{d:stau}).  
From this point of view, we consider the situation that 
$\tau \in {\frak M}_{1,c}(\Rat )$ satisfies $J_{\ker }(G_{\tau })=\emptyset $, 
$J(G_{\tau })\neq \emptyset $, and 
$S_{\tau }\subset F(G_{\tau }).$ 
Under this situation, we have several very strong results. 
Note that there exists an example of $\tau \in {\frak M}_{1,c}({\cal P})$ with $\sharp \G _{\tau }<\infty $ 
such that $J_{\ker }(G_{\tau })=\emptyset $, $J(G_{\tau })\neq \emptyset $, $S_{\tau }\subset F(G_{\tau })$, 
and $\tau $ is not mean stable (see Example~\ref{ex:stfnms}). 
\begin{thm}[Cooperation Principle VI: Exponential rate of convergence] 
\label{t:kjemfhf}
Let $\tau \in {\frak M}_{1,c}(\emRat)$. Suppose that $J_{\ker }(G_{\tau })=\emptyset $, 
$J(G_{\tau })\neq \emptyset $, and $S_{\tau }\subset F(G_{\tau }).$ 
Let $r:= \prod _{L\in \emMin(G_{\tau },\CCI )}\dim _{\CC }(\mbox{{\em LS}}({\cal U}_{f,\tau }(L)))$.  
Then, there exists a constant $\alpha \in (0,1)$, a constant $\lambda \in (0,1)$, and a constant $C>0$ such that 
for each $\varphi \in C^{\alpha }(\CCI )$, 
we have all of the following.
\begin{itemize}
\item[{\em (1)}] 
$\| M_{\tau }^{nr}(\varphi )-\pi _{\tau }(\varphi )\| _{\alpha }\leq 
C\lambda ^{n}\| \varphi -\pi _{\tau }(\varphi )\| _{\alpha }$ for each $n\in \NN .$ 
\item[{\em (2)}]
$\| M_{\tau }^{n}(\varphi -\pi _{\tau }(\varphi ))\| _{\alpha }\leq C\lambda ^{n}\| \varphi -\pi _{\tau }(\varphi )\| _{\alpha }$
 for each $n\in \NN .$ 
\item[{\em (3)}] 
$\| M_{\tau }^{n}(\varphi -\pi _{\tau }(\varphi ))\| _{\alpha }\leq C\lambda ^{n}\| \varphi \| _{\alpha }$
 for each $n\in \NN .$ 
\item[{\em (4)}] 
$\| \pi _{\tau }(\varphi )\| _{\alpha }\leq C\| \varphi \| _{\alpha }.$ 
\end{itemize}
 
\end{thm}
We now consider the spectrum Spec$_{\alpha }(M_{\tau })$ of $M_{\tau }:C^{\alpha }(\CCI )\rightarrow C^{\alpha }(\CCI ).$ 
By Theorem~\ref{t:utauca}, 
${\cal U}_{v,\tau }(\CCI )\subset \mbox{Spec}_{\alpha }(M_{\tau })$ for some $\alpha \in (0,1).$ 
From Theorem~\ref{t:kjemfhf}, we can show that 
the distance between ${\cal U}_{v,\tau }(\CCI )$ and 
$\mbox{Spec}_{\alpha }(M_{\tau })\setminus {\cal U}_{v,\tau }(\CCI )$ is positive. 
\begin{thm}
\label{t:kjemfsp}
Under the assumptions of Theorem~\ref{t:kjemfhf}, 
we have all of the following.
\begin{itemize}
\item[{\em (1)}]
$\mbox{{\em Spec}}_{\alpha }(M_{\tau })\subset \{ z\in \CC \mid |z|\leq \lambda \} 
\cup {\cal U}_{v,\tau }(\CCI )$, 
where $\alpha \in (0,1)$ and $\lambda \in (0,1)$ are the constants in Theorem~\ref{t:kjemfhf}. 
\item[{\em (2)}] 
Let $\zeta \in \CC \setminus (\{ z\in \CC \mid |z|\leq \lambda \} 
\cup {\cal U}_{v,\tau }(\CCI ))$. 
Then, $(\zeta I-M_{\tau })^{-1}:C^{\alpha }(\CCI )\rightarrow C^{\alpha }(\CCI )$ 
is equal to 
$$(\zeta I-M_{\tau })|_{\emLSfc}^{-1}\circ \pi _{\tau }+\sum _{n=0}^{\infty }
\frac{M_{\tau }^{n}}{\zeta ^{n+1}}(I-\pi _{\tau }),$$ 
where $I$ denotes the identity on $C^{\alpha }(\CCI ).$ 

\end{itemize}

\end{thm}
Combining Theorem~\ref{t:kjemfsp} and perturbation theory for linear operators (\cite{K}), 
we obtain the following. In particular, as we remarked in Remark~\ref{r:takagi}, 
we obtain complex analogues of the Takagi  function.  
\begin{thm}
\label{t:kjemfsppt}
Let $m\in \NN $ with $m\geq 2.$
Let $h_{1},\ldots ,h_{m}\in \emRat$. 
Let $G=\langle h_{1},\ldots ,h_{m}\rangle .$ 
Suppose that 
$J_{\ker }(G)=\emptyset, J(G)\neq \emptyset $ and 
$\bigcup _{L\in \emMin(G,\CCI )}L\subset F(G).$ 
Let ${\cal W}_{m}:= \{ (a_{1},\ldots ,a_{m})\in (0,1)^{m}\mid \sum _{j=1}^{m}a_{j}=1 \} 
\cong \{ (a_{1},\ldots ,a_{m-1})\in (0,1)^{m-1}\mid \sum _{j=1}^{m-1}a_{j}<1 \}.$ 
For each $a=(a_{1},\ldots ,a_{m})\in {\cal W}_{m}$, 
let $\tau _{a}:= \sum _{j=1}^{m}a_{j}\delta _{h_{j}}\in {\frak M}_{1,c}(\emRat).$ 
Then we have all of the following. 
\begin{itemize}
\item[{\em (1)}] 
For each $b\in {\cal W}_{m}$,  
there exists an $\alpha \in (0,1)$ and an open neighborhood $V_{b}$ of $b$ in ${\cal W}_{m}$ 
such that for each $a\in V_{b}$, we have 
$\mbox{{\em LS}}({\cal U}_{f,\tau _{a}}(\CCI ))\subset C^{\alpha }(\CCI )$, 
$\pi _{\tau _{a}}(C^{\alpha }(\CCI ))\subset C^{\alpha }(\CCI )$ and 
$(\pi _{\tau _{a}}:C^{\alpha }(\CCI )\rightarrow C^{\alpha }(\CCI ))\in L(C^{\alpha }(\CCI ))$, 
where $L(C^{\alpha }(\CCI))$ denotes the Banach space of bounded linear operators on $C^{\alpha }(\CCI )$ 
endowed with the operator norm,  and such that the map 
$a\mapsto (\pi _{\tau _{a}}:C^{\alpha }(\CCI )\rightarrow C^{\alpha }(\CCI ))\in 
L(C^{\alpha }(\CCI ))$ 
 is real-analytic in $V_{b}.$ 
\item[{\em (2)}] 
Let $L\in \emMin( G,\CCI ).$ 
Then, for each $b\in {\cal W}_{m}$, there exists an $\alpha \in (0,1)$ such that 
the map $a\mapsto T_{L,\tau _{a}}\in (C^{\alpha }(\CCI ),\| \cdot \| _{\alpha })$ is real-analytic 
in an open neighborhood of $b$ in ${\cal W}_{m}.$ Moreover, 
the map $a\mapsto T_{L,\tau _{a}}\in (C(\CCI ),\| \cdot \| _{\infty })$ is real-analytic in 
${\cal W}_{m}.$ In particular, for each 
$z\in \CCI $, the map $a\mapsto T_{L,\tau _{a}}(z)$ is real-analytic in ${\cal W}_{m}.$ 
Furthermore, for any multi-index $n=(n_{1},\ldots ,n_{m-1})\in (\NN \cup \{ 0\})^{m-1}$ and for any 
$b\in {\cal W}_{m},$  
 the function 
 $z\mapsto [(\frac{\partial }{\partial a_{1}})^{n_{1}}\cdots (\frac{\partial }{\partial a_{m-1}})^{n_{m-1}}
(T_{L,\tau _{a}}(z))]|_{a=b}$ belongs to $C_{F(G)}(\CCI ).$ 
\item[{\em (3)}] 
Let $L\in \emMin(G,\CCI )$ and let $b\in {\cal W}_{m}.$ 
For each $i=1,\ldots ,m-1$ and for each $z\in \CCI $, let 
$\psi _{i,b}(z):=[\frac{\partial }{\partial a_{i}}(T_{L,\tau _{a}}(z))]|_{a=b}$  
and let $\zeta _{i,b}(z):= T_{L,\tau _{b}}(h_{i}(z))-T_{L,\tau _{b}}(h_{m}(z)).$  
Then,  
$\psi _{i,b}$ is the unique solution of 
the functional equation $(I-M_{\tau _{b}})(\psi )=\zeta _{i,b}, \psi |_{S_{\tau _{b}}}=0, \psi \in C(\CCI )$,  
where $I$ denotes the identity map. Moreover, 
there exists a number $\alpha \in (0,1)$ such that  
$\psi _{i,b}=\sum _{n=0}^{\infty }M_{\tau _{b}}^{n}(\zeta _{i,b})$ in $(C^{\alpha }(\CCI ),\| \cdot \| _{\alpha }).$  
\end{itemize}  
\end{thm}

We now present a result on the non-differentiability of $\psi _{i,b}$ at points in 
$J(G_{\tau })$. In order to do that, we need several definitions and notations. 
\begin{df}
For a rational semigroup $G$, 
we set $P(G):= \overline{\bigcup _{g\in G}\{ \mbox{ all critical values of }
 g: \CCI \rightarrow \CCI }\} $ where the closure is taken in $\CCI .$ This is called the postcritical set of $G$. 
 We say that a rational semigroup $G$ is hyperbolic if $P(G)\subset F(G).$ 
 For a polynomial semigroup $G$, 
 we set $P^{\ast }(G):= P(G)\setminus \{ \infty \} .$ 
For a polynomial semigroup $G$, 
we set $\hat{K}(G):=\{ z\in \CC \mid G(z)\mbox{ is bounded in } \CC \} .$ 
Moreover, for each polynomial $h$, we set 
$K(h):=\hat{K}(\langle h\rangle ).$  
\end{df}
\begin{rem}
\label{r:hypms}
Let $\G \in \Cpt(\Ratp)$ and suppose that $\langle \G \rangle $ is hyperbolic and 
$J_{\ker }(\langle \G \rangle )=\emptyset .$ 
Then by \cite[Propositions 3.63, 3.65]{Splms10}, there exists an neighborhood 
${\cal U}$ of $\G $ in $\Cpt(\Rat)$ such that for each $\G '\in {\cal U}$, 
$\G ' $ is mean stable, $J_{\ker }(\langle \G '\rangle )=\emptyset $, 
$J(\langle \G '\rangle )\neq \emptyset $ and $\bigcup _{L\in \Min (\langle \G '\rangle ,\CCI )}L\subset F(\langle \G '\rangle ).$  
\end{rem}
\begin{df}
\label{d:dfu}
Let $m\in \NN .$ 
Let $h=(h_{1},\ldots, h_{m})\in (\Rat)^{m}$ be an element such that 
$h_{1},\ldots ,h_{m}$ are 
mutually distinct. 
We set 
$\Gamma := \{ h_{1},\ldots ,h_{m}\} .$ 
Let $f:\Gamma ^{\NN }\times 
\CCI \rightarrow \Gamma ^{\NN }\times \CCI $ be the map 
defined by $f(\g ,y)=(\sigma (\g ), \g _{1}(y))$, 
where $\g =(\g _{1},\g_{2},\ldots )\in \G ^{\NN }$ and 
$\sigma :\G ^{\NN }\rightarrow \G ^{\NN }$ 
is the shift map ($(\g _{1},\g _{2},\ldots )\mapsto (\g _{2},\g _{3},\ldots )$).  
This map $f:\G ^{\NN }\times \CCI \rightarrow \G ^{\NN }\times \CCI $ 
is called the skew product associated with $\Gamma .$  
Let $\pi :\GN \times \CCI \rightarrow \GN $ and $\pi _{\CCI }:\GN \times \CCI \rightarrow \CCI $ be
 the canonical projections. 
Let $\mu \in {\frak M}_{1}(\Gamma ^{\NN }\times \CCI )$ be an $f$-invariant Borel probability measure. 
Let ${\cal W}_{m}:= \{ (a_{1},\ldots ,a_{m})\in (0,1)^{m}\mid \sum _{j=1}^{m}a_{j}=1\}.$ 
For each $p=(p_{1},\ldots ,p_{m})\in {\cal W}_{m}$, 
we define a function $\tilde{p}:\Gamma ^{\NN }\times \CCI \rightarrow \RR $ by 
$\tilde{p}(\gamma ,y):=p_{j}$  if $\gamma _{1}=h_{j} $ 
(where $\gamma =(\gamma _{1},\gamma _{2},\ldots )$), and  we set 
$$u(h,p,\mu ):= 
\frac{-(\int _{\Gamma ^{\NN }\times \CCI }\log \tilde{p}(\g, y)\ d\mu (\g, y) )}
{\int _{\Gamma ^{\NN }\times \CCI }\log \| D(\g _{1})_{y}\| _{s} \ d\mu (\g ,y)}$$ 
(when the integral of the denominator converges), 
where $\| D(\g _{1})_{y} \| _{s}$ denotes the norm of the derivative of $\g _{1}$ at $y$ with respect to the 
spherical metric on $\CCI .$  
\end{df}
\begin{df} \label{d:green} 
Let $h=(h_{1},\ldots , h_{m})\in {\cal P}^{m}$ be an element such that 
$h_{1},\ldots ,h_{m}$ are 
mutually distinct. 
We set 
$\Gamma := \{ h_{1},\ldots ,h_{m}\} .$ 
For any $(\gamma ,y)\in \Gamma ^{\NN }\times \CC $, 
let $G_{\gamma }(y):= \lim _{n\rightarrow \infty }\frac{1}{\deg (\gamma _{n,1})}
\log ^{+}|\gamma _{n,1}(y)|$, 
where $\log ^{+}a:=\max\{\log a,0\} $ for each $a>0.$  
By the arguments in \cite{Se}, for each $\gamma \in \Gamma ^{\NN }$, 
$G_{\gamma }(y) $ exists, 
$G_{\gamma }$ is subharmonic on $\CC $, and 
$G_{\gamma }|_{A_{\infty ,\gamma }}$ is equal to the Green's function on 
$A_{\infty ,\gamma }$ with pole at $\infty $, 
where $A_{\infty ,\g }:= \{ z\in \CCI \mid \g _{n,1}(z)\rightarrow \infty \mbox{ as } n\rightarrow 
\infty \} .$ 
Moreover, $(\gamma ,y)\mapsto G_{\gamma }(y)$ is continuous on $\Gamma ^{\NN }\times \CC .$ 
Let $\mu _{\gamma }:=dd^{c}G_{\gamma }$, where $d^{c}:=\frac{i}{2\pi }(\overline{\partial }-\partial ).$ 
Note that by the argument in \cite{J1,Se}, 
$\mu _{\gamma }$ is a Borel probability measure on $J_{\gamma }$ such that 
$\mbox{supp}\, \mu _{\gamma }=J_{\gamma }.$ 
Furthermore, for each $\gamma \in \Gamma ^{\NN }$, 
let $\Omega (\gamma )=\sum _{c} G_{\gamma }(c)$, where $c$ runs over all critical points of 
$\gamma _{1}$ in $\CC $, counting multiplicities.   
\end{df}
\begin{rem}
\label{r:maxrelent}
Let $h=(h_{1},\ldots ,h_{m})\in (\Ratp)^{m}$ be an element such that 
$h_{1},\ldots ,h_{m}$ are mutually distinct. 
Let $\Gamma =\{ h_{1},\ldots ,h_{m}\} $ and 
let $f:\Gamma ^{\NN }\times \CCI \rightarrow \Gamma ^{\NN }\times \CCI $ 
be the skew product map associated with $\Gamma .$ 
Moreover, let $p=(p_{1},\ldots ,p_{m})\in {\cal W}_{m}$ and 
let $\tau =\sum _{j=1}^{m}p_{j}\delta _{h_{j}}\in {\frak M}_{1}(\Gamma ).$ 
Then, there exists a unique $f$-invariant ergodic Borel probability measure 
$\mu $ on $\Gamma ^{\NN }\times \CCI $ such that $\pi _{\ast }(\mu )=\tilde{\tau }$ and   
$h_{\mu }(f|\sigma )=\max _{\rho \in {\frak E}_{1}(\Gamma ^{\NN }\times \CCI ): 
f_{\ast }(\rho )=\rho, \pi _{\ast }(\rho )=\tilde{\tau} }  h_{\rho }(f|\sigma )=\sum _{j=1}^{m}p_{j}\log (\deg (h_{j}))$, 
where $h_{\rho }(f|\sigma )$ denotes the relative metric entropy 
of $(f,\rho )$ with respect to $(\sigma, \tilde{\tau })$, and 
${\frak E}_{1}(\cdot )$ denotes the space of ergodic measures (see \cite{S3}).  
This $\mu $ is called the {\bf maximal relative entropy measure} for $f$ with respect to 
$(\sigma ,\tilde{\tau }).$   
\end{rem}
\begin{df}
\label{d:phe}
Let $V$ be a non-empty open subset of $\CCI .$ Let $\varphi :V \rightarrow \CC $ be a function and 
let $y\in V $ be a point. Suppose that $\varphi $ is bounded around $y.$ 
Then we set  
$$\mbox{H\"{o}l}(\varphi ,y):= 
\sup \{ \beta \in [0,\infty ) \mid \limsup _{z\rightarrow y,z\neq y}
\frac{|\varphi (z)-\varphi (y)|}{d(z,y)^{\beta }}<\infty \} \in [0,\infty ], $$
where $d$ denotes the spherical distance. 
This is called the {\bf pointwise H\"{o}lder exponent of $\varphi $ at $y.$} 
\end{df}
\begin{rem} 
If $\mbox{H\"{o}l}(\varphi, y)<\infty $, then 
$\mbox{H\"{o}l}(\varphi, y)=\inf \{ \beta \in [0,\infty )\mid 
\limsup _{z\rightarrow y, z\neq y}\frac{|\varphi (z)-\varphi (y)|}{d(z,y)^{\beta }}=\infty \}. $ 
If $\mbox{H\"{o}l}(\varphi ,y)<1$, then 
$\varphi $ is non-differentiable at $y.$ 
If 
 $\mbox{H\"{o}l}(\varphi ,y)>1$, then 
$\varphi $ is differentiable at $y$ and the derivative at $y$ is equal to $0.$
In \cite[Definition 3.80]{Splms10}, ``$\limsup _{z\rightarrow y}$'' should be replaced by 
``$\limsup _{z\rightarrow y,z\neq y}$'' and we should add the following. 
``If $\{ \beta \in [0,\infty )\mid 
\limsup _{z\rightarrow y, z\neq y}\frac{|\varphi (z)-\varphi (y)|}{d(z,y)^{\beta }}=\infty \}=\emptyset $, 
then we set $\mbox{H\"{o}l}(\varphi ,y)=\infty .$''  
\end{rem}
We now present a result on the non-differentiability of 
$\psi _{i,b}(z)=[\frac{\partial }{\partial a_{i}}(T_{L,\tau _{a}}(z))]|_{a=b}$ at points 
in $J(G_{\tau }).$ 
\begin{thm}
\label{t:psinondiff}
Let $m\in \NN $ with $m\geq 2.$ 
Let $h=(h_{1},\ldots ,h_{m})\in (\emRatp)^{m}$ and we set 
$\Gamma := \{ h_{1},h_{2},\ldots ,h_{m}\} .$ 
Let $G=\langle h_{1},\ldots ,h_{m}\rangle .$
Let ${\cal W}_{m}:= \{ (a_{1},\ldots ,a_{m})\in (0,1)^{m}\mid \sum _{j=1}^{m}a_{j}=1 \} 
\cong \{ (a_{1},\ldots ,a_{m-1})\in (0,1)^{m-1}\mid \sum _{j=1}^{m-1}a_{j}<1 \}.$ 
For each $a=(a_{1},\ldots ,a_{m})\in {\cal W}_{m}$, 
let $\tau _{a}:= \sum _{j=1}^{m}a_{j}\delta _{h_{j}}\in {\frak M}_{1,c}(\emRat).$ 
Let $p=(p_{1},\ldots ,p_{m})\in {\cal W}_{m}.$  
Let $f:\Gamma ^{\NN }\times \CCI \rightarrow \Gamma ^{\NN }\times \CCI $ be the  
skew product associated with $\Gamma .$  
Let 
$\tau := \sum _{j=1}^{m}p_{j}\delta _{h_{j}}\in {\frak M}_{1}(\Gamma )
\subset {\frak M}_{1}({\cal P }).$
Let 
$\mu \in {\frak M}_{1}(\Gamma ^{\NN }\times \CCI )$ be the maximal relative entropy measure 
 for $f:\Gamma ^{\NN }\times \CCI \rightarrow \Gamma ^{\NN }\times \CCI $ with respect to 
 $(\sigma ,\tilde{\tau }).$ 
Moreover, let 
$\lambda := (\pi _{\CCI })_{\ast }(\mu )\in {\frak M}_{1}(\CCI ).$ 
Suppose that  
$G$ is hyperbolic, and 
$h_{i}^{-1}(J(G))\cap h_{j}^{-1}(J(G))=\emptyset $ for each 
$(i,j)$ with $i\neq j$. 
For each $L\in \emMin(G,\CCI )$,   
for each $i=1,\ldots ,m-1$ and for each $z\in \CCI $, let 
$\psi _{i,p,L}(z):=[\frac{\partial }{\partial a_{i}}(T_{L,\tau _{a}}(z))]|_{a=p}.$  
Then, we have all of the following. 
\begin{enumerate}
\item \label{t:psinondiff1}
$G_{\tau }=G$ is mean stable, $J_{\ker }(G)=\emptyset $, and $S_{\tau }\subset F(G_{\tau }).$  
Moreover, $0<\dim _{H}(J(G))<2 $, {\em supp} $\lambda =J(G)$, 
and $\lambda (\{ z\} )=0$ for each $z\in J(G)$.  
\item \label{t:psinondiff2}
Suppose $\sharp \emMin (G,\CCI )\neq 1.$ 
Then there exists a Borel subset $A$ of $J(G)$ with $\lambda (A)=1$ such that 
for each $z_{0}\in A$,  for each $L\in \emMin(G,\CCI )$ and for each 
$i=1,\ldots, m-1$, exactly one of the following {\em (a),(b),(c)} holds.  
\begin{itemize}
\item[{\em (a)}] $\mbox{{\em H\"{o}l}}(\psi _{i,p,L},z_{1})=\mbox{{\em H\"{o}l}}(\psi _{i,p,L},z_{0})< u(h,p,\mu )$ 
for each $z_{1}\in h_{i}^{-1}(\{ z_{0}\} )\cup h_{m}^{-1}(\{ z_{0}\} ).$  
\item[{\em (b)}] 
$\mbox{{\em H\"{o}l}}(\psi _{i,p,L},z_{0})= u(h,p,\mu )\leq \mbox{{\em H\"{o}l}}(\psi _{i,p,L},z_{1})$ for 
each $z_{1}\in  h_{i}^{-1}(\{ z_{0}\} )\cup h_{m}^{-1}(\{ z_{0}\} ).$
\item[{\em (c)}] 
$\mbox{{\em H\"{o}l}}(\psi _{i,p,L},z_{1})=u(h,p,\mu )<\mbox{{\em H\"{o}l}}(\psi _{i,p,L},z_{0})$ 
for 
each $z_{1}\in  h_{i}^{-1}(\{ z_{0}\} )\cup h_{m}^{-1}(\{ z_{0}\} ).$ 
\end{itemize} 
\item \label{t:psinondiff3}
If $h=(h_{1},\ldots ,h_{m})\in {\cal P}^{m}$, then 
$$
u(h,p,\mu )=\frac{-(\sum _{j=1}^{m}p_{j}\log p_{j})}
{\sum _{j=1}^{m}p_{j}\log \deg (h_{j})+\int _{\Gamma ^{\NN }}\Omega (\gamma )\ d\tilde{\tau }(\gamma )}
$$
and 
$$2 >  \dim _{H}(\lambda )
=  \frac{\sum _{j=1}^{m}p_{j}\log \deg (h_{j})-\sum _{j=1}^{m}p_{j}\log p_{j}}
{\sum _{j=1}^{m}p_{j}\log \deg (h_{j})+\int _{\Gamma ^{\NN }}\Omega (\gamma )\ d\tilde{\tau }(\gamma )}>0, 
$$
where $\dim _{H}(\lambda ):= \inf \{ \dim _{H}(A)\mid A \mbox{ is a Borel subset of }\CCI , \lambda (A)=1\} .$ 
\item \label{t:psinondiff4} 
Suppose $h=(h_{1},\ldots ,h_{m})\in {\cal P}^{m}.$ 
Moreover, suppose that at least one of the following {\em (a)}, {\em (b)}, and {\em (c)} holds:
{\em (a)} $\sum _{j=1}^{m}p_{j}\log (p_{j}\deg (h_{j}))>0.$
{\em (b)} $P^{\ast }(G)$ is bounded in $\CC .$ 
{\em (c)} $m=2.$ 
Then, $u(h,p,\mu )<1.$ 
\end{enumerate}
\end{thm} 
\section{Tools}
\label{Tools} 
In this section, we introduce some fundamental tools to prove the main results. 

Let $G$ be a rational semigroup. Then, for each $g\in G$, $g(F(G))\subset F(G), g^{-1}(J(G))\subset J(G).$ 
If $G$ is generated by a compact family $\Lambda $ of Rat, then 
$J(G)=\bigcup _{h\in \Lambda }h^{-1}(J(G))$ (this is called the backward self-similarity). 
If $\sharp J(G)\geq 3$, then $J(G)$ is a perfect set and 
$J(G)$ is equal to the closure of the set of repelling cycles of elements of $G$. 
In particular, $J(G)=\overline{\cup _{g\in G}J(g)}$ if $\sharp J(G)\geq 3.$ 
We set $E(G):=\{ z\in \CCI \mid \sharp \bigcup _{g\in G}g^{-1}(\{ z\} )<\infty \} .$ 
If $\sharp J(G)\geq 3$, then $\sharp E(G)\leq 2$ and 
for each $z\in J(G)\setminus E(G)$, $J(G)=\overline{\bigcup _{g\in G}g^{-1}(\{ z\} ) }.$ 
If $\sharp J(G)\geq 3$, then $J(G)$ is the smallest set in 
$\{ \emptyset \neq K\subset \CCI \mid K \mbox{ is compact}, \forall g\in G, g(K)\subset K\} $ 
with respect to the inclusion. For more details on these  
properties of rational semigroups, 
see \cite{HM,GR, S3}. 

 For fundamental tools and lemmas  of random complex dynamics, see \cite{Splms10}.  
\section{Proofs}
\label{Proofs} 
In this section, we give the proofs of the main results. 

\subsection{Proofs of results in \ref{Stability}}
\label{Proofs of Stability} 
In this subsection, we give the proofs of the results in subsection~\ref{Stability}. 
We need several lemmas. 

\begin{df}
\label{d:opensethyp}
Let $W$ be an open subset of $\CCI $ with $\sharp (\CCI \setminus W)\geq 3$ 
and let $g:W\rightarrow W$ be a holomorphic map. 
Let $\{ W_{j}\} _{j\in J}=\mbox{Con}(W).$  
For each connected component $W_{j}$ of $W$, we take the hyperbolic metric $\rho _{j}.$  
For each $z\in W$, we denote by $\| Dg_{z}\| _{h}$ the norm of the derivative 
of $g$ at $z$ which is measured from the hyperbolic metric on the component $W_{i_{1}}$ of $W$ containing $z$ 
to that on the component $W_{i_{2}}$ of $W$ containing $g(z).$ 
Moreover, for each subset $L$ of $W$ and for each $r\geq 0$,  
we set $d_{h}(L,r):= \bigcup _{j\in J} \{z\in W_{j}\mid d_{\rho _{j}}(z,L\cap W_{j})<r\} $, 
where $d_{\rho _{j}}(z,L\cap W_{j})$ denotes the distance from $z$ to $L\cap W_{j}$ with respect to 
the hyperbolic distance on $W_{j}.$ Similarly, for each $z\in W$, we denote by 
$\| Dg_{z}\| _{s}$ the norm of the derivative of $g$ at $z $ with respect to the spherical metric on $\CCI .$   
\end{df}

\begin{lem}
\label{l:Lattnear}
Let $\G \in \emCpt(\emRat )$ and let $G=\langle \G \rangle .$ 
Let $L\in \emMin(G,\CCI)$ be attracting for $(G,\CCI ).$ 
Let 
$W:= \bigcup _{A\in \mbox{{\em Con}}(F(G)), A\cap L\neq \emptyset }A$ and 
let $W'$ be a relative compact open subset of $W$ including $L.$ 
Then there exists an open neighborhood ${\cal U}$ of $\G $ in $\emCpt(\emRat)$ such that 
both of the following hold.
\begin{itemize}
\item[{\em (1)}] 
For each $\Omega \in {\cal U}$, there exists a unique 
$L'\in \emMin (\langle \Omega ,\CCI )$ with $L'\subset W'.$ 
\item[{\em (2)}] 
For each $\Omega \in {\cal U},$ the above $L'$ is attracting for $(\langle \Omega \rangle ,\CCI ).$ 
\end{itemize} 

\end{lem}
\begin{proof}
Let $\{ W_{j}\} _{j=1}^{s}=\mbox{Con}(W)$.  
For each connected component $W_{j}$, we take the hyperbolic metric $\rho _{j}$.
Let $d_{\rho _{j}}$ be the distance on $W_{j}$ induced by $\rho _{j}.$  
For each $r>0$, we set $d_{h}(L,r):=\bigcup  _{j=1}^{s}\{ z\in W\mid d_{\rho _{j}}(z, L\cap W_{j})<r\} .$ 
Let $(U,V,n)$ be as in Definition~\ref{d:Lattracting}. 
Then $V\cap W\subset \overline{V\cap W}\subset U\cap W\subset \overline{U\cap W}\subset F(G)$ and 
$\gamma _{n,1}(\overline{U\cap W})\subset V\cap W$ for each $\gamma \in \GN .$ 
Therefore, by \cite[Theorem 2.11]{Mi}, there exists a constant $0<c<1$ 
such that for each $\g \in \GN $ and for each $j_{1},j_{2}\in \{ 1,\ldots ,s\},$ 
if $\g _{n,1}(W_{j_{1}})\subset W_{j_{2}}$,  
then 
$d_{\rho _{j_{1}}}(\g _{n,1}(z), \g _{n,1}(w))
\leq cd_{\rho _{j_{2}}}(z,w)$ for each $z,w\in W_{j_{1}}\cap \overline{W'}.$
Thus, replacing $n$ by a larger number if necessary, 
 we may assume that there exists a number $\epsilon _{1}\in (0,1)$ such that 
for each $\gamma \in \GN $, 
$\gamma _{n,1}(\overline{W'})\subset d_{h }(L,\epsilon _{1})\subset \overline{d_{h }(L,\epsilon _{1})}\subset W'.$ 
Hence, there exists an open neighborhood ${\cal U}$ of $\G $ in $\Cpt(\Rat)$ and a number 
$\epsilon _{2}\in (\epsilon _{1},1)$ such that 
for each $\Omega \in {\cal U}$ and for each 
$\gamma \in \Omega ^{\NN }$,  
\begin{equation}
\label{eq:Lattnear1}
\gamma _{n,1}(\overline{W'})\subset d_{h }(L,\epsilon _{2})\subset 
\overline{d_{h }(L,\epsilon _{2})}\subset W'.
\end{equation} 
Let $\Omega \in {\cal U}.$ Setting $\Omega _{n}:= \{ \gamma _{n}\circ \cdots \circ \gamma _{1}\mid \gamma _{j}\in \Omega 
(\forall j)\} $, 
we obtain that there exists an element $L_{0}\in \Min(\langle \Omega _{n}\rangle ,\CCI )$ 
with $L_{0}\subset W'.$ Then, 
for each $g\in \langle \Omega \rangle $, 
$g(\langle \Omega \rangle (L_{0}))\subset \langle \Omega \rangle (L_{0}).$ 
Taking ${\cal U}$ so small, we may assume that 
$\langle \Omega \rangle (L_{0})\subset W'.$ 
Hence, there exists an element $L'\in \Min(\langle \Omega \rangle ,\CCI )$ 
with $L'\subset W'.$ 
From (\ref{eq:Lattnear1}) and \cite[Theorem 2.11]{Mi}, it follows that 
there exists no $L''\in \Min(\langle \Omega \rangle ,\CCI )$ with 
$L''\neq L'$ such that $L''\subset W'.$ 
Moreover, by (\ref{eq:Lattnear1}) and \cite[Theorem 2.11]{Mi} again, we obtain that 
$L'$ is attracting for $(\langle \Omega \rangle ,\CCI ).$ 
Thus, we have proved our lemma.   
\end{proof} 

\begin{lem}
\label{l:Lattdense}
Let $\G \in \emCpt(\emRat) $ and let $G=\langle \G \rangle .$ 
Let $L\in \emMin(G,\CCI )$ be attracting for $(G,\CCI ).$ 
Then $L=\overline{\{ z\in L\mid \exists g\in G\mbox{ s.t. }
g(z)=z, |m(g,z)|<1\} }, $ where 
$m(g,z)$ denotes the multiplier of $g$ at $z.$   

\end{lem}
\begin{proof}
Let $z\in L.$ Let $U\in \mbox{Con}(F(G))$ with $z\in U.$ 
Let $B$ be an open neighborhood of $z$ in $U.$ 
Since $L\in \Min(G,\CCI )$ and since $L$ is attracting,  
the argument in the proof of Lemma~\ref{l:Lattnear} implies that 
there exists an element $g\in G$ such that 
$\overline{g(B)}\subset B.$ Then there exists an attracting fixed point of $g$ in $B.$
Thus the statement of our lemma holds. 
\end{proof} 
\begin{lem}
\label{l:Lattnh}
Let $\G \in \emCpt(\emRat )$ and let $G=\langle \G \rangle .$ 
Let $L\in \emMin(G,\CCI)$ be attracting for $(G,\CCI ).$ 
Let ${\cal V}$ be a neighborhood of $L$ in the space $\emCpt(\CCI ).$ 
Then there exists an open neighborhood ${\cal U}$ of $\G $ in $\emCpt(\emRat)$ 
and an open neighborhood ${\cal V}'$ of $L$ in $\emCpt(\CCI)$ with ${\cal V}'\subset {\cal V}$ 
such that 
for each $\Omega \in {\cal U}$, there exists a unique 
$L'\in \emMin (\langle \Omega \rangle ,\CCI )$ with $L'\in {\cal V}'.$ Moreover, 
this $L'$ is attracting for $(\langle \Omega \rangle ,\CCI ).$ 
\end{lem}
\begin{proof}
By Lemma~\ref{l:Lattnear}, Lemma~\ref{l:Lattdense} and Implicit function theorem, 
the statement of our lemma holds. 
\end{proof}

\begin{lem}
\label{l:suffLatt}
Let $\G \in \emCpt(\emRat) $ and let $G=\langle \G \rangle .$ 
Let $L\in \emMin (G,\CCI )$ with $L\subset F(G).$  
Suppose 
that 
for each 
$g\in G$ and for each $U\in \mbox{{\em Con}} (F(G))$ with $U\cap L\neq \emptyset $ and 
$g(U)\subset U$, either {\em (a)} $g\in \emRatp$ and $U$ is not a subset of a Siegel disk or a Hermann ring of $g$ 
or {\em (b)} $g\in \mbox{{\em Aut}}(\CCI )$ and $g$ is loxodromic or parabolic.  
Then, $L$ is attracting for $(G,\CCI )$. 
\end{lem}
\begin{proof}
Let $W:=\bigcup _{A\in \mbox{Con}(F(G)),A\cap L\neq \emptyset }A$ and we take the hyperbolic metric on each connected 
component of $W.$ 
Since $L$ is a compact subset of $F(G)$, we have $\sharp \mbox{Con}(W)<\infty .$ Moreover, 
from assumptions (a) and (b) and \cite[Theorem 2.11]{Mi}, we obtain that 
if $A\in \mbox{Con}(W)$ and if $\gamma _{n,m}(A)\subset A$, then 
$\| D(\gamma _{n,m})_{z}\| _{h }<1$ for each $z\in A$.  
From these arguments, it is easy to see that $L$ is attracting for 
$(G,\CCI ).$ 
\end{proof} 
\noindent {\bf Proof of Lemma~\ref{l:Lclassify}:} 
Lemma~\ref{l:suffLatt} implies that if $L\subset F(G)$ and (3) in Lemma~\ref{l:Lclassify} does 
not hold, then $L$ is attracting. We now suppose that $L\cap J(G)\neq \emptyset .$ 
Let $z\in L\cap J(G).$ 
By $J(G)=\cup _{h\in \G } h^{-1}(J(G))$ (\cite[Lemma 0.2]{S4}), 
there exists an element $g_{0}\in \G $ with $g_{0}(z)\in J(G).$ Then 
$g_{0}(z)\in L\cap J(G).$ Thus we have proved our lemma. 
\qed 
\begin{lem}
\label{l:inceps}
Let $U$ be a non-empty open subset of $\CCI .$ Let 
${\cal Y}$ be a closed subset of $\emRat.$ 
Suppose that ${\cal Y}$ is strongly $U$-admissible. 
Let $\G \in \emCpt({\cal Y}).$ 
Let $h_{0}$ be an interior point of $\G $ with respect to the topology in the space ${\cal Y}.$ 
Let $K\in \emCpt(U).$ Then, there exists an $\epsilon >0$ such that 
for each $z\in K$, 
$\{ h(z)\mid h\in \G \} \supset B(h_{0}(z),\epsilon ).$ 
\end{lem}
\begin{proof}
Let $w\in K$. Then there exists a holomorphic family 
$\{ g_{\lambda }\} _{\lambda \in \Lambda }$ of rational maps 
with $\bigcup _{\lambda \in \Lambda }\{ g_{\lambda }\} \subset {\cal Y}$ 
and a point $\lambda _{0}\in \Lambda $ 
such that $g_{\lambda _{0}}=h_{0}$ and  
$\lambda \mapsto g_{\lambda }(w)$ is non-constant in any neighborhood of $\lambda _{0}.$ 
By the argument principle, there exists a $\delta _{w}>0$, an $\epsilon _{w}>0$ and a 
neighborhood $V_{w}$ of $\lambda _{0}$ such that 
for any $z\in K\cap B(w,\delta _{w})$, the map 
$\Psi _{z}: \lambda \mapsto g_{\lambda }(z)$ satisfies that 
$\Psi _{z}(V_{w})\supset B(h_{0}(z),\epsilon _{w}).$ 
Since $K$ is compact, there exists a finite family 
$\{ B(w_{j},\delta _{w_{j}})\} _{j=1}^{s}$ in $\{ B(w,\delta _{w})\} _{w\in K}$ such that 
$\bigcup _{j=1}^{s}B(w_{j},\delta _{w_{j}})\supset K.$ 
From these arguments, the statement of our lemma holds. 
\end{proof}
We now prove Lemma~\ref{l:bebg}. \\ 
\noindent {\bf Proof of Lemma~\ref{l:bebg}:} 
Let $G=\langle \G \rangle .$ By Lemma~\ref{l:Lclassify}, 
we have a bifurcation element for $(\G ,L).$ 
Let $g\in \G $ be a bifurcation element for $(\G ,L).$  
Suppose we have $g\in \mbox{int}(\G ).$ We consider the following two cases. 
Case (1): $(L,g)$ satisfies condition (1) in Definition~\ref{d:bifele}. 
Case (2): $(L,g)$ satisfies condition (2) in Definition~\ref{d:bifele}.  

 We now consider Case (1). 
Then there exists a point $z\in L\cap J(G)$  such that 
$g(z)\in J(G).$ Let ${\cal U}$ be an open  neighborhood of $g$ in int$(\G )$.  
Let $A:=\{ h(z)\mid h\in {\cal U}\} .$ Then 
$A$ is an open subset of $\CCI $ and  
$A \cap J(G)\neq \emptyset $. 
It follows that $\overline{G(A)}=\CCI .$ 
Since $A\subset L$, we obtain that $L=\CCI .$ However, this contradicts our assumption. 
Therefore, $g$ must belong to $\partial \G .$ 

 We now consider Case (2).  Let $\gamma _{1},\ldots, \gamma _{n-1}\in \G , U$ be as in 
 condition (2) in Definition~\ref{d:bifele}. We set $h=g\circ \gamma _{n-1}\circ \cdots \circ \gamma _{1}.$  
We may assume that $U$ is a Siegel disk or Hermann ring of $h.$   
 Then there exists a biholomorphic map $\zeta :U\rightarrow B$, where $B$ is the unit disk or a round annulus, 
and a $\theta \in \RR \setminus \QQ $,  
such that 
 $r_{\theta }\circ \zeta =\zeta \circ h$ on $U$, where $r_{\theta }(z):= e^{2\pi i\theta }z.$  
Let $z_{0}\in L\cap U$ be a point. 
By Lemma~\ref{l:inceps}, it follows that there exists an open subset $W$ of $\CCI $ such that 
$W\subset G(z_{0})$ and $W\cap \partial U \neq \emptyset .$ 
Therefore $J(G)\cap \mbox{int}(L)\neq \emptyset .$ 
Hence, we obtain $L=\CCI .$ However, this contradicts our assumption. 
 Therefore, $g$ must belong to $\partial \G .$ 
 
 Thus, we have proved Lemma~\ref{l:bebg}.
\qed 

We now prove Theorem~\ref{t:msminr}.\\ 
\noindent {\bf Proof of Theorem~\ref{t:msminr}:}
Let ${\cal U}'$ be a small open neighborhood of $\G $ in ${\cal U}.$ 
Let $\G ' \in {\cal U}'$ be an element such that 
$\G \subset \mbox{int}(\G ') $ with respect to 
the topology in the space ${\cal Y}.$ If ${\cal U}' $ is so small,  
then Lemma~\ref{l:Lattnh} implies that 
for each $j=1,\ldots ,r$, there exists a unique element 
$L'_{j}\in \Min(\langle \G ' \rangle ,\CCI )$ with 
$L'_{j}\in {\cal V}_{j}$, and this $L'_{j}$ is attracting for 
$(\langle \G ' \rangle ,\CCI ).$ 
Taking ${\cal U}'$ so small, the inclusion $\G \subset \G ' $ and  
Remark~\ref{r:minimal} imply that 
for each $j=1,\ldots ,r$,  
$L'_{j}$ is the unique element in $\Min(\langle \G ' \rangle ,\CCI )$ which contains 
$L_{j}$.  

Suppose that there exists an element $L'\in \Min(\langle \G ' \rangle ,\CCI )\setminus \{ L'_{j}\} _{j=1}^{r}$. 
Since $\langle \G \rangle (L')\subset L'$, 
Remark~\ref{r:minimal} implies that there exists a minimal 
set $K\in \Min(\langle \G \rangle ,\CCI )$ such that $K\subset L'.$ 
Since $L_{j}\subset L'_{j}$ for each $j=1,\ldots r$, and 
since $L'\cap \bigcup _{j=1}^{r}L'_{j}=\emptyset $, we obtain 
that $K\not\in \{ L_{j}\} _{j=1}^{r}.$ 
Hence $K$ is not attracting for $(\langle \G \rangle , \CCI ).$ 
Let $g\in \G $ be a bifurcation element for 
$(\G , K).$ 
Then, $g\in \mbox{int}(\G ')$ and $g$ is a bifurcation element for $(\G ' ,L').$ 
However, this contradicts Lemma~\ref{l:bebg}. 
Therefore, $\Min(\langle \G '\rangle ,\CCI )=\{ L'_{j}\} _{j=1}^{r}.$ 
Moreover, from the above arguments and Remark~\ref{r:msallatt},   
it follows that $\G '$ is mean stable and 
$\sharp (\Min(\langle \G '\rangle ,\CCI ))=r.$ 
Thus we have proved Theorem~\ref{t:msminr}.      
\qed 

\begin{lem}
\label{l:msmincons}
Let $\G \in \emCpt(\emRat)$ be mean stable and suppose $J(\langle \G \rangle ))\neq \emptyset .$ 
Then, there exists an open neighborhood ${\cal U}$ of $\G $ in $\emCpt(\emRat)$ with respect to the 
Hausdorff metric such that 
for each $\G '\in {\cal U}$, 
$\langle \G' \rangle $ is mean stable, $\sharp (J(\langle \G '\rangle ))\geq 3$,  and 
$\sharp \emMin(\langle \G \rangle ,\CCI )=\sharp \emMin(\langle \G '\rangle ,\CCI ).$ 
\end{lem}
\begin{proof}
Since $\G $ is mean stable, $J_{\ker }(\langle \G \rangle )=\emptyset .$ 
Combining this with that $J(\langle \G \rangle )\neq \emptyset $ and \cite[Theorem 3.15-3]{Splms10}, 
we obtain $\sharp (J(\langle \G \rangle ))\geq 3.$ 
By \cite[Theorem 3.1]{HM} and \cite[Lemma 2.3(f)]{S3}, 
the repelling cycles of elements of $\langle \G \rangle $ is dense in $J(\langle \G \rangle ).$ 
Combining it with implicit function theorem, we obtain that there exists a neighborhood 
${\cal U}'$ of $\G $ in $\Cpt(\Rat )$ such that 
for each $\G '\in {\cal U}'$, $\sharp (J(\langle \G '\rangle ))\geq 3.$  
 
By \cite[Theorem 3.15-6]{Splms10}, $\sharp (\Min(\langle \G \rangle ,\CCI ))<\infty .$ 
Let $S_{\G }:= \bigcup _{L\in \Min(\langle \G \rangle ,\CCI )}L.$ 
By \cite[Proposition 3.65]{Splms10}, $S_{\G }\subset F(\langle \G \rangle ).$ 
Let $W:= \bigcup _{A\in \mbox{Con}(F(\langle \G \rangle )), A\cap S_{\G }\neq \emptyset }A.$ 
We use the notation in Definition~\ref{d:opensethyp} for this $W.$  
Let $0<\epsilon _{2}<\epsilon _{1}.$ 
Since $\G $ is mean stable, 
there exists an $n\in \NN $ such that 
for each $\g \in \GN $, 
$\gamma _{n,1}(d_{h}(S_{\G },\epsilon _{1}))\subset d_{h}(S_{\G },\epsilon _{2}).$ 
Moreover, for each $z\in \CCI $, there exists a map $g_{z}\in \langle \G \rangle $ such that 
$g_{z}(z)\in d_{h}(S_{\G },\epsilon _{1}).$ 
Therefore, there exist finitely many points $z_{1},\ldots z_{s}$ in $\CCI $ and 
positive numbers $\delta _{1},\ldots ,\delta _{s}$ with $\cup _{j=1}^{s}B(z_{j},\delta _{j})=\CCI $ such that 
for each $j=1,\ldots ,s$, 
$\overline{g_{z_{j}}(B(z_{j},\delta _{j}))}\subset d_{h}(S_{\G },\epsilon _{1}).$ 
Let $\epsilon _{3}\in (\epsilon _{2},\epsilon _{1}).$ 
Let ${\cal U} (\subset {\cal U}')$ be a small neighborhood of $\G $ in $\Cpt(\Rat )$. 
Then for each $\G '\in {\cal U}$ and for each $\g \in \G '^{\NN }$, 
$\gamma _{n,1}(d_{h}(S_{\G },\epsilon _{1}))\subset d_{h}(S_{\G },\epsilon _{3})$. Moreover,   
for each $\G '\in {\cal U}$ and for each $z\in \CCI $, 
there exists a map $g_{z,\G '}\in \langle \G '\rangle $ such that 
$g_{z,\G '}(z)\in d_{h}(S_{\G },\epsilon _{1}). $ 
Hence, for each $\G ' \in {\cal U}$, 
$\G '$ is mean stable and 
$\bigcup _{L'\in \Min(\langle \G '\rangle ,\CCI )}L'\subset d_{h}(S_{\G },\epsilon _{1}).$ 
Combining this with  Lemma~\ref{l:Lattnear}, and shrinking ${\cal U}$ if necessary, 
we obtain that 
for each $\G '\in {\cal U}$, 
$\sharp (\Min(\langle \G '\rangle ,\CCI ))=\sharp (\Min(\langle \G \rangle ,\CCI )).$  
\end{proof}

We now prove Theorem~\ref{t:msminrme}. \\ 
{\bf Proof of Theorem~\ref{t:msminrme}:} 
There exists a sequence $\{ \tau _{n}\} _{n=1}^{\infty }$ in 
${\frak M}_{1,c}({\cal Y})$ with $\sharp \G _{\tau _{n}}<\infty (\forall n)$ such that 
$\tau _{n}\rightarrow \tau $ in $({\frak M}_{1,c}({\cal Y}),{\cal O})$ 
as $n\rightarrow \infty .$ 
Therefore, by Lemma~\ref{l:Lattnh}, we may assume that 
$\sharp \G _{\tau }<\infty .$ 
We write $\tau =\sum _{j=1}^{s}p_{j}\delta _{h_{j}}$, 
where $\sum _{j=1}^{s}p_{j}=1$, $p_{j}>0$ for each $j$, and $h_{j}\in {\cal Y}$ for each 
$j$. By Theorem~\ref{t:msminr}, enlarging the support of $\tau $, 
we obtain an element $\rho '\in {\cal U}$ such that 
statements (1) and (2) in our theorem with $\rho $ being replaced by $\rho '$ hold. 
Let $\rho $ be a finite measure which is close enough to $\rho '.$ 
By Lemma~\ref{l:msmincons} and Lemma~\ref{l:Lattnh}, we obtain that 
this $\rho $ has the desired property. 
Thus we have proved Theorem~\ref{t:msminrme}. 
\qed 

We now prove Theorem~\ref{t:pmsod}. \\
{\bf Proof of Theorem~\ref{t:pmsod}:} 
Let $\tau \in {\cal M}_{1,c}({\cal Y}).$ 
Since $\G _{\tau }$ is compact in ${\cal P}$, 
we obtain that $\{ \infty \} $ is an attracting 
minimal set for $(G_{\tau },\CCI ).$ 
By Theorem~\ref{t:msminrme} and Lemma~\ref{l:msmincons}, 
the statements in our theorem hold. 
\qed 

We now prove Theorem~\ref{t:noattmin}. \\ 
{\bf Proof of Theorem~\ref{t:noattmin}:} 
Let $\G ' \in \Cpt(\Rat)$ be an element such that $\G \subset \mbox{int} (\G ')$ 
with respect to the topology in ${\cal Y}.$ 
We now show the following claim.\\ 
Claim: $\Min (\langle \G '\rangle ,\CCI )=\{ \CCI \} .$ 

 To prove this claim, suppose this is not true. 
Then  $\Min (\langle \G \rangle ,\CCI )\neq \{ \CCI \} .$ 
Since there exists no attracting minimal set for $(\langle \G \rangle ,\CCI )$, 
from Lemma~\ref{l:bebg} it follows that there exists a bifurcation element $g\in \G $ for $\G .$ 
Then $g\in \mbox{int}(\G ')$ and $g$ is a bifurcation element for $\G '.$   However, this contradicts Lemma~\ref{l:bebg}. 
Thus, we have proved the claim. 

Let $h\in \mbox{int}(\G ')$ be an element and 
let $z\in J(\langle \G '\rangle )$ be a point which is not a critical value of $h.$ 
Then we obtain that int$(\langle \G '\rangle ^{-1}(\{ z\}))\neq \emptyset .$ 
Therefore,  $K:=\overline{F(\langle \G '\rangle )}$ is not equal to $\CCI .$ 
By Remark~\ref{r:minimal} and the above claim, it follows that 
$K=\emptyset .$ Thus $J(\langle \G '\rangle )=\CCI .$ 
Hence, we have proved statement (1) in our theorem. 

Statement (2) in our theorem easily follows from statement (1).    
\qed 

We now prove Corollary~\ref{c:noattminme}. \\
{\bf Proof of Corollary~\ref{c:noattminme}:} 
Let $\epsilon >0$ be a small number. 
Let $\{ h_{j,\epsilon }\} _{j=1}^{\infty }$ be a dense countable subset of 
$\overline{B(\G _{\tau },\epsilon )}$ with respect to the topology in ${\cal Y}.$ 
Let $\{ p_{j,\epsilon }\} _{j=1}^{\infty }$ be a sequence of positive numbers such that 
$\sum _{j=1}^{\infty }p_{j,\epsilon }=1.$ Let 
$\tau _{\epsilon }:=(1-\epsilon )\tau +\epsilon \sum _{j=1}^{\infty }p_{j,\epsilon }\delta _{h_{j,\epsilon }}.$   
Then $\G _{\tau }\subset \mbox{int}(\G _{\tau _{\epsilon }})$ and 
$\tau _{\epsilon }\rightarrow \tau $ in $({\frak M}_{1,c}({\cal Y}),{\cal O})$ as $\epsilon \rightarrow 0.$ 
Let $\epsilon >0$ be a small number and let $\rho := \tau _{\epsilon }.$ 
By Theorem~\ref{t:noattmin}, this $\rho $ has the desired property.  
\qed 

We now prove Corollary~\ref{c:msminfull}.\\ 
{\bf Proof of Corollary~\ref{c:msminfull}:} 
Corollary~\ref{c:msminfull} easily follows from Theorem~\ref{t:msminrme} 
and Corollary~\ref{c:noattminme}.  
\qed 

\begin{df}
\label{d:taun}
Let ${\cal Y}$ be a closed subset of $\Rat .$ For each 
$\tau \in {\frak M}_{1}({\cal Y})$ and for each $n\in \NN $, 
let $\tau ^{n}:= \otimes _{j=1}^{n}\tau \in {\frak M}_{1}({\cal Y}^{n}).$ 
\end{df}
The following lemma is easily obtained by some fundamental observations. 
The proof is left to the readers. 
\begin{lem}
\label{l:rhontaun}
If $\rho _{n}\rightarrow \rho $ in $({\frak M}_{1,c}(\emRat ^{m}),{\cal O})$ as 
$n\rightarrow \infty $ and if 
$\tau _{n}\rightarrow \tau $ in $({\frak M}_{1,c}(\emRat ),{\cal O})$ as 
$n\rightarrow \infty $, 
then $\rho _{n}\otimes \tau _{n}\rightarrow \rho \otimes \tau $ in 
$({\frak M}_{1,c}(\emRat ^{m+1}),{\cal O})$ as $n\rightarrow \infty .$ 
In particular, if $\nu _{k}\rightarrow \tau $ in $({\frak M}_{1,c}(\emRat ),{\cal O})$ 
as $k\rightarrow \infty $, then 
$\nu _{k}^{m}\rightarrow \tau ^{m}$ in $({\frak M}_{1,c}(\emRat ^{m}),{\cal O})$ 
as $k\rightarrow \infty $, for each $m\in \NN .$ 
\end{lem}

We now prove Theorem~\ref{t:msmtaust}.\\ 
{\bf Proof of Theorem~\ref{t:msmtaust}:} 
Statement \ref{t:msmtaust1} follows from Lemma~\ref{l:msmincons}. 
We now prove statements \ref{t:msmtaust1-1},\ref{t:msmtaust1-2},\ref{t:msmtaust2}. Let $\Omega $ be a small open neighborhood of $\tau $ in 
$({\frak M}_{1,c}({\cal Y}),{\cal O})$ such that for each $\nu \in \Omega $, 
$\nu $ is mean stable, $\sharp (J(G_{\nu }))\geq 3$ and $\sharp (\Min(G_{\nu },\CCI ))=\sharp (\Min(G_{\tau },\CCI )).$ 
For each $L\in \Min(G_{\tau },\CCI )$, let $U_{L}$ be an open neighborhood of 
$L$ in $\Cpt(\CCI )$ such that 
$U_{L}\cap U_{L'}=\emptyset $ if $L, L'\in \Min(G_{\tau },\CCI )$ and $L\neq L'.$  
Let $L\in \Min(G_{\tau },\CCI )$ be an element. 
Let $r_{L}:=\dim _{\CC }(\mbox{LS}({\cal U}_{f,\tau}(L))).$ 
For each $r\in \NN $ and each $\nu \in \Omega $, 
we set $\Lambda _{r,\nu }:= \{ h_{r}\circ \cdots \circ h_{1}\mid h_{j}\in \G _{\nu } (\forall j)\} $ and 
set $G_{\nu }^{r}:= \langle \Lambda _{r,\nu } \rangle .$ 
By \cite[Theorem 3.15-12]{Splms10}, we have 
$r_{L}=\sharp (\Min (G_{\tau }^{r_{L}},L)).$ 
Let $\{ L_{j}\} _{j=1}^{r_{L}}=\Min(G_{\tau }^{r_{L}},\CCI ).$ 
By the proof of Lemma 5.16 in \cite{Splms10}, we may assume that 
for each $j=1,\ldots, r_{L}$ and for each $h\in \G _{\tau }$, 
$h(L_{j})\subset L_{j+1}$, where $L_{r_{L}+1}:= L_{1}.$ 
Moreover, by Lemma~\ref{l:Lattnh}, shrinking $\Omega $ if necessary, we obtain that 
for each $\nu \in \Omega $, 
there exists a unique $Q_{L, \nu }\in \Min(G_{\nu },\CCI )$ such that 
$Q_{L,\nu }\in U_{L}.$ Moreover, by Lemma~\ref{l:Lattnh} again, we may assume that 
the map $\nu \mapsto Q_{L,\nu }\in \Cpt(\CCI )$ is continuous on $\Omega .$ 
For each $j=1,\ldots, r_{L}$, 
let $V_{L,j}:= B(L_{j},\epsilon )$ (where $\epsilon >0$ is a small number)  
such that $V_{L,i}\cap V_{L,j}=\emptyset $ if $i\neq j.$ 
By Lemma~\ref{l:Lattnear}, shrinking $\Omega $ if necessary,  
there exists a unique element $L_{j,\nu }\in \Min(G_{\nu }^{r_{L}},\CCI )$ with $L_{j,\nu }\subset V_{L,j}.$ 
By Lemma~\ref{l:Lattnh}, we may assume that for each $j$, the map $\nu \mapsto L_{j,\nu }$ is continuous on $\Omega .$ 
Then, $\tilde{L}_{j+1,\nu }:= \bigcup _{h\in \G _{\nu }}h(L_{j,\nu })$ belongs to $\Min (G_{\nu }^{r_{L}},\CCI )$ and 
shrinking $\Omega $ if necessary, we obtain $\tilde{L}_{j+1,\nu }\subset V_{L,j+1}$, 
where $V_{L,r_{L}+1}:= V_{L,1}.$  
By the uniqueness statement of Lemma~\ref{l:Lattnear}, it follows that 
for each $j=1,\ldots ,r_{L}$, we have $\tilde{L}_{j+1,\nu }=L_{j+1,\nu }$, 
where $L_{r_{L}+1,\nu }:=L_{1,\nu }.$   
Since $\tilde{Q}_{L,\nu }:= \bigcup _{j=1}^{r_{L}}L_{j,\nu }$ belongs to $\Min(G_{\nu },\CCI )$ and 
$\tilde{Q}_{L,\nu }\in U_{L}$ (shrinking $\Omega $ if necessary), 
we obtain that $\tilde{Q}_{L,\nu }=Q_{L,\nu }.$  
From these arguments, it follows that 
for each $\nu \in \Omega $, 
$\Min(G_{\nu }^{r_{L}},Q_{L,\nu })=\{ L_{j,\nu }\} _{j=1}^{r_{L}}$, 
$L_{j+1,\nu }=\bigcup _{h\in \G _{\nu }}h(L_{j,\nu })$, and 
$\sharp (\Min(G_{\nu }^{r_{L}},Q_{L,\nu }))=r_{L}. $ 
We now prove the following claim.\\ 
Claim 1: For each $\nu \in \Omega $, $\dim _{\CC }(\mbox{LS}({\cal U}_{f,\nu }(Q_{L,\nu })))\geq r_{L}.$ 

To prove this claim, let $a_{L}:= \exp (2\pi i/r_{L})$ and let $\psi _{i}:= \sum _{j=1}^{r_{L}}a_{L}^{ij}1_{L_{j,\nu }}\in C(Q_{L,\nu }).$ 
Then $M_{\nu }(\psi _{i})=\sum a_{L}^{ij}1_{L_{j-1,\nu }}=a_{L}^{i}\sum a_{L}^{i(j-1)}1_{L_{j-1,\nu }}=a_{L}^{i}\psi _{i}$, 
where $L_{0,\nu }:= L_{r_{L},\nu }.$  
Hence, the above claim holds. 

 For each $\nu \in \Omega $, 
 let $\zeta _{\nu }:= \{ A\in \mbox{Con}(F(G_{\nu }))\mid A\cap Q_{L,\nu }\neq \emptyset \} $ and 
 let $W_{\nu }:= \bigcup _{A\in \zeta _{\nu }}A.$ 
Shrinking $\Omega $ if necessary, we obtain that  
for each $A\in \zeta _{\tau }$, there exists a unique element 
$\alpha _{\nu }(A)\in \zeta _{\nu }$ such that 
$A\cap \alpha _{\nu }(A)\neq \emptyset .$ 
It is easy to see that for each $\nu \in \Omega $, 
$\alpha _{\nu }: \zeta _{\tau }\rightarrow \zeta _{\nu }$ is bijective. 
This $\alpha _{\nu }$ induces  a linear isomorphism 
$\Psi _{\nu }:C_{W_{\nu }}(W_{\nu })\cong C_{W_{\tau }}(W_{\tau })$. 
Let $\tilde{M}_{\nu }:C_{W_{\tau }}(W_{\tau })\rightarrow C_{W_{\tau }}(W_{\tau })$ 
be the linear operator defined by 
$\tilde{M}_{\nu }:= \Psi _{\nu }\circ M_{\tau }\circ \Psi _{\nu }^{-1}.$ 
Then $\dim _{\CC }(C_{W_{\tau }}(W_{\tau }))<\infty $ and 
$\nu \mapsto (\tilde{M}_{\nu }: C_{W_{\tau }}(W_{\tau })\rightarrow C_{W_{\tau }}(W_{\tau }))$ 
is continuous. Moreover, by \cite[Theorem 3.15-8, Theorem 3.15-1]{Splms10}, 
each unitary eigenvalue of $M_{\tau }:C_{W_{\tau }}(W_{\tau })\rightarrow C_{W_{\tau }}(W_{\tau })$ is simple. 
Therefore, taking $\Omega $ small enough, we obtain that 
the dimension of the space of finite linear combinations of unitary eigenvectors 
of $\tilde{M}_{\nu }:C_{W_{\tau }}(W_{\tau }) \rightarrow C_{W_{\tau }}(W_{\tau })$ is less than or 
equal to $r_{L}.$ Combining this with Claim 1 and \cite[Theorem 3.15-10, Theorem 3.15-1]{Splms10}, we obtain that 
statement \ref{t:msmtaust2} of our theorem holds. 
By these arguments, statements \ref{t:msmtaust1-1},\ref{t:msmtaust1-2}, \ref{t:msmtaust2} hold. 

 We now prove statement \ref{t:msmtaust3} of our theorem. 
 For each $L\in \Min(G_{\tau },\CCI )$ and each $i=1,\ldots ,r_{L}$, 
 we set $\tilde{\psi }_{L,i}=\sum _{j=1}^{r_{L}}a_{L}^{ij}1_{L_{j}}\in C(L).$ 
 Then $M_{\tau }(\tilde{\psi }_{L,i})=a_{L}^{i}\tilde{\psi }_{L,i}.$ 
By \cite[Theorem 3.15-9]{Splms10}, 
there exists a unique element $\varphi _{L,i}\in C(\CCI )$ such that 
$\varphi _{L,i}|_{L}=\tilde{\psi }_{L,i}$, 
such that $\varphi _{L,i}|_{L'}=0$ for any $L'\in \Min(G_{\tau },\CCI )$ with $L'\neq L$, 
and such that $M_{\tau }(\varphi _{L,i})=a_{L}^{i}\varphi _{L,i}.$ 
Similarly, by using the notation in the previous arguments,  
for each $\nu \in \Omega $, for each $L\in \Min(G_{\tau },\CCI ), $ and for each 
$i=1,\ldots, r_{L}$, 
we set 
$\tilde{\psi} _{L,i,\nu }:=\sum _{j=1}^{r_{L}}a_{L}^{ij}1_{L_{j,\nu }}\in C(Q_{L,\nu })$.  
By \cite[Theorem 3.15-9]{Splms10}, 
there exists a unique element  
$\varphi _{L,i,\nu }\in C(\CCI )$ 
such that 
$\varphi _{L,i,\nu }|_{Q_{L,\nu }}=\tilde{\psi }_{L,i,\nu }$, 
such that $\varphi _{L,i,\nu }|_{Q'}=0$ for any $Q'\in \Min(G_{\nu },\CCI )$ with $Q'\neq Q_{L,\nu }$, 
and such that $M_{\nu }(\varphi _{L,i,\nu })=a_{L}^{i}\varphi _{L,i,\nu }.$ 
By statement \ref{t:msmtaust2} of our theorem, it follows that 
$\{ \varphi _{L,i,\nu }\} _{L\in \Min(G_{\tau },\CCI ),i=1,\ldots ,r_{L}}$ is a basis of 
$\mbox{LS}({\cal U}_{f,\nu }(\CCI )).$ 

Let $L\in \Min(G_{\tau },\CCI )$ and let $i=1,\ldots ,r_{L}.$  
We now prove that $\nu \mapsto \varphi _{L,i,\nu }\in C(\CCI )$ is continuous on $\Omega $. 
For simplicity, we prove that $\nu \mapsto \varphi _{L,i,\nu }\in C(\CCI )$ is continuous at $\nu =\tau .$ 
In order to do that, let $A_{j}$ be a relative compact open subset of $\CCI $ such that 
each connected component of $A_{j}$ intersects $L_{j}$, 
such that for each $\nu \in \Omega $, 
$L_{j,\nu }\subset A_{j}\subset \overline{A_{j}}\subset F(G_{\nu })$,  such that 
$\varphi _{L,i,\nu }|_{A_{j}}\equiv a_{L}^{ij}$, and such that 
$\{ \overline{A_{j}}\} _{j=1}^{r_{L}}$ are mutually disjoint. 
For each $j=1,\ldots ,r_{L}$, 
let $A_{j}'$ be an open subset of $A_{j}$ such that 
$L_{j}\subset A_{j}'\subset \overline{A_{j}'}\subset A_{j}.$   
Then there exists a number $s\in \NN $ and a neighborhood $\Omega '$ of $\tau $ in $({\frak M}_{1,c}({\cal Y}),{\cal O})$ 
such that for each $j=1,\ldots ,r_{L}$, for each $\nu \in \Omega '$, and for each 
$\g \in \G _{\nu }^{\NN }$, 
$\g _{s,1}(A_{j})\subset A_{j}'.$ 
Moreover, for each $K\in \Min(G_{\tau },\CCI )$ with $K\neq L$, let 
$B_{K}$ and $B_{K}' $ be two open subsets of $\CCI $ such that 
$K\subset B_{K}'\subset \overline{B_{K}'}\subset B_{K}\subset \overline{B_{K}}\subset F(G_{\tau })$ and such that 
each connected component of $B_{K}$ intersects $K.$ Then shrinking $\Omega '$ if necessary, 
there exists a number $s_{K}\in \NN $ such that 
for each $\nu \in \Omega '$ and for each $\g \in \G _{\nu }^{\NN }$, 
$\g _{s_{K},1}(B_{K})\subset B_{K}'.$ 
We may assume that $s_{K}=s$ for each $K\in \Min(G_{\tau },\CCI )$ with $K\neq L.$ 
Let $C:=\bigcup _{j=1}^{r_{L}}A_{j}\cup \bigcup _{K\in \Min(G_{\tau },\CCI ), K\neq L}B_{K}.$ 
Then for each $z\in \CCI $, 
$\lim _{n\rightarrow \infty }\int _{(\Rat )^{\NN }}1_{C}(\g _{n,1}(z)) d\tilde{\tau }(\g )=1.$ 
Let $\epsilon \in (0,1)$ be a small number. 
Let $z\in \CCI .$ Then there exists a number 
$l_{z}\in \NN $ such that 
$\tau ^{l_{z}}(\{ (\g _{1},\ldots ,\g _{l_{z}})\in (\Rat )^{l_{z}}\mid 
\g _{l_{z}}\circ \cdots \circ \g _{1}(z)\in C\} )\geq 1-\epsilon .$ 
Hence there exists a compact disk neighborhood $U_{z}$ of $z$ such that 
$\tau ^{l_{z}}(\{ (\g _{1},\ldots ,\g _{l_{z}})\in (\Rat )^{l_{z}}\mid 
\g _{l_{z}}\circ \cdots \circ \g _{1}(U_{z})\subset C\} )\geq 1-2\epsilon .$ 
Let $\{ z_{k}\} _{k=1}^{t}$ be a finite subset of $\CCI $ such that 
$\CCI =\bigcup _{k=1}^{t}U_{z_{k}}.$ 
We may assume that there exists an $l\in \NN $ such that 
for each $k=1,\ldots ,t$, $l_{z_{k}}=l.$ 
Taking $\Omega '$ so small, we obtain that 
for each $\nu \in \Omega '$ and for each $k=1,\ldots, t$, 
\begin{equation}
\label{eq:1-3ep}
\nu ^{l}(\{ (\g _{1},\ldots ,\g _{l})\in (\Rat )^{l}\mid \g _{l}\circ \cdots \circ \g _{1}(U_{z_{k}})\subset C\} ) 
\geq 1-3\epsilon .
\end{equation}  
 For each $k=1\ldots ,t$ and for each $j=1,\ldots ,r_{L}$, 
 we set $B_{k,j}:=\{ (\g _{1},\ldots ,\g _{sr_{L}l})\in (\Rat )^{sr_{L}l}\mid 
 \g _{sr_{L}l}\circ \cdots \circ \g _{1}(U_{z_{k}})\subset A_{j}\} .$ 
We may assume that $\tau ^{sr_{L}l}(\partial B_{k,j})=0$ for each $k,j.$ 
By Lemma~\ref{l:rhontaun}, taking $\Omega '$ so small, we obtain that for each $\nu \in \Omega '$,
  for each $k=1,\ldots ,t$, and for each $j=1,\ldots ,r_{L}$, 
\begin{equation}
\label{eq:bkjnu}
|\nu ^{sr_{L}l}(B_{k,j})-\tau ^{sr_{L}l}(B_{k,j})|<\epsilon .
\end{equation}  
Let $z\in \CCI $ and let $u\in \{ 1,\ldots ,t\} $ be such that 
$z\in U_{z_{u}}.$ 
Then for each $\nu \in \Omega '$ and each $i=1,\ldots ,r_{L}$, 
since $\varphi _{L,i,\nu }\in C_{F(G_{\nu }))}(\CCI )$ (\cite[Theorem 3.15-1]{Splms10}), we obtain that 
\begin{align*}
\ & \varphi _{L,i,\nu }(z)= M_{\nu }^{sr_{L}l}(\varphi _{L,i,\nu })(z) =  
\int _{\g \in \Rat ^{\NN }}\varphi _{L,i,\nu }(\g _{sr_{L}l,1}(z)) d\tilde{\nu }(\g )\\ 
 = &  \int _{\{ \g \in \Rat ^{\NN }\mid  \g _{sr_{L}l,1}(U_{z_{u}})\subset C\} } \varphi _{L,i,\nu }(\g _{sr_{L}l,1}(z)) d\tilde{\nu }(\g ) 
+\int _{\{ \g \in \Rat ^{\NN }\mid  \g _{sr_{L}l,1}(U_{z_{u}})\not\subset C\} } \varphi _{L,i,\nu }(\g _{sr_{L}l,1}(z)) d\tilde{\nu }(\g )\\  
 = & \sum _{j=1}^{r_{L}}a_{L}^{ij}\nu ^{sr_{L}l}(B_{u,j})   
 +\int _{\{ \g \in \Rat ^{\NN }\mid  \g _{sr_{L}l,1}(U_{z_{u}})\not\subset C\}} 
\varphi _{L,i,\nu }(\g _{sr_{L}l,1}(z)) d\tilde{\nu }(\g ). 
\end{align*} 
Combining this equation and (\ref{eq:1-3ep}), (\ref{eq:bkjnu}), we obtain 
$| \varphi _{L,i,\nu }(z)-\varphi _{L,i}(z)| \leq \sum _{j=1}^{r_{L}}\epsilon +3\epsilon \cdot 2=(r_{L}+6)\epsilon .$ 
Therefore, $\varphi _{L,i,\nu }\rightarrow \varphi _{L,i}$ in $C(\CCI )$ as $\nu \rightarrow \tau $. 
From these arguments, we obtain that $\nu \mapsto \varphi _{L,i,\nu }$ is continuous on $\Omega .$        
 
In order to construct $\{ \rho _{L,i,\nu }\} $ in statement \ref{t:msmtaust3} of our theorem, 
let $L\in \Min (G_{\tau },\CCI )$. 
 By the proof of Lemma 5.16 in \cite{Splms10}, 
 for each $j=1,\ldots ,r_{L}$, 
there exists an element $\omega _{L,j} \in {\frak M}_{1}(L_{j})$ such that 
for each $\varphi \in C(L_{j})$, 
$M_{\tau }^{nr_{L}}(\varphi )\rightarrow \omega _{L,j}(\varphi )\cdot 1_{L_{j}}$ 
in $C(L_{j})$ as $n\rightarrow \infty .$ 
We now prove the following claim. 

Claim 2. For each $\varphi \in C(\overline{A_{j}})$, 
$M_{\tau }^{nsr_{L}}(\varphi )\rightarrow \omega _{L,j}(\varphi )1_{\overline{A_{j}}}$ in 
$C(\overline{A_{j}})$ as $n\rightarrow \infty .$ 

 To prove this claim, let $\varphi \in C(\overline{A_{j}}).$ 
 Since $\overline{A_{j}}\subset F(G_{\tau })$, 
 $\{ M_{\tau }^{nsr_{L}}(\varphi )\} _{n\in \NN }$ is uniformly bounded and equicontinuous 
 on $\overline{A_{j}}.$ 
 Let $z\in \overline{A_{j}}$ be any point. Let $D_{z}\in \mbox{Con}(F(G_{\tau }))$ with 
$z\in D_{z}$ and  
let $w\in L_{j}\cap D_{z}$ be a point.  
By \cite[Theorem 3.15-4]{Splms10}, for $\tilde{\tau }$-a.e. $\g \in (\Rat )^{\NN }$, 
$d(\g _{nsr_{L},1}(z),\g _{nsr_{L},1}(w))\rightarrow 0$ as $n\rightarrow \infty .$ 
Therefore, $|M_{\tau }^{nsr_{L}}(\varphi )(z)-M_{\tau }^{nsr_{L}}(\varphi )(w)|\rightarrow 0$ as 
$n\rightarrow \infty .$ 
From these arguments, it follows that 
there exists a constant function $\xi : \overline{A_{j}}\rightarrow \RR $ such that 
$M_{\tau }^{nsr_{L}}(\varphi )\rightarrow \xi $ in $C(\overline{A_{j}})$ as $n\rightarrow \infty .$ 
Thus, we have proved Claim 2. 

 By using the arguments similar to the above, we obtain that 
for each $\nu \in \Omega $ and for each $j=1,\ldots ,r_{L}$, 
there exists an element $\omega _{L,j,\nu }\in {\frak M}_{1}(\overline{A_{j}})$ such that 
for each $\varphi \in C(L_{j,\nu })$, $M_{\nu }^{nr_{L}}(\varphi )\rightarrow \omega _{L,j,\nu }(\varphi )1_{L_{j,\nu }}$ 
in $C(L_{j,\nu })$ as $n\rightarrow \infty $, and such that 
for each $\varphi \in C(\overline{A_{j}})$, 
$M_{\nu }^{nsr_{L}}(\varphi )\rightarrow \omega _{L,j,\nu }(\varphi )1_{\overline{A_{j}}}$ in 
$C(\overline{A_{j}})$ as $n\rightarrow \infty .$ 
Since $L_{j,\nu }$ is the unique minimal set for $(G_{\nu }^{r_{L}},\overline{A_{j}})$ and 
$L_{j,\nu }$ is attracting for $(G_{\nu }^{r_{L}},\CCI )$, we obtain 
supp$\, \omega _{L,j,\nu }=L_{j,\nu }.$ 
For each $\nu \in \Omega $, for each $L\in \Min(G_{\tau },\CCI )$  
and  for each $i=1,\ldots ,r_{L}$, 
let $\rho _{L,i,\nu }:= \frac{1}{r_{L}}\sum _{j=1}^{r_{L}}a_{L}^{-ij}\omega _{L,j,\nu }\in 
C(Q_{L,\nu })^{\ast }\subset C(\CCI )^{\ast }.$ 
Then by the proofs of Lemmas 5.16 and 5.14 from \cite{Splms10}, 
we obtain that $M_{\nu }^{\ast }(\rho _{L,i,\nu })=a_{L}^{i}\rho _{L,i,\nu }$, 
that  $\rho _{L,i,\nu }(\varphi _{L,j,\nu })=\delta _{ij}$, 
that $\rho _{L,i,\nu }(\varphi _{L',j,\nu })=0$ if $L\neq L'$, 
that $\{ \rho _{L,i,\nu }|_{C(Q_{L,\nu })} \mid i=1,\ldots, r_{L}\}$ is a basis of 
$\mbox{LS}({\cal U}_{f,\nu ,\ast }(Q_{L,\nu }))$, 
that $\{ \rho _{L,i,\nu } \mid L\in \Min(G_{\tau },\CCI ),i=1,\ldots, r_{L}\}$ is a basis of 
$\mbox{LS}({\cal U}_{f,\nu ,\ast }(\CCI ))$, and that 
$\pi _{\nu }(\varphi )=\sum _{L\in \Min(G_{\tau },\CCI )}\sum _{i=1}^{r_{L}}\rho _{L,i,\nu }(\varphi )\cdot \varphi _{L,i,\nu }$ 
for each $\varphi \in C(\CCI ).$  
   
We now prove that 
for each $L\in \Min(G_{\tau },\CCI )$ and for each $i=1,\ldots ,r_{L}$, 
the map $\nu \mapsto \rho _{L,i,\nu }\in C(\CCI )^{\ast }$ is continuous on $\Omega .$ 
For simplicity, we prove that 
$\nu \mapsto \rho _{L,i,\nu }\in C(\CCI )^{\ast }$ is continuous at $\nu =\tau .$ 
Let $\varphi \in C(\overline{A_{j}}).$ Let $\epsilon >0.$ 
Then there exists an $n\in \NN $ such that 
$\| M_{\tau }^{nsr_{L}}(\varphi )-\omega _{L,j}(\varphi )1_{\overline{A_{j}}}\| _{\infty }<\epsilon $, 
where $\| \psi \| _{\infty }:= \sup _{z\in \overline{A_{j}}}| \psi (z)| $ for each 
$\psi \in C(\overline{A_{j}}).$ If $\Omega '$ is a small open neighborhood of 
$\tau $ in $({\frak M}_{1,c}({\cal Y}),{\cal O})$, then 
for each $\nu \in \Omega '$, 
$\| M_{\nu }^{nsr_{L}}(\varphi )-M_{\tau }^{nsr_{L}}(\varphi )\| _{\infty }<\epsilon .$ 
Hence, for each $\nu \in \Omega '$, 
$\| M_{\nu }^{nsr_{L}}(\varphi )-\omega _{L,j}(\varphi )1_{\overline{A_{j}}}\| _{\infty }<2\epsilon .$ 
Therefore, for each $\nu \in \Omega ' $ and for each $l\in \NN $, 
 $\| M_{\nu }^{lsr_{L}}(M_{\nu }^{nsr_{L}}(\varphi )-\omega _{L,j}(\varphi )1_{\overline{A_{j}}})\| _{\infty }<2\epsilon .$ 
 Thus, 
 $\| M_{\nu }^{(l+n)sr_{L}}(\varphi )-\omega _{L,j}(\varphi )1_{\overline{A_{j}}}\| _{\infty }<2\epsilon .$ 
Moreover, $M_{\nu }^{(l+n)sr_{L}}(\varphi )\rightarrow \omega _{L,j,\nu }(\varphi )1_{\overline{A_{j}}}$ 
in $C(\overline{A_{j}})$ as $l\rightarrow \infty .$ 
Hence, we obtain  that for each $\nu \in \Omega '$, \ 
$| \omega _{L,\nu ,j}(\varphi )-\omega _{L,j}(\varphi )|\leq 2\epsilon .$ 
From these arguments, it follows that 
the map $\nu \mapsto \omega _{L,j,\nu }\in C(\overline{A_{j}})^{\ast }\subset C(\CCI )^{\ast }$ is continuous 
at $\nu =\tau .$ 
Therefore, for each $L\in \Min(G_{\tau },\CCI )$ and for each $i=1,\ldots ,r_{L}$, 
the map $\nu \mapsto \rho _{L,i,\nu }\in  C(\CCI )^{\ast }$ is continuous 
at $\nu =\tau .$  
Thus, for each $L\in \Min(G_{\tau },\CCI )$ and for each $i=1,\ldots ,r_{L}$, 
the map $\nu \mapsto \rho _{L,i,\nu }\in C(\CCI )^{\ast }$ is continuous on $\Omega .$ 
Hence, we have proved statement \ref{t:msmtaust3} of our theorem. 

 We now prove statement \ref{t:msmtaust4} of our theorem. 
 For each $L\in \Min (G_{\tau },\CCI )$, let 
 $V_{L}$ be an open subset of $F(G_{\tau }) $ with $L\subset V_{L}$ such that 
 for each $L,L'\in \Min (G_{\tau },\CCI )$ with $L\neq L'$, 
$\overline{V_{L}}\cap \overline{V_{L'}}=\emptyset .$ 
By statement \ref{t:msmtaust1-1} and Lemma~\ref{l:Lattnh}, 
for each $L\in \Min (G_{\tau },\CCI )$, 
there exists a continuous map $\nu \mapsto Q_{L,\nu }\in  \Cpt(\CCI )$ on $\Omega $ 
with respect to the Hausdorff metric   
such that $Q_{L,\tau }=L$, such that for each $\nu \in \Omega $, 
$\{ Q_{L,\nu }\} _{L\in \Min(G_{\tau },\CCI )}=\Min (G_{\nu },\CCI )$,  
and such that for each $\nu \in \Omega $ and for each $L\in \Min (G_{\tau },\CCI )$, 
$Q_{L,\nu }\subset V_{L}.$ For each $L\in \Min (G_{\tau },\CCI )$, 
let $\varphi _{L}:\CCI \rightarrow [0,1]$ be a continuous function such that 
$\varphi _{L}|_{V_{L}}\equiv 1$ and $\varphi _{L}|_{V_{L'}}\equiv 0$ for each 
$L'\in \Min (G_{\tau },\CCI )$ with $L'\neq L.$ 
By \cite[Theorem 3.15-15]{Splms10}, 
it follows that for each $z\in \CCI $ and for each $\nu \in \Omega $,  
$T_{Q_{L,\nu }, \nu }(z)=\lim _{n\rightarrow \infty }M_{\nu }^{n}(\varphi _{L})(z).$ 
Combining this with \cite[Theorem 3.14]{Splms10}, 
we obtain $T_{Q_{L,\nu },\nu }=\lim _{n\rightarrow \infty }M_{\nu }^{n}(\varphi _{L})$ in 
$(C(\CCI ), \| \cdot \| _{\infty }).$  By \cite[Theorem 3.15-6,8,9]{Splms10}, 
for each $\nu \in \Omega $ there exists a number $r\in \NN $ such that 
for each $\psi \in \mbox{LS}({\cal U}_{f,\nu }(\CCI ))$, 
$M_{\nu }^{r}(\psi )=\psi .$ 
Therefore, for each $\nu \in \Omega $ and for each 
$L\in \Min (G_{\tau },\CCI )$, 
$T_{Q_{L,\nu },\nu }=\lim _{n\rightarrow \infty }M_{\nu }^{nr}(\varphi _{L})
=\lim _{n\rightarrow \infty }M_{\nu }^{nr}
(\varphi _{L}-\pi _{\nu }(\varphi _{L})+\pi _{\nu }(\varphi _{L}))=\pi _{\nu }(\varphi _{L}).$ 
Combining this with statement \ref{t:msmtaust3} of our theorem, 
it follows that for each $L\in \Min (G_{\tau },\CCI )$, 
the map $\nu \mapsto T_{Q_{L,\nu } ,\nu }\in (C(\CCI ), \| \cdot \| _{\infty })$ 
is continuous on $\Omega .$  
Thus, we have proved statement \ref{t:msmtaust4} of our theorem. 

Hence, we have proved Theorem~\ref{t:msmtaust}.
\qed 

\ 

We now prove Theorem~\ref{t:ABCD}. \\ 
{\bf Proof of Theorem~\ref{t:ABCD}:} 
It is trivial that $B\subset C.$ 
By Theorem~\ref{t:msmtaust}, 
we obtain that 
$A\subset B$ and $A\subset D.$ 
In order to show $C\subset A$, let $\tau \in C.$ 
If there exists a non-attracting minimal set for $(G_{\tau },\CCI )$, 
or if there exists no attracting minimal set for $(G_{\tau }, \CCI )$,  
then by Theorem~\ref{t:msminr} and Corollary~\ref{c:noattminme}, we obtain a contradiction. 
Hence, $C\subset A.$ Therefore, we obtain $A=B=C.$ 
In order to show $D\subset A$, let 
$\tau \in D.$ 
By \cite[Theorem 3.15-10]{Splms10}, 
we have $\dim _{\CC }(\mbox{LS}({\cal U}_{f,\tau }(\CCI ))=
\sum _{L\in \Min(G_{\tau },\CCI )}\dim _{\CC }(\mbox{LS}({\cal U}_{f,\tau }(L))).$ 
By Corollary~\ref{c:noattminme}, there exists an attracting minimal set for $(G_{\tau },\CCI ).$ 
Theorem~\ref{t:msminr} implies that 
if $\Omega $ is a small neighborhood of $\tau $ in $({\frak M}_{1,c}({\cal Y}),{\cal O})$, 
then for each $\nu \in \Omega $ and for each attracting minimal set $L$ for 
$(G_{\tau },\CCI )$, there exists a unique attracting minimal set $Q_{L,\nu }$ for 
$(G_{\nu },\CCI )$ which is close to $L.$ By \cite[Theorem 3.15-12]{Splms10} and the arguments 
in the proof of Theorem~\ref{t:msmtaust}, it follows that 
if $\Omega $ is small enough, then for each 
$\nu \in \Omega $ and for each attracting minimal set $L$ for $(G_{\tau },\CCI )$,  
$\dim _{\CC }(\mbox{LS}({\cal U}_{f,\nu }(Q_{L,\nu })))=
\dim _{\CC }(\mbox{LS}({\cal U}_{f,\tau }(L))).$ 
Combining this, Theorem~\ref{t:msminrme} and \cite[Theorem 3.15-10]{Splms10}, we obtain that 
if there exists a non-attracting minimal set $L'$ for $(G_{\tau },\CCI )$, 
then there exists a $\nu '\in \Omega $ such that 
$\dim _{\CC }(\mbox{LS}({\cal U}_{f,\nu '}(\CCI )))<\dim _{\CC }(\mbox{LS}({\cal U}_{f,\tau }(\CCI ))).$ 
However, this contradicts $\tau \in D.$ 
Therefore, we obtain that each element $L\in \Min(G_{\tau },\CCI )$ is attracting 
for $(G_{\tau },\CCI ).$ 
By Remark~\ref{r:msallatt}, it follows that $\tau \in A.$ Therefore, 
$D\subset A.$ 

From these arguments, we obtain $A=B=C=D.$ 

 By Theorem~\ref{t:msmtaust}, we obtain that $A\subset E.$ In order to show $E\subset A$, 
 let $\tau \in E.$ 
Suppose that there exists a non-attracting minimal set $K$ for $(G_{\tau }, \CCI ).$ 
Since there exists a neighborhood $\Omega '$ of $\tau $ such that each 
$\nu \in \Omega '$ satisfies $J_{\ker }(G_{\nu })=\emptyset $, 
Corollary~\ref{c:noattminme} implies that 
there exists an attracting minimal set for $(G_{\tau },\CCI ).$ 
Moreover, since $J_{\ker }(G_{\tau })=\emptyset $ and 
$\sharp (J(G_{\tau }))\geq 3$, \cite[Theorem 3.15-6]{Splms10} implies that 
$\sharp (\Min (G_{\tau },\CCI )) <\infty .$ 
Let $\epsilon := \min \{ d(z,w)\mid z\in K, w\in L, L\in \Min (G_{\tau },\CCI ), L\neq K\} >0.$ 
Let $\varphi \in C(\CCI )$ be an element such that 
$\varphi |_{K}\equiv 1$ and $\varphi |_{\CCI \setminus B(K,\epsilon /2)}\equiv 0.$ 
Then by \cite[Theorem 3.15-13]{Splms10}, $\pi _{\tau }(\varphi )\neq 0.$ 
Since $\tau \in E$, there exists an open neighborhood $\Omega $ of $\tau $ such that 
for each $\nu \in \Omega $, $J_{\ker }(G_{\nu })=\emptyset $ and 
such that the map $\nu \mapsto \pi _{\nu }(\varphi )\in C(\CCI )$ defined on $\Omega $ is continuous at $\tau .$ 
By Theorem~\ref{t:msminrme}, for each neighborhood ${\cal U}$ of $\tau $ in $({\frak M}_{1,c}({\cal Y}),{\cal O})$, 
there exists an element $\rho \in {\cal U} \cap A$ such that each minimal set for $(G_{\rho },\CCI )$ 
is included in $\CCI \setminus B(K,\epsilon /2).$ Therefore, by \cite[Theorem 3.15-2]{Splms10}, 
$\pi _{\rho }(\varphi )=0.$ However, this contradicts that 
the map $\nu \mapsto \pi _{\nu }(\varphi )\in C(\CCI )$ is continuous 
at $\tau $ and that $\pi _{\tau }(\varphi )\neq 0.$ Thus, 
each element of $\Min (G_{\tau },\CCI )$ is attracting for $(G_{\tau },\CCI ).$ 
By Remark~\ref{r:msallatt}, it follows that $\tau \in A.$ Hence, we have proved 
$E\subset A.$  

Thus, we have proved Theorem~\ref{t:ABCD}.  
\qed 

To prove Theorem~\ref{t:bifur}, we need the following. 
\begin{lem}
\label{l:bifmin}
Let ${\cal Y}$ be a subset of $\emRatp$ satisfying condition $(\ast )$. 
For each $t\in [0,1]$. let $\mu _{t}$ be an element of ${\frak M}_{1,c}({\cal Y}).$ 
Suppose that all conditions {\em (1)(2)(3)} in Theorem~\ref{t:bifur} are satisfied. 
Then, statements {\em (a)} and {\em (c)} in Theorem~\ref{t:bifur} hold. 
Moreover, $\sharp \emMin (G _{\mu _{t}}, \CCI )<\infty $ 
for each $t\in [0,1].$ 
\end{lem}
\begin{proof}
Since $F(G _{\mu _{1}})\neq \emptyset $ and 
$G _{\mu _{t}}\subset G _{\mu _{1}}$ for each 
$t\in [0,1]$, we obtain that 
$F(G _{\mu _{t}})\neq \emptyset $. 
Moreover, we have that for each $t\in [0,1]$, 
$\mbox{int}(\G _{\mu _{t}})\neq \emptyset $ in the topology of ${\cal Y}.$ 
Therefore, \cite[Lemma 5.34]{Splms10} implies that 
for each $t\in [0,1]$, 
$J_{\ker }(G _{\mu _{t}})=\emptyset $. 
Moreover, since ${\cal Y}\subset \Ratp$, we have that for each $t\in [0,1]$,  
$\sharp J(G _{\mu _{t}})\geq 3.$ 
Thus, by \cite[Theorem 3.15]{Splms10}, it follows that for each $t\in [0,1]$, 
all statements (with $\tau =\mu _{t}$ ) in \cite[Theorem 3.15]{Splms10} hold. 
In particular, $\sharp (\Min (G _{\mu _{t}},\CCI ))<\infty $ 
for each $t\in [0,1]$, and statement (a) of Theorem~\ref{t:bifur} holds. 

To show that statement (c) of Theorem~\ref{t:bifur}, it suffices to show that 
there exists an element $u\in [0,1)$ such that 
for each $t\in [u,1]$, 
$\sharp (\Min (G_{\mu _{t}},\CCI ))=\sharp (\Min (G_{\mu _{1}},\CCI )).$ 
In order to show it, we first note that 
by Zorn's lemma, we have 
\begin{equation}
\label{eq:bifmin1}
\sharp (\Min (G_{\mu _{t}},\CCI ))\geq \sharp (\Min (G_{\mu _{1}},\CCI )) \mbox{ for each }
t\in [0,1].
\end{equation} 
Let $\{ L_{j}\} _{j=1}^{r}=\Min (G_{\mu _{1}},\CCI )$, where 
$L_{i}\neq L_{j}$ for each $(i,j)$ with $i\neq j.$ 
Since $J_{\ker }(G_{\mu _{1}})=\emptyset, $ 
there exists an element $w_{j}\in L_{j}\cap F(G_{\mu _{1}})$ for each $j=1,\ldots ,r.$ 
Let $\epsilon >0$ be a small number such that 
$W:= \bigcup _{j=1}^{r}B(w_{j},\epsilon )\subset F(G_{\mu _{1}}).$ 
Since $J_{\ker }(G_{\mu _{1}})=\emptyset $, \cite[Theorem 3.15-7]{Splms10} implies that 
for each $z\in \CCI $, 
there exists an element $g_{z}\in G_{\mu _{1}}$ and a neighborhood $V_{z}$ of $z$ in $\CCI $ 
such that $g_{z}(\overline{V_{z}})\subset W.$ 
Since $\CCI $ is compact, there exist finitely many points $z_{1},\ldots ,z_{n}\in \CCI $ 
such that $\CCI =\bigcup _{j=1}^{n}V_{z_{j}}.$ Then there exists an element $u\in [0,1)$ 
such that for each $j=1,\ldots, n$ and for each $t\in [u,1]$,   
there exists an element $g_{z_{j},t}\in G_{\mu _{t}}$ 
with $g_{z_{j},t}(\overline{V_{j}})\subset W.$ 
Moreover, we have $G_{\mu _{t}}\subset G_{\mu _{1}}$ and  
$F(G_{\mu _{1}})\subset F(G_{\mu _{t}})$ for each $t\in [u,1].$ 
Applying \cite[Theorem 3.15-4]{Splms10} (with $\tau =\mu _{t}, t\in [u,1]$), 
it follows that for each $t\in [u,1]$ and for each 
$L\in \Min (G_{\mu _{t}},\CCI )$, 
there exists a unique element $L'\in \Min (G_{\mu _{1}},\CCI )$ with $L\subset L'.$ 
Therefore $\sharp (\Min (G_{\mu _{t}},\CCI ))\leq \sharp (\Min (G_{\mu _{1}},\CCI ))$ for each 
$t\in [u,1].$ Combining this with (\ref{eq:bifmin1}), we obtain 
$\sharp (\Min (G_{\mu _{t}},\CCI ))=\sharp (\Min (G_{\mu _{1}},\CCI ))$ for each 
$t\in [u,1].$ Thus we have proved our lemma. 
\end{proof}

We now prove Theorem~\ref{t:bifur}. \\ 
{\bf Proof of Theorem~\ref{t:bifur}:} 
By Lemma~\ref{l:bifmin}, statements (a) and (c) of our theorem hold and 
$\sharp (\Min (G _{\mu _{t}}, \CCI ))<\infty $ 
for each $t\in [0,1].$ 

We now prove statement (b). 
By Lemma~\ref{l:Lclassify} and Remark~\ref{r:msallatt}, we obtain that 
for each $t\in B$, $\mu _{t}$ is not mean stable, and that 
for each $t\in [0,1)\setminus B$, $\mu _{t}$ is mean stable. 
Combining this with assumption (4) and Lemma~\ref{l:msmincons},  
we obtain that $B\neq \emptyset .$  
We now let $t_{1},t_{2}\in [0,1]$ be such that $t_{1}<t_{2}$. 
By assumption (2) and Remark~\ref{r:minimal},  
for each $L\in \Min (G _{\mu _{t_{2}}},\CCI )$, 
there exists an $L'\in \Min (G _{\mu _{t_{1}}},\CCI )$ 
with 
$L'\subset L.$ In particular, 
$\infty >\sharp (\Min (G_{\mu _{t_{1}}},\CCI ))\geq 
\sharp (\Min(G _{\mu _{t_{2}}},\CCI )).$ 
 We now let $t_{0}\in [0,1)$ be such that 
 there exists a bifurcation element $g\in \G _{\mu _{t_{0}}}$ for 
 $\G _{\mu _{t_{0}}}.$ 
 Let $t\in [0,1]$ with $t>t_{0}$. Then 
 $\G _{\mu _{t_{0}}}\subset \mbox{int}(\G _{\mu _{t}})$. 
By the above argument, Theorem~\ref{t:msminr}, Corollary~\ref{c:noattminme} and assumption (3) of our theorem,  
it follows that 
$\sharp (\Min (G _{\mu _{t_{0}}},\CCI ))>
\sharp (\Min (G _{\mu _{t}},\CCI )).$   
From these arguments, it follows that 
$1\leq \sharp B \leq \sharp (\Min(G _{\mu _{0}},\CCI ))
-\sharp (\Min(G _{\mu _{1}},\CCI ))<\infty .$

Thus, we have proved Theorem~\ref{t:bifur}. 
\qed 

\subsection{Proofs of results in \ref{Spectral}}
\label{Proofs of Spectral}
In this subsection, we give the proofs of the results in 
subsection~\ref{Spectral}. 

We now prove Theorem~\ref{t:utauca}.\\ 
{\bf Proof of Theorem~\ref{t:utauca}:}
By \cite[Theorem 3.15-6,8,9]{Splms10}, there exists an $r\in \NN $ 
such that 
for each $\varphi \in \mbox{LS}({\cal U}_{f,\tau }(\CCI ))$, 
$M_{\tau }^{r}(\varphi )=\varphi .$ 
Since $J_{\ker }(G_{\tau })=\emptyset $, 
for each $z\in \CCI $, there exists a map $g_{z}\in G_{\tau }$ and a 
compact disk neighborhood $U_{z}$ of $z$ in $\CCI $ such that 
$g_{z}(U_{z})\subset F(G_{\tau }).$ 
Since $\CCI $ is compact, there exists a finite family 
$\{ z_{j}\} _{j=1}^{s}$ in $\CCI $ such that 
$\bigcup _{j=1}^{s}\mbox{int}(U_{z_{j}})=\CCI .$ 
Since $G_{\tau }(F(G_{\tau }))\subset F(G_{\tau })$, 
replacing $r$ by a larger number if necessary,  
we may assume that for each $j=1,\ldots ,s$, 
there exists an element $\beta ^{j}=(\beta _{1}^{j},\ldots ,\beta _{r}^{j})\in \G _{\tau }^{r}$ 
such that $g_{z_{j}}=\beta _{r}^{j}\circ \cdots \circ \beta _{1}^{j}.$  
For each $j=1,\ldots ,s$, let 
$V_{j}$ be a compact neighborhood of $\beta ^{j}$ in $\G _{\tau }^{r}$ 
such that for each $\zeta =(\zeta _{1},\ldots ,\zeta _{r})\in V_{j}$, 
$\zeta _{r}\cdots \zeta _{1}(U_{z_{j}})\subset F(G_{\tau }).$ 
Let $a:=\max \{ \tau ^{r}(\G _{\tau }^{r}\setminus V_{j})\mid j=1,\ldots ,s\} \in [0,1).$ 
Let $$C_{1}:= 2\max \{ \max\{ \| D(\zeta _{r}\circ \cdots \circ \zeta _{1})_{z}\| _{s}\mid 
(\zeta _{1},\ldots ,\zeta _{r})\in \G _{\tau }^{r}, z\in \CCI  \} ,1\} \geq 2.$$   
Let $\alpha \in [0,1)$ be a number such that 
$aC_{1}^{\alpha }<1.$ Let $C_{2}>0$ be a number such that 
for each $z\in \CCI $, there exists a $j\in \{ 1,\ldots ,s\} $ with 
$B(z,C_{2})\subset \mbox{int}(U_{z_{j}}).$  
Let $\varphi \in \mbox{LS}({\cal U}_{f,\tau }(\CCI )).$ 
Let $z_{0},z\in \CCI $ be two points. 
If $d(z,z_{0})>C_{1}^{-1}C_{2}$, then 
$$| \varphi (z)-\varphi (z_{0})|/d(z,z_{0})^{\alpha }\leq 
2\| \varphi \| _{\infty }\cdot (C_{1}C_{2}^{-1})^{\alpha }.$$ 
We now suppose that there exists an $n\in \NN $ such that 
$C_{1}^{-n-1}C_{2}\leq d(z,z_{0})\leq C_{1}^{-n}C_{2}.$ 
Then, for each $j\in \NN $ with $1\leq j\leq n$ and 
for each $(\g _{1},\ldots ,\g_{rj})\in \G _{\tau }^{rj}$, we have 
$d(\g _{rj}\circ \cdots \circ \g _{1}(z),\g _{rj}\circ \cdots \circ \g _{1}(z_{0}))<C_{2}.$ 
Let $i_{0}\in \{ 1,\ldots ,s\} $ be a number such that 
$B(z_{0},C_{2})\subset U_{z_{i_{0}}}.$ 
Let $A(0):= \{ \g \in \G _{\tau }^{\NN }\mid 
(\g _{1},\ldots ,\g _{r})\in V_{i_{0}}\} $ and 
$B(0):=\{ \g \in \G _{\tau }^{\NN }\mid 
(\g _{1},\ldots ,\g _{r})\not\in V_{i_{0}}\} .$ 
 Inductively, for each $j=1,\ldots ,n-1$, 
 let $A(j):= \{ \g \in B(j-1)\mid \exists i \mbox{ s.t. } 
 B(\g _{rj,1}(z_{0}),C_{2})\subset U_{z_{i}}, (\g _{rj+1},\ldots ,\g _{rj+r})\in V_{i}\} $ 
 and $B(j):= B(j-1)\setminus A(j).$ 
Then for each $j=1,\ldots ,n-1$, 
$\tilde{\tau }(B(j))\leq a\tilde{\tau }(B(j-1)).$ Therefore, 
$\tilde{\tau }(B(n-1))\leq a^{n}.$ 
Moreover, we have 
$\G _{\tau }^{\NN }=\amalg _{j=0}^{n-1}A(j)\amalg B(n-1).$ 
Furthermore, by \cite[Theorem 3.15-1]{Splms10}, $\varphi \in C_{F(G_{\tau })}(\CCI ).$ 
Thus, we obtain that 
\begin{align*}
\    & |\varphi (z)-\varphi (z_{0})| = |M_{\tau }^{rn}(\varphi )(z)-M_{\tau }^{rn}(\varphi )(z_{0})| \\ 
\leq & |\sum _{j=0}^{n-1}\int _{A(j)}\varphi (\g _{rn,1}(z))-\varphi (\g _{rn,1}(z_{0})) d\tilde{\tau }(\g )| 
       + |\int _{B(n-1)}\varphi (\g _{rn,1}(z))-\varphi (\g _{rn,1}(z_{0}))d\tilde{\tau }(\g )|\\ 
\leq & \int _{B(n-1)}|\varphi (\g _{rn,1}(z))-\varphi (\g _{rn,1}(z_{0}))|d\tilde{\tau }(\g )\\ 
\leq & 2a^{n}\| \varphi \| _{\infty }\leq a^{n}(C_{1}^{n+1}C_{2}^{-1})^{\alpha }d(z,z_{0})^{\alpha }2\| \varphi \| _{\infty }
\leq C_{1}^{\alpha }C_{2}^{-\alpha }2\| \varphi \| _{\infty } d(z,z_{0})^{\alpha }.        
\end{align*}  
From these arguments, it follows that 
$\varphi $ belongs to $C^{\alpha }(\CCI ).$ 

Let $\{ \rho _{j}\} _{j=1}^{q}$ be a basis of 
$\mbox{LS}({\cal U}_{f,\tau ,\ast }(\CCI ))$ and let 
$\{ \varphi _{j}\} _{j=1}^{q}$ be a basis of 
$\mbox{LS}({\cal U}_{f,\tau }(\CCI ))$ such that 
for each $\psi \in C(\CCI )$, 
$\pi _{\tau }(\psi )=\sum _{j=1}^{q}\rho _{j}(\psi )\varphi _{j}.$ 
Then for each $\psi \in C(\CCI )$, 
$\| \pi _{\tau }(\psi )\| _{\alpha }\leq \sum _{j=1}^{q}|\rho _{j}(\psi )|\| \varphi _{j}\| _{\alpha }
\leq (\sum _{j=1}^{q}\| \rho _{j}\| _{\infty }\| \varphi _{j}\| _{\alpha })\| \psi \| _{\infty }$, 
where $\| \rho _{j}\| _{\infty }$ denotes the operator norm of 
$\rho _{j}:(C(\CCI ),\| \cdot \| _{\infty })\rightarrow \CC .$ 
 
 We now let $L\in \Min(G_{\tau },\CCI)$ and let $\alpha \in (0,\alpha _{0})$. 
 By \cite[Theorem 3.15-15]{Splms10}, 
 $T_{L,\tau }\in \mbox{LS}({\cal U}_{f,\tau }(\CCI )).$ 
 Thus $T_{L,\tau }\in C^{\alpha }(\CCI ).$ 
 
Thus, we have proved Theorem~\ref{t:utauca}.  
\qed 

\begin{rem}
\label{r:unifalpha}
Suppose $\tau \in {\frak M}_{1,c}(\Rat )$, 
$J(G_{\tau })\neq \emptyset $ and that $\tau $ is mean stable. 
Then by using Theorem~\ref{t:msmtaust}--\ref{t:msmtaust1} and the method of proof of 
Theorem~\ref{t:utauca}, it is easy to see that 
there exists an $\alpha \in (0,1)$ and a neighborhood ${\cal U}$ of $\tau $ in 
${\frak M}_{1,c}(\Rat)$ such that 
for each $\nu \in {\cal U}$, we have that 
$\nu $ is mean stable and $\mbox{LS}({\cal U}_{f,\nu }(\CCI ))\subset C^{\alpha }(\CCI ).$  
\end{rem} 

In order to prove Theorem~\ref{t:kjemfhf}, we need several lemmas. 
Let $\tau \in {\frak M}_{1,c}(\Rat).$ 
Suppose 
$J_{\ker }(G_{\tau })=\emptyset $ and 
$\sharp J(G_{\tau })\geq 3.$ Then  
all statements in \cite[Theorem 3.15]{Splms10} hold. 
Let $L\in \Min(G_{\tau },\CCI )$ and 
let $r_{L}:=\dim _{\CC }(\mbox{LS}({\cal U}_{f,\tau }(L))).$ 
By using the notation in the proof of Theorem~\ref{t:msmtaust}, 
by \cite[Theorem 3.15-12]{Splms10}, we have $r_{L}=\sharp (\Min (G_{\tau }^{r_{L}},L)).$   

\begin{lem}
\label{l:supai}
Let $\tau \in {\frak M}_{1,c}(\emRat).$ 
Suppose 
$J_{\ker }(G_{\tau })=\emptyset $ and 
$\sharp J(G_{\tau })\geq 3.$  
Let $L\in \emMin(G_{\tau },\CCI )$ and 
let $r_{L}:=\dim _{\CC }(\mbox{{\em LS}}({\cal U}_{f,\tau }(L))).$ 
Let $\{ L_{j}\} _{j=1}^{r_{L}}=\emMin(G_{\tau }^{r_{L}},L).$ 
For each $j$, 
let $\{ A_{i}\} _{i\in I_{j}}$ be the set $\{ A\in \mbox{{\em Con}}(F(G_{\tau }))\mid A\cap L_{j}\neq \emptyset \} .$ 
Let $W_{L,j}:=\bigcup _{i\in I_{j}}A_{i}.$ 
For each $i\in I_{j}$, we take the hyperbolic metric in $A_{i}.$ 
Then, there exists an $m\in \NN $ with $r_{L}| m$ such that for each $j$ and for each 
$\alpha \in (0,1)$, 
$\sup _{i\in I_{j}}\int _{\emRat ^{\NN }}\sup _{z\in A_{i}}\{ \| D(\g _{m,1})_{z}\| _{h}^{\alpha }\} d\tilde{\tau }(\g )<1.$  
\end{lem}
\begin{proof}
For each $g\in G_{\tau }^{r_{L}}$, we have $g(\overline{W_{L,j}})\subset \overline{W_{L,j}}.$ 
Combining this with \cite[Theorem 3.15-7]{Splms10}, we obtain that 
for each $z\in \partial W_{L,j}$, there exists a map $g_{z}\in G_{\tau }$ 
and an open disk neighborhood $U_{z}$ of $z$ such that 
$g_{z}(\overline{U_{z}})\subset W_{L,j}.$ 
Then there exists a finite family $\{ z_{l}\} _{l=1}^{t}$ such that 
$\partial W_{L,j}\subset \bigcup _{l=1}^{t}U_{z_{l}}.$ 
Since $G_{\tau }^{r_{L}}(W_{L,j})\subset W_{L,j}$, we may assume that 
there exists a $k\in \NN $ with $r_{L}|k$ and a 
finite family $\{ a^{l}=(\alpha _{1}^{l},\ldots ,\alpha _{k}^{l})\in \G_{\tau }^{k}\} _{l=1}^{t}$ 
such that for each $l=1,\ldots ,t$, 
$g_{z_{l}}=\alpha _{k}^{l}\circ \cdots \circ \alpha _{1}^{l}.$ 
Let 
$K_{0}:=(W_{L,j}\setminus \bigcup _{l=1}^{t}U_{z_{l}})\cup \bigcup _{l=1}^{t}g_{z_{l}}(U_{z_{l}})$ 
and let $\{ B_{1},\ldots ,B_{u}\} $ be the set 
$\{ A\in \mbox{Con}(F(G_{\tau }))\mid A\cap K_{0}\neq \emptyset \} .$ 
By \cite[Theorem 3.15-4]{Splms10}, 
for each $v=1,\ldots u $, there exists an element $h_{v}\in G_{\tau }$  
such that $\sup _{z\in K_{0}\cap B_{v}}\| D(h_{v})_{z}\| _{h}<1.$ 
We may assume that there exists a $k'\in \NN $ with $k|k'$ such that 
for each $v$, the element $h_{v}$ is a product of 
$k'$-elements of $\G_{\tau }.$ Let $m=2k'.$ Then this $m$  is the desired number. 
\end{proof} 
\begin{lem}
\label{l:sphypm}
Let $\Lambda \in \emCpt(\emRat) $ and let $G=\langle \Lambda \rangle .$ 
Suppose that $\sharp (J(G))\geq 3.$ 
For each element $A\in \mbox{{\em Con}}(F(G))$, we take the hyperbolic metric in $A.$ 
Let $K$ be a compact subset of $F(G).$ 
Then, there exists a positive constant $C_{K}$ such that 
for each $g\in G$ and for each $z\in K$, 
$\| Dg_{z}\| _{s}/\| Dg_{z}\| _{h}\leq C_{K}.$ 
\end{lem}
\begin{proof}
By conjugating $G$ by an element of Aut$(\CCI )$, we may assume that 
$\infty \in J(G).$ For each $U\in \mbox{Con}(F(G))$, let 
$\rho _{U}=\rho _{U}(z)|dz|$ be the hyperbolic metric on $U.$ 
Since $G$ is generated by a compact subset of Rat, \cite{St3} implies that 
$J(G)$ is uniformly perfect (for the definition of uniform perfectness, see \cite{St3} and \cite{BP}). 
Therefore, by \cite{BP}, there exists a constant $C_{1}\geq 1$ such that 
for each $U\in \mbox{Con}(F(G))$ and for each $z\in U$, 
$C_{1}^{-1}\frac{1}{d_{e}(z,\partial U)}\leq \rho _{U}(z)\leq C_{1}\frac{1}{d_{e}(z,\partial U)} $, 
where $d_{e}(z,\partial U):= \inf \{ |z-w|\mid w\in \partial U\cap \CC \} .$ 
Let $z_{0}\in J(G)$ be a point. Let $g\in G$ and let $z\in K$. Let 
$U,V\in \mbox{Con}(F(G))$ be such that $z\in U$ and $g(z)\in V.$ 
Then 
\begin{align*}
\| Dg_{z}\| _{s}/\| Dg_{z}\| _{h}  & =  \frac{\sqrt{1+|z|^{2}}}{\sqrt{1+|g(z)|^{2}}}\frac{\rho _{U}(z)}{\rho _{V}(g(z))}
                                \leq \frac{\sqrt{1+|z|^{2}}}{\sqrt{1+|g(z)|^{2}}}C_{1}^{2}
                               \frac{d_{e}(g(z),\partial V)}{d_{e}(z,\partial U)}\\ 
&                                \leq \frac{\sqrt{1+|z|^{2}}}{\sqrt{1+|g(z)|^{2}}}C_{1}^{2}
                               \frac{|z_{0}|+|g(z)|}{d_{e}(z,\partial U)}.  
\end{align*}
Therefore the statement of our lemma holds. 
\end{proof}

\begin{lem}
\label{l:suptheta}
Under the notations and assumptions of Lemma~\ref{l:supai}, 
let $j\in \{ 1,\ldots, r_{L}\} .$  
For each $\alpha \in (0,1)$, let 
$\theta _{\alpha }:=\sup _{i\in I_{j}}\int _{(\emRat )^{\NN }}\sup _{z\in A_{i}}\{ \| D(\g _{m,1})_{z}\| _{h}^{\alpha }\} 
d\tilde{\tau }(\g ) (<1)$, where $m$ is the number in Lemma~\ref{l:supai}. 
Then, we have the following. 
\begin{enumerate}
\item[{\em (1)}] For each $n\in \NN $, 
$\sup _{i\in I_{j}}\int _{(\emRat )^{\NN }}\sup _{z\in A_{i}}\{ \| D(\g _{nm,1})_{z}\| _{h}^{\alpha }\}
\leq \theta _{\alpha }^{n}.$ 
\item[{\em (2)}] 
Let $i\in I_{j}$ and let $K$ be a non-empty compact subset of $A_{i}.$ 
Then there exists a constant $\tilde{C}_{K}\geq 1$ such that 
for each $\alpha \in (0,1)$, for each $\varphi \in C^{\alpha }(\CCI )$, for each $z,w\in K$, and 
for each $n\in \NN $, 
$|M_{\tau }^{nm}(\varphi )(z)-M_{\tau }^{nm}(\varphi )(w)|\leq \| \varphi \| _{\alpha }
\theta _{\alpha }^{n}\tilde{C}_{K}d(z,w)^{\alpha }.$  
\end{enumerate}
\end{lem}
\begin{proof}
Let $i\in I_{j}$ and let $\alpha \in (0,1).$  
Then we have 
\begin{align*}
\    & \int _{\G _{\tau }^{\NN }}\sup _{z\in A_{i}}\| D(\g _{nm,1})_{z}\| _{h}^{\alpha } d\tilde{\tau }(\g )\\ 
\leq & \sum _{k\in I_{j}}\int _{\{ \g \in \G _{\tau }^{\NN }\mid \g _{(n-1)m,1}(A_{i})\subset A_{k}\} } 
\sup _{z\in A_{i}}\{ \| D(\g _{nm,(n-1)m+1})_{\g _{(n-1)m,1}(z)}\| _{h}^{\alpha }\cdot \| D(\g _{(n-1)m,1})_{z}\| _{h}^{\alpha }\} 
d\tilde{\tau }(\g )\\ 
\leq &  \sum _{k\in I_{j}}\theta _{\alpha }\int _{\{ \g \in \G _{\tau }^{\NN }\mid \g _{(n-1)m,1}(A_{i})\subset A_{k}\} }
\sup _{z\in A_{i}}\| D(\g _{(n-1)m,1})_{z}\| _{h} d\tilde{\tau }(\g )\\ 
=    & \theta _{\alpha }\int _{\G _{\tau }^{\NN }}\sup _{z\in A_{i}}\| D(\g _{(n-1)m,1})_{z}\| _{h}d\tilde{\tau }(\g ). 
\end{align*} 
Therefore, statement (1) of our lemma holds. 

We now prove statement (2) of our lemma. 
Let $\tilde{K}$ be a compact subset of $A_{i}$ such that 
for each $a,b\in K$, the geodesic arc between $a$ and $b$ with respect to the hyperbolic metric on $A_{i}$ 
is included in $\tilde{K}.$ 
Let $C_{\tilde{K}}$ be the number obtained in Lemma~\ref{l:sphypm} with $\Lambda =\G _{\tau }.$ 
Let $\tilde{C}_{K}:=C_{\tilde{K}}.$ 
Let $\alpha \in (0,1)$, $\varphi \in C^{\alpha }(\CCI )$ and let $z,w\in K.$ 
Let $n\in \NN .$ 
Then we obtain 
\begin{align*}
|M_{\tau }^{nm}(\varphi )(z)-M_{\tau }^{nm}(\varphi )(w)|
& \leq \int _{\G _{\tau }^{\NN }}|\varphi (\g _{nm,1}(z))-\varphi (\g _{nm,1}(w))| d\tilde{\tau }(\g )\\ 
& \leq \int _{\G _{\tau }^{\NN }}\| \varphi \| _{\alpha }d(\g _{nm,1}(z),\g _{nm,1}(w))^{\alpha }d\tilde{\tau }(\g )\\ 
& \leq \| \varphi \| _{\alpha }\int _{\G _{\tau }^{\NN }}\tilde{C}_{K}^{\alpha }\sup _{a\in A_{i}}\{ \| D(\g _{nm,1})_{a}\| _{h}^{\alpha }\} 
            d(z,w)^{\alpha } d\tilde{\tau }(\g )\leq \| \varphi \| _{\alpha }\theta _{\alpha }^{n}\tilde{C}_{K}d(z,w)^{\alpha }. 
\end{align*}
Therefore, statement (2) of our lemma holds. 
\end{proof}
We now prove Theorem~\ref{t:kjemfhf}.\\ 
\noindent {\bf Proof of Theorem~\ref{t:kjemfhf}:} 
Let $L\in \Min(G_{\tau },\CCI ).$ Let 
$r_{L}:=\dim _{\CC }(\mbox{LS}({\cal U}_{f,\tau }(\CCI ))).$ 
By using the notation in the proof of Theorem~\ref{t:msmtaust}, 
let $\{ L_{j}\} _{j=1}^{r_{L}}=\Min (G_{\tau }^{r_{L}},L).$ 
For each $j\in \{ 1,\ldots ,r_{L}\} $,  
let $W_{L,j}:=\bigcup _{A\in \mbox{Con}(F(G_{\tau })):A\cap L_{j}\neq \emptyset }A.$ 
For each A$\in \mbox{Con}(W_{L,j})$, we take the hyperbolic metric on $A.$ 
Let $H_{j}:=d_{h}(L_{j},1)$ be the $1$-neighborhood of $L_{j}$ in $W_{L,j}$ with respect 
to the hyperbolic metric (see Definition~\ref{d:opensethyp}). 
Let $\{ A_{i}\} _{i=1}^{t}=\{ A\in \mbox{Con}(F(G_{\tau }))\mid A\cap L_{j}\neq \emptyset \} .$ 
Let $H_{j,i}:= H_{j}\cap A_{i}$ and $L_{j,i}:=L_{j}\cap A_{i}.$ 
By Lemma~\ref{l:suptheta}, 
there exists a family $\{ D_{0,\alpha }\} _{\alpha \in (0,1)} $ of 
positive constants, a family $\{ D_{1,\alpha }\} _{\alpha \in (0,1)}$ of positive constants,
 and  
a family $\{ \lambda _{1,\alpha }\} _{\alpha \in (0,1)}\subset (0,1)$ 
such that for each $\alpha \in (0,1)$, for each $L\in \Min(G_{\tau },\CCI )$, 
for each $i$, for each $j$, for each $\g \in \G _{\tau }^{\NN }$, 
for each $z,w\in H_{j,i}$, for each $n\in \NN $ and for each $\varphi \in C^{\alpha }(\CCI )$, 
\begin{equation}
\label{eq:d0d1}
|M_{\tau }^{n}(\varphi )(z)-M_{\tau }^{n}(\varphi )(w)|\leq D_{0,\alpha }\lambda _{1,\alpha }^{n}
\| \varphi \| _{\alpha } d(z,w)^{\alpha }\leq D_{1,\alpha }\lambda _{1,\alpha }^{n}\| \varphi \| _{\alpha }.
\end{equation}
For each subset $B$ of $\CCI $ and for each bounded function $\psi :B\rightarrow \CC $, 
we set $\| \psi \| _{B}:=\sup _{z\in B}|\psi (z)|.$ 
For each $i=1,\ldots ,t$, let $x_{i}\in L_{j,i}$ be a point. 
Let $\varphi \in C^{\alpha }(\CCI )$. By (\ref{eq:d0d1}), we obtain   
$\sup _{z\in H_{j,i}}|M_{\tau }^{nr_{L}}(\varphi )(z)-M_{\tau }^{nr_{L}}(\varphi )(x_{i})|\leq 
D_{1,\alpha }\lambda _{1,\alpha }^{n}\| \varphi \| _{\alpha }$ for each $i, j, n.$  
Therefore, for each $j$ and for each $l,n \in \NN $, 
\begin{equation}
\label{eq:d1l1a}
\| M_{\tau }^{lr_{L}}(M_{\tau }^{nr_{L}}(\varphi )-\sum _{i=1}^{t}M_{\tau }^{nr_{L}}(\varphi )(x_{i})\cdot 1_{H_{j,i}})\| _{H_{j}}
\leq D_{1,\alpha }\lambda _{1,\alpha }^{n}\| \varphi \| _{\alpha }. 
\end{equation}
We now consider $M_{\tau }^{r_{L}}:C_{H_{j}}(H_{j})\rightarrow C_{H_{j}}(H_{j}).$ 
We have $\dim _{\CC }(C_{H_{j}}(H_{j}))<\infty .$ 
Moreover, by the argument in the proof of Theorem~\ref{t:msmtaust}, 
$M_{\tau }^{r_{L}}:C_{H_{j}}(H_{j})\rightarrow C_{H_{j}}(H_{j})$ has 
exactly one unitary eigenvalue $1$, and has exactly one unitary eigenvector 
$1_{H_{j}}.$ Therefore, there exists a constant $\lambda _{2}\in (0,1)$ 
and a constant $D_{2}>0$, each of which depends only on $\tau $ 
and does not depend on $\alpha $ and $\varphi $, such that 
for each $l\in \NN $, 
\begin{align}
\label{eq:d2la2}
\ & \| M_{\tau }^{lr_{L}}(\sum _{i=1}^{t}M_{\tau }^{nr_{L}}(\varphi )(x_{i})1_{H_{j,i}})
-\lim _{m\rightarrow \infty }M_{\tau }^{mr_{L}}(\sum _{i=1}^{t}
M_{\tau }^{nr_{L}}(\varphi )(x_{i})1_{H_{j,i}})\| _{H_{j}} \notag \\  
\leq & D_{2}\lambda _{2}^{l}\| \sum _{i=1}^{t}M_{\tau }^{nr_{L}}(\varphi )(x_{i})1_{H_{j,i}}\| _{H_{j}}
\leq D_{2}\lambda _{2}^{l}t\| \varphi \| _{\CCI }.
\end{align} 
Since $\lambda _{2}$ does not depend on $\alpha $, we may assume that 
for each $\alpha \in (0,1)$, 
$\lambda _{1,\alpha }\geq \lambda _{2}.$ 
From (\ref{eq:d1l1a}) and (\ref{eq:d2la2}), 
it follows that for each $n\in \NN $ and for each $l_{1},l_{2}\in \NN $ with $l_{1},l_{2}\geq n$, 
\begin{align*}
\    & \| M_{\tau }^{(l_{1}+n)r_{L}}(\varphi )-M_{\tau }^{(l_{2}+n)r_{L}}(\varphi )\| _{H_{j}}\\ 
\leq & \| M_{\tau }^{(l_{1}+n)r_{L}}(\varphi )-M_{\tau }^{l_{1}r_{L}}(\sum _{i=1}^{t}
       M_{\tau }^{nr_{L}}(\varphi )(x_{i})1_{H_{j,i}})\| _{H_{j}}\\ 
\    & \ + \| M_{\tau }^{l_{1}r_{L}}(\sum _{i=1}^{t}M_{\tau }^{nr_{L}}(\varphi )(x_{i})1_{H_{j,i}})
         -\lim _{m\rightarrow \infty }M_{\tau }^{mr_{L}}(\sum _{i=1}^{t}M_{\tau }^{nr_{L}}(\varphi )(x_{i})1_{H_{j,i}})\| _{H_{j}}\\ 
\    & \ + \| \lim _{m\rightarrow \infty }M_{\tau }^{mr_{L}}(\sum _{i=1}^{t}M_{\tau }^{nr_{L}}(\varphi )(x_{i})1_{H_{j,i}})
      - M_{\tau }^{l_{2}r_{L}}(\sum _{i=1}^{t}M_{\tau }^{nr_{L}}(\varphi )(x_{i})1_{H_{j,i}})\| _{H_{j}}\\ 
\    & \ +\| M_{\tau }^{l_{2}r_{L}}(\sum _{i=1}^{t}M_{\tau }^{nr_{L}}(\varphi )(x_{i})1_{H_{j,i}})
      -M_{\tau }^{(l_{2}+n)r_{L}}(\varphi )\| _{H_{j}}\\ 
\leq & 2D_{1,\alpha }\lambda _{1,\alpha }^{n}\| \varphi \| _{\alpha }+D_{2}\lambda _{2}^{l_{1}}t\| \varphi \| _{\alpha }
      +D_{2}\lambda _{2}^{l_{2}}t\| \varphi \| _{\alpha } 
      \leq (2D_{1,\alpha }+2D_{2}t)\lambda _{1,\alpha }^{n}\| \varphi \| _{\alpha }.      
\end{align*} 
Letting $l_{1}\rightarrow \infty $, we obtain that 
for each $l_{2}\in \NN $ with $l_{2}\geq n$, 
$\| \pi _{\tau }(\varphi )-M_{\tau }^{(l_{2}+n)r_{L}}(\varphi )\| _{H_{j}}\leq 
(2D_{1,\alpha }+2D_{2}t)\lambda _{1,\alpha }^{n}\| \varphi \| _{\alpha }.$ 
In particular, for each $n\in \NN $, 
$\| \pi _{\tau }(\varphi )-M_{\tau }^{2nr_{L}}(\varphi )\| _{H_{j}}\leq 
(2D_{1,\alpha }+2D_{2}t)\lambda _{1,\alpha }^{n}\| \varphi \| _{\alpha }.$ 
Therefore,  for each $n\in \NN $, 
\begin{equation}
\label{eq:pitauhj}
\| \pi _{\tau }(\varphi )-M_{\tau }^{nr_{L}}(\varphi )\| _{H_{j}}\leq 
(2D_{1,\alpha }+2D_{2}t)\lambda _{1,\alpha }^{-1/2}(\lambda _{1,\alpha }^{1/2})^{n}\| M_{\tau }^{r_{L}}\| _{\alpha }\| \varphi \| _{\alpha }, 
\end{equation}  
where $\| M_{\tau }^{r_{L}}\| _{\alpha }$ denotes the operator norm of 
$M_{\tau }^{r_{L}}:C^{\alpha }(\CCI )\rightarrow C^{\alpha }(\CCI ).$ 
Let $U:=\bigcup _{L,j}H_{j}$ and let $r:=\prod _{L}r_{L}.$ 
From the above arguments, it follows that there exists a family $\{ D_{3,\alpha }\} _{\alpha \in (0,1)} $ 
of positive constants and a family $\{ \lambda _{3,\alpha }\} _{\alpha \in (0,1)} \subset (0,1)$ 
such that for each $\alpha \in (0,1)$, for each $\varphi \in C^{\alpha }(\CCI )$ and for each 
$n\in \NN $, 
\begin{equation}
\label{eq:UleqD3}
\| \pi _{\tau }(\varphi )-M_{\tau }^{rn}(\varphi )\| _{\overline{U}}\leq 
D_{3,\alpha }\lambda _{3,\alpha }^{n}\| \varphi \| _{\alpha }. 
\end{equation} 
By \cite[Theorem 3.15-5]{Splms10}, for each $z\in \CCI $, there exists a map $g_{z}\in G_{\tau }$ and a compact disk neighborhood 
$U_{z}$ of $z$ such that 
$g_{z}(U_{z})\subset U.$ 
Since $\CCI $ is compact, there  exists a finite family 
$\{ z_{j}\} _{j=1}^{s}\subset \CCI $ such that 
$\bigcup _{j=1}^{s}\mbox{int}(U_{z_{j}})=\CCI .$ 
Since $G_{\tau }(U)\subset U$, we may assume that 
there exists a $k$ such that 
for each $j=1,\ldots ,s$, there exists an element $\beta ^{j}=(\beta _{1}^{j},\ldots ,\beta _{k}^{j})\in \G _{\tau }^{k}$ 
with $g_{z_{j}}=\beta _{k}^{j}\circ \cdots \circ \beta _{1}^{j}.$ 
We may also assume that $r|k.$ 
For each $j=1,\ldots ,s$, let $V_{j}$ be a compact neighborhood of $\beta ^{j}$ in $\G _{\tau }^{k}$ such that 
for each $\zeta =(\zeta _{1},\ldots ,\zeta _{k})\in V_{j}$, 
$\zeta _{k}\circ \cdots \circ \zeta _{1}(U_{z_{j}})\subset U.$ Let 
$a:=\max \{ \tau ^{k}(\G _{\tau }^{k}\setminus V_{j})\mid j=1,\ldots ,s\} \in [0,1).$ 
Let $C_{1}:=2\max \{ \max \{ \| D(\zeta _{k}\circ \cdots \circ \zeta _{1})_{z}\| _{s}\mid 
(\zeta _{1},\ldots ,\zeta _{k})\in \G _{\tau }^{k},z\in \CCI \} ,1\} .$ 
Let $\alpha _{1}\in (0,1)$ be such that 
$aC_{1}^{\alpha _{1}}<1$ and $\mbox{LS}({\cal U}_{f,\tau }(\CCI ))\subset C^{\alpha _{1}}(\CCI ).$ 
 Let $C_{2}>0$ be a constant such that 
 for each $z\in \CCI $ there exists a $j\in \{ 1,\ldots ,s\} $ with  
$B(z,C_{2})\subset U_{z_{j}}.$ 
Let $n\in \NN .$ 
Let $z_{0}\in \CCI $ be any point. 
Let $i_{0}\in \{ 1,\ldots ,s\} $ be such that $B(z_{0},C_{2})\in U_{z_{i_{0}}}.$ 
Let $A(0):=\{ \g \in \G _{\tau }^{\NN }\mid (\g _{1},\ldots ,\g _{k})\in V_{i_{0}}\} $ 
and let $B(0):= \{ \g\in \G _{\tau }^{\NN }\mid (\g _{1},\ldots ,\g _{k})\not\in V_{i_{0}}\} .$ 
Inductively, for each $j=1,\ldots ,n-1$, let 
$A(j):=\{ \g \in B(j-1)\mid \exists i \mbox{ s.t. }B(\g _{kj,1}(z_{0}),C_{2})\subset U_{z_{i}}, 
(\g _{kj+1},\ldots ,\g _{kj+k})\in V_{i}\} $ and let 
$B(j):= B(j-1)\setminus A(j).$ Then for each $j=1,\ldots ,n-1$, 
$\tilde{\tau }(B(j))\leq a\tilde{\tau }(B(j-1))\leq \cdots \leq a^{j+1}$ and 
$\tilde{\tau }(A(j))\leq \tilde{\tau }(B(j-1))\leq a^{j}.$  
Moreover, we have 
$\G _{\tau }^{\NN }=\amalg _{j=0}^{n-1}A(j)\amalg B(n-1).$ 
Therefore, we obtain that 
\begin{align}
\label{eq:mtauknz0}
\    & |M_{\tau }^{kn}(\varphi )(z_{0})-\pi _{\tau }(\varphi )(z_{0})|=
       |M_{\tau }^{kn}(\varphi )(z_{0})-M_{\tau }^{kn}(\pi _{\tau }(\varphi ))(z_{0})| \notag \\ 
\leq & \left|\sum _{j=0}^{n-1}\int _{A(j)}(\varphi (\g _{kn,1}(z_{0}))
        -\pi _{\tau }(\varphi )(\g _{kn,1}(z_{0})))d\tilde{\tau }(\g )\right| \notag \\ 
\    & \ +\left|\int _{B(n-1)}(\varphi (\g _{kn,1}(z_{0}))
        -\pi _{\tau }(\varphi )(\g _{kn,1}(z_{0})))d\tilde{\tau }(\g )\right|. 
\end{align}
For each $j=0,\ldots ,n-1$, there exists a Borel subset $A'(j)$ of 
$\G _{\tau }^{(k+1)j}$ such that 
$A(j)=A'(j)\times \G _{\tau }\times \G _{\tau }\times \cdots .$ 
Hence, by (\ref{eq:UleqD3}), we obtain that for each $\alpha \in (0,1)$ 
and for each $\varphi \in C^{\alpha }(\CCI )$, 
\begin{align}
\label{eq:apjph}
\    & \left| \int _{A(j)}(\varphi (\g _{kn,1}(z_{0}))-\pi _{\tau }(\varphi )(\g _{kn,1}(z_{0}))) 
       d\tilde{\tau }(\g )\right| \notag \\ 
= & \left|  \int _{A'(j)}(M_{\tau }^{k(n-j-1)}(\varphi )(\g _{k(j+1)}\circ \cdots \circ \g _{1}(z_{0}))
              -\pi _{\tau }(\varphi )(\g _{k(j+1)}\circ \cdots \circ \g _{1}(z_{0}))) d\tau (\g _{k(j+1)})\cdots 
              d\tau (\g _{1})\right| \notag \\ 
\leq & D_{3,\alpha }\lambda _{3,\alpha }^{n-j-1}\| \varphi \| _{\alpha }\tilde{\tau }(A(j))
\leq D_{3,\alpha }\lambda _{3,\alpha }^{n-j-1}a^{j}\| \varphi \|_{\alpha }.    
\end{align}
By (\ref{eq:mtauknz0}) and (\ref{eq:apjph}), 
it follows that 
\begin{align*}
\    & |M_{\tau }^{kn}(\varphi )(z_{0})-\pi _{\tau }(\varphi )(z_{0})|
\leq  \sum _{j=0}^{n-1}D_{3,\alpha }\lambda _{3,\alpha }^{n-j-1}a^{j}\| \varphi \| _{\alpha }
        +a^{n}(\| \varphi \| _{\infty }+\| \pi _{\tau }(\varphi )\| _{\infty })\\ 
\leq & \left( D_{3,\alpha }n(\max \{ \lambda _{3,\alpha }, a\} )^{n-1} 
       +a^{n}(1+\| \pi _{\tau }\| _{\infty } )\right) \| \varphi \| _{\alpha },  
\end{align*} 
where $\| \pi _{\tau }\| _{\infty }$ denotes the operator norm of 
$\pi _{\tau }:(C(\CCI ),\| \cdot \| _{\infty })\rightarrow (C(\CCI ),\| \cdot \| _{\infty }).$ 
For each $\alpha \in (0,1)$, let 
$\zeta _{\alpha }:=\frac{1}{2}(1+\max \{ \lambda _{3,\alpha },a\} )<1.$ 
From these arguments, it follows that there exists a family $\{ C_{3,\alpha }\} _{\alpha \in (0,1)}$ 
of positive constants such that for each $\alpha \in (0,1)$,  
for each $\varphi \in C^{\alpha }(\CCI )$ and for each $n\in \NN $,  
\begin{equation}
\label{eq:mtauC3}
\| M_{\tau }^{kn}(\varphi )-\pi _{\tau }(\varphi )\| _{\infty }
\leq C_{3,\alpha }\zeta _{\alpha }^{n}\| \varphi \| _{\alpha }. 
\end{equation} 
For the rest of the proof, let 
$\alpha \in (0,\alpha _{1}).$ Let $\eta _{\alpha }:=
\max \{ \lambda _{1,\alpha },aC_{1}^{\alpha }\} \in (0,1).$ 
Let $z,z_{0}\in \CCI .$ 
If $d(z,z_{0})\geq C_{1}^{-1}C_{2}$, then 
\begin{equation}
\frac{| M_{\tau }^{kn}(\varphi )(z)-M_{\tau }^{kn}(\varphi )(z_{0})
-(\pi _{\tau }(\varphi )(z)-\pi _{\tau }(\varphi )(z_{0}))|}{d(z,z_{0})^{\alpha }}
\leq 2C_{3,\alpha }\zeta _{\alpha }^{n}\| \varphi \| _{\alpha } (C_{1}C_{2}^{-1})^{\alpha }.
\end{equation}
We now suppose that there exists an $m\in \NN $ such that 
$C_{1}^{-m-1}C_{2}\leq d(z,z_{0})<C_{1}^{-m}C_{2}.$ 
Then for each $\g \in \G _{\tau }^{\NN }$ and for 
each $j=1,\ldots ,m$, 
\begin{equation}
\label{eq:dgkjC2}
d(\g _{kj,1}(z),\g _{kj,1}(z_{0}))<C_{2}. 
\end{equation}  
Let $n\in \NN $. Let $\tilde{m}:=\min \{ n,m\} .$ 
Let $i_{0}\in  \{ 1,\ldots ,s\} $ be such that $B(z_{0},C_{2})\subset U_{z_{i_{0}}}$ and 
let $A(0),B(0),\ldots, A(\tilde{m}-1),B(\tilde{m}-1)$ be as before. 
Let $\varphi \in C^{\alpha }(\CCI )$ and let $n\in \NN $. Then 
we have 
\begin{align}
\label{eq:a0am}
\    & |M_{\tau }^{kn}(\varphi )(z)-M_{\tau }^{kn}(\varphi )(z_{0})
      -(\pi _{\tau }(\varphi )(z)-\pi _{\tau }(\varphi )(z_{0}))|\notag \\ 
\leq & \left| \sum _{j=0}^{\tilde{m}-1}\int _{A(j)}[\varphi (\g _{kn,1}(z))-\varphi (\g _{kn,1}(z_{0}))
        -(\pi _{\tau }(\varphi )(\g _{kn,1}(z))-\pi _{\tau }(\varphi )(\g _{kn,1}(z_{0})))]
        d\tilde{\tau }(\g )\right| \notag \\ 
\    & \ + \left| \int _{B(\tilde{m}-1)}[\varphi (\g _{kn,1}(z))-\varphi (\g _{kn,1}(z_{0}))
        -(\pi _{\tau }(\varphi )(\g _{kn,1}(z))-\pi _{\tau }(\varphi )(\g _{kn,1}(z_{0})))]
        d\tilde{\tau }(\g )\right| . 
\end{align} 
Let $A'(j)$ be as before. By (\ref{eq:d0d1}) and (\ref{eq:dgkjC2}), 
we obtain that for each $j=0,\ldots ,\tilde{m}-1$, 
\begin{align}
\label{eq:ajineq}
\ & \left| \int _{A(j)}[\varphi (\g _{kn,1}(z))-\varphi (\g _{kn,1}(z_{0}))
        -(\pi _{\tau }(\varphi )(\g _{kn,1}(z))-\pi _{\tau }(\varphi )(\g _{kn,1}(z_{0})))]
        d\tilde{\tau }(\g )\right|\notag \\
= & \left| \int _{A(j)}(\varphi (\g _{kn,1}(z))-\varphi (\g _{kn,1}(z_{0}))d\tilde{\tau }(\g )\right| 
    \notag \\ 
= & \left| \int _{A'(j)}[M_{\tau }^{k(n-j-1)}(\varphi )(\g _{k(j+1),1}(z))-
           M_{\tau }^{k(n-j-1)}(\varphi )(\g _{k(j+1),1}(z_{0}))] 
           d\tau (\g _{k(j+1)})\cdots d\tau (\g _{1})\right| \notag \\ 
\leq & \int _{A'(j)}D_{0,\alpha }d(\g _{k(j+1),1}(z),\g _{k(j+1),1}(z_{0}))^{\alpha }
       \lambda _{1,\alpha }^{n-j-1} \| \varphi \| _{\alpha } 
       d\tau (\g _{k(j+1)})\cdots d\tau (\g _{1}) \notag \\ 
\leq & D_{0,\alpha }C_{1}^{\alpha (j+1)}d(z,z_{0})^{\alpha }
       \lambda _{1,\alpha }^{n-j-1}a^{j}\| \varphi \| _{\alpha }\notag \\ 
\leq & D_{0,\alpha }C_{1}^{\alpha }\eta _{\alpha }^{n-1}\| \varphi \| _{\alpha }
      d(z,z_{0})^{\alpha }.   
\end{align}
Let $B'(\tilde{m}-1)$ be a Borel subset of $\G _{\tau }^{k\tilde{m}}$ such that 
$B(\tilde{m}-1)=B'(\tilde{m}-1)\times \G _{\tau }\times \G _{\tau }\times \cdots .$ 
We now consider the following two cases. Case (I): $\tilde{m}=m$. Case (II): 
$\tilde{m}=n.$ 

Suppose we have Case (I). 
Then by (\ref{eq:mtauC3}), we obtain that  
\begin{align}
\label{eq:Bm-1c1}
\    & \left| \int _{B(\tilde{m}-1)}[\varphi (\g _{kn,1}(z))-\varphi (\g _{kn,1}(z_{0}))
        -(\pi _{\tau }(\varphi )(\g _{kn,1}(z))-\pi _{\tau }(\varphi )(\g _{kn,1}(z_{0})))]
        d\tilde{\tau }(\g )\right| \notag \\ 
\leq & \int _{B'(m-1)}| M_{\tau }^{k(n-m)}(\g _{km}\circ \cdots \circ \g_{1}(z))
        -\pi _{\tau }(\varphi )(\g _{km}\circ \cdots \circ \g_{1}(z))| 
       d\tau (\g _{km})\cdots d\tau (\g _{1})\notag \\
\    & \ +\int _{B'(m-1)}  | M_{\tau }^{k(n-m)}(\g _{km}\circ \cdots \circ \g_{1}(z_{0}))
        -\pi _{\tau }(\varphi )(\g _{km}\circ \cdots \circ \g_{1}(z_{0}))| 
       d\tau (\g _{km})\cdots d\tau (\g _{1})\notag \\
\leq & 2C_{3,\alpha }\zeta _{\alpha }^{n-m}\| \varphi \| _{\alpha }a^{m} 
       \leq  2C_{3,\alpha }\zeta _{\alpha }^{n-m}\| \varphi \| _{\alpha }a^{m}
       \cdot (C_{1}^{m+1}C_{2}^{-1}d(z,z_{0}))^{\alpha }\notag \\ 
=    & 2C_{3,\alpha }\zeta _{\alpha }^{n-m}(aC_{1}^{\alpha })^{m}(C_{1}C_{2}^{-1})^{\alpha }
       \| \varphi \| _{\alpha } d(z,z_{0})^{\alpha }  
       \leq  2C_{3,\alpha }(C_{1}C_{2}^{-1})^{\alpha }\zeta _{\alpha }^{n-m}\eta _{\alpha }^{m}
       \| \varphi \| _{\alpha }d(z,z_{0})^{\alpha }.          
\end{align}
We now suppose we have Case (II). 
Since $\mbox{LS}({\cal U}_{f,\tau }(\CCI ))\subset C^{\alpha }(\CCI )$, 
we obtain  
\begin{align}
\label{eq:Bm-1c2}
\    &   \left| \int _{B(\tilde{m}-1)}[\varphi (\g _{kn,1}(z))-\varphi (\g _{kn,1}(z_{0}))
        -(\pi _{\tau }(\varphi )(\g _{kn,1}(z))-\pi _{\tau }(\varphi )(\g _{kn,1}(z_{0})))]
        d\tilde{\tau }(\g )\right| \notag \\ 
\leq & \int _{B(n-1)}| \varphi (\g _{kn,1}(z))-\varphi (\g _{kn,1}(z_{0}))| d\tilde{\tau }(\g ) 
        + \int _{B(n-1)}| \pi _{\tau }(\varphi )(\g _{kn,1}(z))
                      -\pi _{\tau }(\varphi )(\g _{kn,1}(z_{0}))| d\tilde{\tau }(\g )\notag \\ 
\leq & C_{1}^{\alpha n}d(z,z_{0})^{\alpha }a^{n}\| \varphi \| _{\alpha }
       +C_{1}^{\alpha n}d(z,z_{0})^{\alpha }a^{n}\| \pi _{\tau }(\varphi )\| _{\alpha }\notag \\ 
\leq & C_{1}^{\alpha n}a^{n}(1+E_{\alpha })\| \varphi \| _{\alpha }d(z,z_{0})^{\alpha },          
\end{align}
where $E_{\alpha }$ denotes the number in Theorem~\ref{t:utauca}. 
Let $\xi _{\alpha }:= \frac{1}{2}(\max \{ \zeta _{\alpha },\eta _{\alpha }\} +1) \in (0,1).$ 
Combining (\ref{eq:a0am}), (\ref{eq:ajineq}), (\ref{eq:Bm-1c1}) and (\ref{eq:Bm-1c2}), 
it follows that there exists a constant $C_{4,\alpha }>0$ such that 
for each $\varphi \in C^{\alpha }(\CCI )$, 
\begin{equation}
\label{eq:mtauval}
 |M_{\tau }^{kn}(\varphi )(z)-M_{\tau }^{kn}(\varphi )(z_{0})
-(\pi _{\tau }(\varphi )(z)-\pi _{\tau }(\varphi )(z_{0}))| 
\leq C_{4,\alpha }\xi _{\alpha }^{n}\| \varphi \| _{\alpha }d(z,z_{0})^{\alpha }.
\end{equation}  
Let $C_{5,\alpha }=C_{3,\alpha }+C_{4,\alpha }$. 
By (\ref{eq:mtauC3}) and (\ref{eq:mtauval}), we obtain that 
for each $\varphi \in C^{\alpha }(\CCI )$ and for each $n\in \NN $, 
\begin{equation}
\label{eq:mtauC5}
\| M_{\tau }^{kn}(\varphi )-\pi _{\tau }(\varphi )\| _{\alpha }
\leq C_{5,\alpha }\xi _{\alpha }^{n}\| \varphi \| _{\alpha }. 
\end{equation}  
From this, statement (3) of our theorem holds. 

Let $\psi \in C^{\alpha }(\CCI ).$ Setting $\varphi =\psi -\pi _{\tau }(\psi )$, 
by (\ref{eq:mtauC5}), we obtain that statement (2) of our theorem holds. 
Statement (4) of our theorem follows from Theorem~\ref{t:utauca}. 
Statement (1) follows from statements (2).  

Thus, we have proved Theorem~\ref{t:kjemfhf}. 
\qed 

We now prove Theorem~\ref{t:kjemfsp}.\\ 
{\bf Proof of Theorem~\ref{t:kjemfsp}:} 
Let $A:= \{ z\in \CC \mid | z |\leq \lambda \} \cup {\cal U}_{v,\tau }(\CCI ).$ 
Let $\zeta \in \CC \setminus A.$ 
Then by Theorem~\ref{t:kjemfhf}, 
 $\sum _{n=0}^{\infty }
\frac{M_{\tau }^{n}}{\zeta ^{n+1}}(I-\pi _{\tau })$ 
converges in the space of bounded linear operators on $C^{\alpha }(\CCI )$ endowed 
with the operator norm. 
Let $\Omega := (\zeta I-M_{\tau })|_{\LSfc}^{-1}\circ \pi _{\tau }+\sum _{n=0}^{\infty }
\frac{M_{\tau }^{n}}{\zeta ^{n+1}}(I-\pi _{\tau }).$ 
Let $U_{\tau }:= \mbox{LS}({\cal U}_{f,\tau }(\CCI )).$ 
Then we have 
\begin{align*}
(\zeta I-M_{\tau })\circ \Omega  
= &  \left( (\zeta I-M_{\tau })|_{U_{\tau }}\circ \pi _{\tau }
    +(\zeta I-M_{\tau })|_{{\cal B}_{0,\tau }}\circ (I-\pi _{\tau })\right) \\ 
\ & \ \circ \left( (\zeta I-M_{\tau })|_{U_{\tau }}^{-1}\circ \pi _{\tau }
     + \sum _{n=0}^{\infty }\frac{M_{\tau }^{n}}{\zeta ^{n+1}}|_{{\cal B}_{0,\tau }}
    (I-\pi _{\tau }) \right) \\ 
= & I|_{U_{\tau }}\circ \pi _{\tau }+(\zeta I-M_{\tau })\circ 
    (\sum _{n=0}^{\infty }\frac{M_{\tau }^{n}}{\zeta ^{n+1}})\circ (I-\pi _{\tau })\\ 
= & \pi _{\tau }+(\sum _{n=0}^{\infty }\frac{M_{\tau }^{n}}{\zeta ^{n}}-
    \sum _{n=0}^{\infty }\frac{M_{\tau }^{n+1}}{\zeta ^{n+1}})\circ (I-\pi _{\tau })=I.      
\end{align*}
Similarly, we have $\Omega \circ (\zeta I-M_{\tau })=I.$ 
Therefore, statements (1) and (2) of our theorem hold. 

Thus, we have proved Theorem~\ref{t:kjemfsp}.  
\qed 

We now prove Theorem~\ref{t:kjemfsppt}.\\ 
\noindent {\bf Proof of Theorem~\ref{t:kjemfsppt}:} 
By using the method in the proofs of \cite[Lemmas 5.1, 5.2]{SU1}, 
we obtain that for each $\alpha \in (0,1)$,  
the map  
$a\in {\cal W}_{m} \mapsto M_{\tau _{a}}\in 
L(C^{\alpha }(\CCI ))$ is real-analytic, where $L(C^{\alpha }(\CCI )) $ denotes the 
Banach space of bounded linear operators on $C^{\alpha }(\CCI )$ endowed with 
the operator norm. 
Moreover, by using the method in the proof of Theorem~\ref{t:utauca}, 
we can show that for each $b\in {\cal W}_{m}$, 
there exists an $\alpha \in (0,1)$ and an open neighborhood $V_{b}$ of $b$ in ${\cal W}_{m}$ 
such that for each $a\in V_{b}$, we have $\mbox{LS}({\cal U}_{f,\tau _{a}}(\CCI ))\subset C^{\alpha }(\CCI ).$ 
In particular, $\pi _{\tau _{a}}(C^{\alpha }(\CCI ))\subset C^{\alpha }(\CCI )$ for each $a\in V_{b}.$   
Statement (1) follows from the above arguments, \cite[Theorem 3.15-10]{Splms10}, 
Theorem~\ref{t:kjemfsp} and \cite[p368-369, p212]{K}. 
We now prove statement (2). 
For each $L\in \Min(G,\CCI )$, 
let $\varphi _{L}:\CCI \rightarrow [0,1]$ be a $C^{\infty }$ function on $\CCI $ such that  
$\varphi _{L}|_{L}\equiv 1$ and such that for each 
$L'\in \Min (G,\CCI )$ with $L'\neq L$, 
$\varphi _{L}|_{L'}\equiv 0.$   
Then, 
by \cite[Theorem 3.15-15]{Splms10}, 
we have that for each $z\in \CCI $, 
$T_{L,\tau _{a}}(z)=\lim _{n\rightarrow \infty }M_{\tau _{a}}^{n}(\varphi _{L})(z).$ 
Combining this with \cite[Theorem 3.14]{Splms10}, 
we obtain $T_{L,\tau _{a}}=\lim _{n\rightarrow \infty }M_{\tau _{a}}^{n}(\varphi )$ 
in $C(\CCI ).$ By \cite[Theorem 3.15-6,8,9]{Splms10}, 
for each $a\in {\cal W}_{m}$, 
there exists a number $r\in \NN $ such that 
for each $\psi \in \mbox{LS}({\cal U}_{f,\tau }(\CCI ))$, 
$M_{\tau _{a}}^{r}(\psi )=\psi .$ Therefore, 
by \cite[Theorem 3.15-1]{Splms10}, 
$T_{L,\tau _{a}}=\lim _{n\rightarrow \infty }M_{\tau _{a}}^{nr}(\varphi _{L})
=\lim _{n\rightarrow \infty }M_{\tau _{a}}^{nr}(\varphi _{L}-\pi _{\tau _{a}}(\varphi _{L})
+\pi _{\tau _{a}}(\varphi _{L}))=\pi _{\tau _{a}}(\varphi _{L}).$ 
Combining this with statement (1) of our theorem and \cite[Theorem 3.15-1]{Splms10}, 
it is easy to see that statement (2) of our theorem holds. 

 We now prove statement (3). 
By taking the partial derivative of $M_{\tau _{a}}(T_{L,\tau _{a}}(z))=T_{L,\tau _{a}}(z)$ 
with respect to $a_{i}$, it is easy to see that 
$\psi _{i,b}$ satisfies the functional equation 
$(I-M_{\tau _{b}})(\psi _{i,b})=\zeta _{i,b}, \psi _{i,b}|_{S_{\tau _{b}}}=0.$ 
Let $\psi \in C(\CCI )$ be a solution of 
$(I-M_{\tau _{b}})(\psi )=\zeta _{i,b}, \psi |_{S_{\tau _{b}}}=0.$ 
Then for each $n\in \NN $, 
\begin{equation}
\label{eq:I-M}
(I-M_{\tau _{b}}^{n})(\psi )=\sum _{j=0}^{n-1}M_{\tau _{b}}^{j}(\zeta _{i,b}).
\end{equation}  
By the definition of $\zeta _{i,b}$,  
$\zeta _{i,b}|_{S_{\tau _{b}}}=0.$ Therefore, by \cite[Theorem 3.15-2]{Splms10}, 
$\pi _{\tau _{b}}(\zeta _{i,b})=0.$ 
Thus, denoting by $C$ and $\lambda $ the constants in Theorem~\ref{t:kjemfhf}, 
we obtain $\| M_{\tau _{b}}^{n}(\zeta _{i,b})\| _{\alpha }\leq C\lambda ^{n}\| \zeta _{i,b}\| _{\alpha }.$ 
Moreover, since $\psi |_{S_{\tau _{b}}}=0$, \cite[Theorem 3.15-2]{Splms10} implies 
$\pi _{\tau _{b}}(\psi )=0.$ Therefore, $M_{\tau _{b}}^{n}(\psi )\rightarrow 0$  
in $C(\CCI ) $ as $n\rightarrow \infty .$ Letting $n\rightarrow \infty $ in (\ref{eq:I-M}), 
we obtain that 
$\psi =\sum _{j=0}^{\infty }M_{\tau _{b}}^{j}(\zeta _{i,b}).$ 
Therefore, we have proved statement (3). 

  Thus, we have proved Theorem~\ref{t:kjemfsppt}. 
\qed  

We now prove Theorem~\ref{t:psinondiff}.\\
\noindent {\bf Proof of Theorem~\ref{t:psinondiff}:}
Statements~\ref{t:psinondiff1},\ref{t:psinondiff3},\ref{t:psinondiff4} follow from 
\cite[Theorem 3.82]{Splms10} and its proof.  
We now prove statement~\ref{t:psinondiff2}. 
By \cite[Theorem 3.82, Theorem 3.15-15]{Splms10}, 
there exists a Borel subset $A$ of $J(G)$ with $\lambda (A)=1$ such that 
for each $L\in \Min(G,\CCI )$ and for each $z\in A$,  
$\mbox{H\"{o}l}(T_{L,\tau _{p}},z)=u(h,p,\mu ).$ 
Let $z_{0}\in A$ be a point, let $L\in \Min(G,\CCI )$, and let $i\in \{ 1,\ldots ,m-1\} .$  
We consider the following three cases. 
Case 1: $\mbox{H\"{o}l}(\psi _{i,p,L},z_{0})<u(h,p,\mu ).$ 
Case 2: $\mbox{H\"{o}l}(\psi _{i,p,L},z_{0})=u(h,p,\mu ).$
Case 3: $\mbox{H\"{o}l}(\psi _{i,p,L},z_{0})>u(h,p,\mu ).$

Suppose we have Case 1. Let $z_{1}\in h_{i}^{-1}(\{ z_{0}\} ).$ 
By the functional equation 
$(I-M_{\tau _{p}})(\psi _{i,p,L})=T_{L,\tau _{p}}\circ h_{i}-T_{L,\tau _{p}}\circ h_{m}$ (see Theorem~\ref{t:kjemfsppt} (3)), 
\cite[Theorem 3.15-1,15]{Splms10}, and the assumption $h_{k}^{-1}(J(G))\cap h_{l}^{-1}(J(G))=\emptyset $ for each 
$(k,l)$ with $k\neq l$, 
there exists a neighborhood $U$ of $z_{1}$ in $\CCI $ such that 
for each $z\in U$, 
\begin{equation}
\label{eq:psindpf1}
\psi _{i,p,L}(z)-\psi _{i,p,L}(z_{1})-p_{i}(\psi _{i,p,L}(h_{i}(z))-\psi _{i,p,L}(z_{0}))=
T_{L,\tau _{p}}(h_{i}(z))-T_{L,\tau _{p}}(z_{0}).
\end{equation}
By equation (\ref{eq:psindpf1}) and the definition of the pointwise H\"{o}lder exponent, it is easy to see that 
$\mbox{H\"{o}l}(\psi _{i,p,L},z_{1})= \mbox{H\"{o}l}(\psi _{i,p,L},z_{0})<u(h,p,\mu ).$ 
We now let $z_{1}\in h_{m}^{-1}(\{ z_{0}\} ).$ Then by the similar method to the above, 
we obtain that $\mbox{H\"{o}l}(\psi _{i,p,L},z_{1})= \mbox{H\"{o}l}(\psi _{i,p,L},z_{0})<u(h,p,\mu ).$ 

 We now suppose we have Case 2. By the same method as that in Case 1, we obtain that 
 $\mbox{H\"{o}l}(\psi _{i,p,L},z_{0})= u(h,p,\mu )\leq \mbox{H\"{o}l}(\psi _{i,p,L},z_{1})$ for 
each $z_{1}\in  h_{i}^{-1}(\{ z_{0}\} )\cup h_{m}^{-1}(\{ z_{0}\} ).$

We now suppose we have Case 3. By the same method as that in Case 1 again, 
we obtain that 
$\mbox{H\"{o}l}(\psi _{i,p,L},z_{1})=u(h,p,\mu )<\mbox{H\"{o}l}(\psi _{i,p,L},z_{0})$ 
for 
each $z_{1}\in  h_{i}^{-1}(\{ z_{0}\} )\cup h_{m}^{-1}(\{ z_{0}\} ).$

Thus we have proved Theorem~\ref{t:psinondiff}. 
\qed 
\section{Examples}
\label{Examples} 
In this section, we give some examples.
\begin{ex}[Proposition 6.1 in \cite{Splms10}]
\label{ex:const}
Let $f_{1}\in {\cal P}.$ 
Suppose that int$(K(f_{1}))$ is not empty. 
Let $b\in \mbox{int}(K(f_{1}))$ be a point. 
Let $d$ be a positive integer such that 
$d\geq 2.$ Suppose that $(\deg (f_{1}),d)\neq (2,2).$ 
Then, there exists a number $c>0$ such that 
for each $\l \in \{ \l\in \Bbb{C}: 0<|\l |<c\} $, 
setting $f_{\l }=(f_{\l ,1},f_{\l ,2})=
(f_{1},\l (z-b)^{d}+b )$ and $G_{\l }:= \langle f_{1},f_{\l, 2}\rangle $, 
 we have all of the following.
\begin{itemize}
\item[{\em (a)}] 
$f_{\l }$ satisfies the open set condition with 
an open subset $U_{\l }$ of $\CCI $ (i.e., $f_{\l ,1}^{-1}(U_{\l })\cup f_{\l, 2}^{-1}(U_{\l })\subset U_{\l }$ and 
$f_{\l ,1}^{-1}(U_{\l })\cap f_{\l ,2}^{-1}(U_{\l })=\emptyset $), 
$f_{\l ,1}^{-1}(J(G_{\l }))\cap f_{\l, 2}^{-1}(J(G_{\l }))=\emptyset $, 
int$(J(G_{\l }))=\emptyset $, 
$J_{\ker }(G_{\l })=\emptyset $, 
$G_{\l }(K(f_{1}))\subset K(f_{1})\subset \mbox{int}(K(f_{\lambda ,2}))$  
and 
$\emptyset \neq K(f_{1})\subset \hat{K}(G_{\l }).$ 
\item[{\em (b)}]
If $K(f_{1})$ is connected, then  
$P^{\ast }(G_{\l })
$ is bounded in $\Bbb{C}$.
\item[{\em (c)}]
If $f_{1}$ is hyperbolic and $K(f_{1})$ is connected, 
then 
$G_{\l }$ is hyperbolic, 
$J(G_{\l } )$ is porous (for the definition of porosity, see \cite{S7}),  and 
$\dim _{H}(J(G_{\l }))<2$.  
\end{itemize} 

\end{ex} 
By Example~\ref{ex:const}, Remark~\ref{r:hypms} and \cite[Proposition 6.4]{Splms10}, 
we can obtain many examples of $\tau \in {\frak M}_{1,c}({\cal P})$ with $\sharp \G _{\tau }<\infty $ 
to which we can apply Theorems~\ref{t:msmtaust}, \ref{t:ABCD}, \ref{t:utauca}, \ref{t:kjemfhf}, 
\ref{t:kjemfsp}, \ref{t:kjemfsppt}, \ref{t:psinondiff}.    
\begin{ex}[Devil's coliseum (\cite{Splms10}) and complex analogue of the Takagi function] 
\label{ex:dc1}
Let $g_{1}(z):=z^{2}-1, g_{2}(z):=z^{2}/4, h_{1}:=g_{1}^{2},$ and  $h_{2}:=
g_{2}^{2}.$ Let $G=\langle h_{1},h_{2}\rangle $ and 
for each $a=(a_{1},a_{2})\in {\cal W}_{2}:=\{ (a_{1},a_{2})\in (0,1)^{2}\mid \sum _{j=1}^{2}a_{j}=1\} 
\cong (0,1)$, let 
$\tau _{a}:= \sum _{i=1}^{2}a_{i}\delta _{h_{i}}.$ 
Then by \cite[Example 6.2]{Splms10},  
setting $A:= K(h_{2})\setminus D(0,0.4)$, 
we have 
$\overline{D(0,0.4)}\subset \mbox{int}(K(h_{1}))$, $h_{2}(K(h_{1}))\subset \mbox{int}(K(h_{1}))$,  
$h_{1}^{-1}(A)\cup h_{2}^{-1}(A)\subset A$, and $h_{1}^{-1}(A)\cap h_{2}^{-1}(A)=\emptyset .$ 
Therefore $h_{1}^{-1}(J(G))\cap h_{2}^{-1}(J(G))=\emptyset $ and $\emptyset \neq K(h_{1})\subset \hat{K}(G).$ 
 Moreover, 
$G$ is hyperbolic and mean stable, and for each $a\in {\cal W}_{2}$, 
we obtain that $T_{\infty ,\tau _{a}}$ is continuous on $\CCI $ and the set of varying points of 
$T_{\infty ,\tau _{a}}$ is equal to $J(G).$ Moreover, by \cite{Splms10}
$\dim _{H}(J(G))<2$ and for each non-empty open subset $U$ of $J(G)$ there exists an uncountable dense 
subset $A_{U}$ of $U$ such that for each $z\in A_{U}$, 
$T_{\infty ,\tau _{a}}$ is not differentiable at $z.$ 
See Figures~\ref{fig:dcjulia} and \ref{fig:dcgraphgrey2}.  
The function $T_{\infty ,\tau _{a}}$ is called a devil's coliseum. 
It is a complex analogue of the devil's staircase. (Remark: 
as the author of this paper pointed out in \cite{Splms10}, 
the devil's staircase can be regarded as the function of probability 
of tending to $+\infty $ regarding the random dynamics on $\RR $ such that at every 
step we choose $h_{1}(x)=3x$ with probability $1/2$ and we choose $h_{2}(x)=3(x-1)+1$ with probability $1/2.$ For the detail, see \cite{Splms10}.)   
By Theorem~\ref{t:kjemfsppt}, 
for each $z\in \CCI $, $a_{1}\mapsto T_{\infty ,\tau _{a}}(z)$ is real-analytic in $(0,1)$, 
and for each $b\in {\cal W}_{2}$, $[\frac{\partial T_{\infty ,\tau _{a}}(z)}{\partial a_{1}}]|_{a=b}=\sum _{n=0}^{\infty }
M_{\tau _{b}}^{n}(\zeta _{1,b})$, where $\zeta _{1,b}(z):=T_{\infty ,\tau _{b}}(h_{1}(z))-T_{\infty ,\tau _{b}}(h_{2}(z)).$ 
Moreover, by Theorem~\ref{t:kjemfsppt}, the function 
$\psi (z):= [\frac{\partial T_{\infty ,\tau _{a}}(z)}{\partial a_{1}}]|_{a=b}$ defined on $\CCI $ 
is H\"{o}lder continuous on $\CCI $ and is locally constant on $F(G).$ 
As mentioned in Remark~\ref{r:takagi}, the function 
$\psi (z)$ defined on $\CCI $ 
can be regarded as a complex 
analogue of the Takagi function. 
By Theorem~\ref{t:psinondiff}, 
there exists an uncountable dense subset $A$ of $J(G)$ 
such that for each $z\in A$, 
either $\psi $ is not differentiable at $z$ or 
$\psi $ is not differentiable at each point $w\in h_{1}^{-1}(\{ z\})\cup h_{2}^{-1}(\{ z\} ).$ 
For the graph of $[\frac{\partial T_{\infty ,\tau _{a}}(z)}{\partial a_{1}}]|_{a_{1}=1/2}$, see 
Figure~\ref{fig:ctgraphgrey1}. 
\end{ex}

We now give an example of $\tau \in {\frak M}_{1,c}({\cal P})$ with $\sharp \G _{\tau }<\infty $ such that 
$J_{\ker }(G_{\tau })=\emptyset $, $J(G_{\tau })\neq \emptyset$, 
$S_{\tau }\subset F(G_{\tau })$ and $\tau $ is not mean stable.  
\begin{ex}
\label{ex:stfnms}
Let $h_{1}\in {\cal P}$ be such that 
$J(h_{1})$ is connected and 
$h_{1}$ has a Siegel disk $S$.  
Let $b\in S$ be a point. Let $d\in \NN $ be such that 
$(\deg (h_{1}),d)\neq (2,2).$  
Then by \cite[Proposition 6.1]{Splms10} (or \cite[Proposition 2.40]{SdpbpI}) and its proof,  
there exists a number $c>0$ such that for each 
$\lambda \in \CC $ with $0<|\lambda |<c$, 
setting $h_{2}(z):=\lambda (z-b)^{d}+b$ and 
$G:=\langle h_{1},h_{2}\rangle $, 
we have $J_{\ker }(G)=\emptyset $ and 
$h_{2}(K(h_{1}))\subset S\subset \mbox{int}(K(h_{1}))\subset \mbox{int}(K(h_{2}))$.  
Then the set of minimal sets for $(G,\CCI )$ is $\{ \{ \infty \} ,L_{0}\} $, 
where $L_{0}$ is a compact subset of $S\ (\subset F(G)).$   
Let $(p_{1},p_{2})\in {\cal W}_{2}$ be any element and let 
$\tau :=\sum _{j=1}^{2}p_{j}\delta _{h_{j}}.$ 
 Then $J_{\ker }(G_{\tau })=\emptyset $, $J(G_{\tau })\neq \emptyset$, 
$S_{\tau }\subset F(G_{\tau })$ and $\tau $ is not mean stable. In fact, 
$L_{0}$ is sub-rotative. 
Even though $\tau $ is not mean stable, 
we can apply Theorems~\ref{t:utauca}, \ref{t:kjemfhf}, \ref{t:kjemfsp}, \ref{t:kjemfsppt}, 
\ref{t:psinondiff} to 
this $\tau .$  
\end{ex}
\begin{ex}
\label{ex:Jt}
By \cite[Example 6.7]{Splms10}, 
we have an example $\tau \in {\frak M}_{1,c}({\cal P})$ such that 
$J_{\ker }(G_{\tau })=\emptyset $ and such that there exists a $J$-touching minimal set for $(G_{\tau },\CCI ).$ 
This $\tau $ is not mean stable but we can apply Theorem~\ref{t:utauca} to this $\tau .$ 
\end{ex}


\begin{thebibliography}{90}
\bibitem{AK} P. Allaart and K. Kawamura, 
{\em Extreme values of some continuous nowhere differentiable functions},
 Math. Proc. Cambridge Philos. Soc. 140 (2006), no. 2, 269--295. 
\bibitem{Be} A. Beardon, 
{\em Iteration of Rational Functions}, 
Graduate Texts in Mathematics 132, Springer-Verlag, 1991.
\bibitem{BP} A. Beardon and Ch. Pommerenke, {\em The Poincar\'e metric of plane domains}, 
J. London Math. Soc., (2), 18 (1978), 475--483.  
\bibitem{Br1} R. Br\"{u}ck,\ 
{\em Connectedness and stability of Julia sets of the composition of 
polynomials of the form $z^{2}+c_{n}$}, 
J. London Math. Soc. {\bf 61} (2000), 462-470.
\bibitem{Br2} R. Br\"{u}ck,\ 
{\em Geometric properties of Julia sets of the composition of 
polynomials of the form $z^{2}+c_{n}$},\ 
Pacific J. Math.,\ {\bf 198} (2001), no. 2,\ 347--372.
\bibitem{BBR} R. Br\"{u}ck,\ M. B\"{u}ger and  
S. Reitz,\ {\em Random iterations of polynomials of the 
form $z^{2}+c_{n}$: Connectedness of Julia sets},\ 
Ergodic Theory Dynam. Systems,\ 
 {\bf 19},\ (1999),\ No.5,\ 1221--1231. 

\bibitem{Bu1} M. B\"{u}ger, 
{\em Self-similarity of Julia sets of the composition of 
polynomials}, 
Ergodic Theory Dynam. Systems, {\bf 17} (1997), 1289--1297.
\bibitem{Bu2} M. B\"{u}ger, {\em 
On the composition of polynomials of the form} $z\sp 2+c\sb n$, 
 Math. Ann. 310 (1998), no. 4, 661--683.
\bibitem{D} R. Devaney, {\em An Introduction to Chaotic 
Dynamical Systems, Second edition},  
 Addison-Wesley Studies in Nonlinearity. Addison-Wesley Publishing Company, 
1989. 
\bibitem{FS} J. E. Fornaess and  N. Sibony,\ 
{\em Random iterations of rational functions},\ 
Ergodic Theory  Dynam. Systems, {\bf 11}(1991),\ 687--708.
\bibitem{GQL} Z. Gong, W. Qiu and Y. Li, 
{\em Connectedness of Julia sets for a quadratic random 
dynamical system}, Ergodic Theory Dynam. Systems, (2003), {\bf 23}, 
1807-1815.

\bibitem{GR} Z. Gong and F. Ren,\ {\em A random dynamical 
system formed by infinitely many functions, }
 Journal of Fudan University, {\bf 35}, 1996,\ 387--392.
\bibitem{HY} M. Hata and M. Yamaguti, 
{\em Takagi function and its generalization}, 
Japan J. Appl. Math., {\bf 1}, pp 183-199 (1984). 


\bibitem{HM} A. Hinkkanen and  G. J. Martin,
{\em The Dynamics of Semigroups of Rational Functions I},
 Proc. London Math. Soc. (3){\bf 73}(1996),\ 358-384.
\bibitem{J1} M. Jonsson,\ 
{\em Dynamics of polynomial skew products on $\CC ^{2}$}, 
Math. Ann. {\bf 314} (1999), 403-447. 
\bibitem{K} T. Kato, {\em Perturbation Theory for Linear Operators}, 
Springer, 1995. 
 \bibitem{MT} K. Matsumoto and I. Tsuda, {\em Noise-induced order}, 
 J. Statist. Phys. 31 (1983) 87-106. 
\bibitem{Mc} C. T. McMullen, {\em Complex Dynamics and Renormalization}, 
Annals of Mathematical Studies 135, Princeton University Press, 1994. 
\bibitem{Mi} J. Milnor, 
{\em Dynamics in One Complex Variable, Third Edition}, 
Annals of Mathematics Studies No. 160, Princeton University Press, 2006.  

 \bibitem{SeSh} T. Sekiguchi and Y. Shiota, \ 
{\em A generalization of Hata-Yamaguti's results on the Takagi function}, 
Japan J. Appl. Math. {\bf 8}, pp203-219, 1991.

\bibitem{Se} O. Sester, {\em Combinatorial configurations of fibered polynomials}, 
Ergodic Theory Dynam. Systems,  {\bf 21} (2001), 915-955. 
\bibitem{St3} R. Stankewitz,\ {\em Uniformly 
perfect sets,\ rational 
semigroups,\ Kleinian groups and IFS's ,\ }Proc. Amer. Math. Soc.
{\bf 128},\ (2000),\ No. 9,\ 2569--2575.

\bibitem{SS} R. Stankewitz and H. Sumi, 
{\em Dynamical properties and structure of Julia sets of postcritically 
bounded polynomial semigroups}, 
Trans. Amer. Math. Soc.,  363 (2011), no. 10, 5293--5319.
\bibitem{S3} H. Sumi,\ {\em Skew product maps related to finitely 
generated rational semigroups,\ } 
Nonlinearity,\ {\bf 13},\ (2000), 
995--1019.

\bibitem{S4} H. Sumi,\ {\em Dynamics of sub-hyperbolic and 
semi-hyperbolic rational semigroups and skew products},\  
Ergodic Theory Dynam. Systems, (2001),\ {\bf 21},\ 563--603.

\bibitem{S7} H. Sumi,\ 
{\em Semi-hyperbolic fibered rational maps and rational 
semigroups},\ 
Ergodic Theory Dynam. Systems, (2006), 
{\bf 26}, 893--922. 
\bibitem{S15} H. Sumi, {\em Interaction cohomology of forward or backward self-similar 
systems}, Adv. Math.,  222 (2009), no. 3, 729--781.

\bibitem{SdpbpIII} H. Sumi, 
{\em Dynamics of postcritically bounded polynomial semigroups III: 
classification of semi-hyperbolic semigroups and random Julia sets 
which are Jordan curves but not quasicircles}, 
Ergodic Theory Dynam. Systems,  
(2010), {\bf 30}, No. 6, 1869--1902. 
\bibitem{Ssugexp} H. Sumi, 
{\em Rational semigroups, random complex dynamics and singular functions on the complex plane}, 
survey article, Selected Papers on Analysis and Differential Equations, Amer. Math. Soc. Transl. (2) Vol. 230, 2010, 161--200. 

\bibitem{SdpbpI} H. Sumi, 
{\em Dynamics of postcritically bounded polynomial semigroups I: connected components 
of the Julia sets}, 
 Discrete Contin. Dyn. Sys. Ser. A, 
 Vol. 29, No. 3, 2011, 1205--1244. 

\bibitem{Splms10} H. Sumi, 
{\em Random complex dynamics and semigroups of holomorphic maps}, 
Proc. London Math. Soc.,  
(2011), 102 (1), 50--112. 
\bibitem{Srcddc} H. Sumi, 
{\em Random complex dynamics and devil's coliseums}, 
preprint 2011, 
http://arxiv.org/abs/1104.3640. 
\bibitem{SU1} H. Sumi and M. Urba\'{n}ski,\ 
{\em Real analyticity of Hausdorff dimension for 
expanding rational semigroups}, 
Ergodic Theory Dynam. Systems (2010), Vol. 30, No. 2, 601-633.   
\bibitem{SU2} H. Sumi and M. Urba\'{n}ski,\ 
{\em Measures and dimensions of Julia sets of semi-hyperbolic 
rational semigroups}, 
Discrete and Continuous Dynamical Systems Ser. A., Vol 30, No. 1, 2011, 313--363. 
\bibitem{YHK} M. Yamaguti, M. Hata, and J. Kigami,
Mathematics of fractals. Translated from the 1993 Japanese original by Kiki Hudson. Translations of Mathematical Monographs, 167. American Mathematical Society, Providence, RI, 1997.
\end{thebibliography}
\end{document}